%% file: stab.tex
\documentclass[11pt, reqno, a4paper,oneside]{amsart}
\usepackage{graphicx, epsfig, psfrag}
\usepackage{amsfonts,amsmath,amssymb,amsbsy,amsthm}
\usepackage{bm}
\usepackage{color}
\usepackage{float}
\usepackage{hyperref}
\usepackage{mathrsfs}
\usepackage{longtable}

\usepackage[left=2cm,right=2cm,top=2cm,bottom=2cm]{geometry}

\usepackage{environments}
\numberwithin{theorem}{section}
\numberwithin{equation}{section}

\renewcommand{\paragraph}[1]{\subsubsection{#1}}

\input{notation}

\definecolor{cocol}{rgb}{0.7, 0, 0}
\definecolor{ascol}{rgb}{0, 0, 0.9}
\definecolor{lzcol}{rgb}{0, 0.7, 0}
\definecolor{todocol}{rgb}{0.0, 0.4, 0.0}

\usepackage{ulem}

\definecolor{delcol}{rgb}{0.7, 0.0, 0.7}
\newcommand{\delfinal}[1]{{#1}}

\begin{document}


\title[(In-)Stability and Stabilisation of QNL-Type A/C
Methods]{(In-)Stability and Stabilisation of QNL-Type
  Atomistic-to-Continuum Coupling Methods}

\author{C. Ortner}
\address{C. Ortner\\ Mathematics Institute \\ Zeeman Building \\
  University of Warwick \\ Coventry CV4 7AL \\ UK}
\email{christoph.ortner@warwick.ac.uk}

\author{A. V. Shapeev}
\address{A. V. Shapeev \\ School of Mathematics \\ University of
  Minnesota \\ 206 Church St SE \\  Minneapolis \\ MN 55455 \\ USA}
\email{ashapeev@math.umn.edu}

\author{L. Zhang} \address{L. Zhang \\ Department of Mathematics,
  Institute of Natural Sciences, and Ministry of Education Key
  Laboratory of Scientific and Engineering Computing (MOE-LSC) \\
  Shanghai Jiao Tong University \\ 800 Dongchuan Road \\ Shanghai
  200240 \\ China}
\email{lzhang2012@sjtu.edu.cn}

\date{\today}

\thanks{CO and LZ were supported by the EPSRC grant EP/H003096
  ``Analysis of atomistic-to-continuum coupling methods''. AS was
  supported by AFOSR Award FA9550-12-1-0187.}



\begin{abstract}
  We study the stability of ghost force-free energy-based
  atomistic-to-continuum coupling methods. In 1D we essentially
  complete the theory by introducing a universally stable a/c coupling
  as well as a stabilisation mechanism for unstable coupling schemes.

  We then present a comprehensive study of a two-dimensional scalar
  planar interface setting, as a step towards a general 2D/3D
  vectorial analysis. Our results point out various new challenges.
  For example, we find that none of the ghost force-free methods known
  to us are universally stable (i.e., stable under general interaction
  and general loads). We then explore to what extent our 1D
  stabilisation mechanism can be extended.
\end{abstract}

\maketitle


\section{Introduction}
Atomistic-to-continuum (a/c) coupling schemes are a class of
computational multiscale methods for the efficient simulation of
crystalline solids in the presence of defects. Different variants have
been among the tools of computional materials science for many decades
\cite{Mullins1982, Fischmeister1989, Ortiz:1995a}. More recently, a
numerical analysis theory of a/c coupling has emerged; we refer to
\cite{acta} for an introduction, a summary of the state of the art,
and extensive references.

While the {\em consistency} theory of a/c coupling methods has a solid
foundation \cite{Or:2011a,OrtnerShapeev2013}, understanding their
stability properties essentially remains an outstanding open
problem. The main difficulty is that the a/c model interface, even if
treated consistently, can generate new eigenmodes present in neither
the atomistic nor continuum model, which can render a/c coupling
methods unstable. Indeed, we emphasize that we are not only concerned
with questions of analysis, but also with the construction of stable
schemes.

In one dimension, an essentially complete survey of stability is
presented in the review article \cite{acta}, which is partially based
on the results of the present article. In dimension greater than one
very little is known.  Some recent progress for 2D force-based
blending methods \cite{2012-MMS-bqcf.stab, 2013-PRE-bqcfcomp} remains
incomplete and partially based on numerical evidence. For a sharp
interface force-based coupling scheme more comprehensive analytical
results are presented in \cite{LuMin201212}, but even these are
restricted to flat a/c interfaces and are dependent upon conditions
that cannot be readily checked analytically.

In the present work we focus on the stability of a particular class of
conservative a/c schemes, generally called quasinonlocal (QNL) type
coupling schemes. In one dimension we present examples of stability
and instability (\S \ref{sec:qnl1d}), construct a new ``universally
stable'' scheme (\S \ref{sec:refl}), and further show how unstable QNL
schemes can be stabilised (\S \ref{sec:stab1d}).

We then consider a two-dimensional model problem, for which our
results are more limited, in that we need to make much more stringent
assumptions on the deformation and interaction potential than in
1D. Within these assumptions, we show that there is a source of
instability in 2D interfaces, which was not present in the 1D setting
(\S\ref{sec:2d:qnlhess_rep}).  Moreover, we show that this instability
is universal. It is not only present in specific instances of QNL type
a/c couplings (Proposition \ref{th:2d:Hqnl_simple}), but in a fairly
wide class of generalized {\it geometric reconstruction} methods
\cite{PRE-ac.2dcorners} (\S \ref{sec:univ_stab_2d}) which cover most
of the existing methods. This new source of instability is more severe
than the instabilities observed in 1D and cannot be ``easily''
stabilised. To be precise, we show that stabilising QNL type schemes
in 2D severely affects their consistency when the system approaches a
bifurcation point (\S\ref{sec:stab2d}).

\clearpage

\section{A General 1D QNL Formulation}

\subsection{Notation for lattice functions}
\label{sec:intro:notation}
For a lattice function $v : \Z \to \R$ and $\rho \in \Z \setminus
\{0\}$, we define the finite difference operators
\begin{displaymath}
  D_\rho v(\xi) := v(\xi+\rho) - v(\xi).
\end{displaymath}
For some finite interaction stencil $\Rg = \{\pm 1, \dots, \pm \rcut
\}$, where $\rcut \in \N$ is a fixed cut-off radius, we define
\begin{displaymath}
  Dv(\xi) := \b( D_\rho v(\xi) \b)_{\rho \in \Rg}.
\end{displaymath}
The space of compact displacements is defined by
\begin{displaymath}
  \Usz := \b\{ u : \Z \to \R \bsep {\rm supp}(u) \text{ is bounded} \b\}.
\end{displaymath}

Each lattice function $v : \Z \to \R$ is identified with its canonical
continuous piecewise affine interpolant. In particular, we define the
gradients $\D v(x) := v(\xi) - v(\xi-1)$ for $x \in (\xi-1, \xi)$. 

If $H : \Usz \to \Usz^*$ is a linear operator (or, $\< H \cdot, \cdot
\>$ a bilinear form on $\Usz$), then we define the associated
stability constant
\begin{displaymath}
  \gamma(H) := \inf_{\substack{ u \in \Usz \\ \| \D u \|_{L^2} = 1 }}
  \< H u, u \>.
\end{displaymath}
We say that $H$ is stable if $\gamma(H) > 0$.

\subsection{Many-body interactions for an infinite chain}
We consider finite-range many-body interactions of deformed
configurations of the infinite chain $\Z$. Let $V \in C^2(\R^{\Rg})$
be the many-body site energy potential with partial derivatives
\begin{displaymath}
  V_\rho(\bfg) := \frac{\partial V(\bfg)}{\partial g_\rho} \quad
  \text{and} \quad
  V_{\rho\vsig}(\bfg) := \frac{\partial^2 V(\bfg)}{\partial
    g_\rho \partial g_\vsig} \quad \text{for } \bfg = (g_\rho)_{\rho\in\Rg} \in \R^\Rg.
\end{displaymath}
We assume that $V$ is invariant under reflections of the local
configuration, that is,
\begin{equation}
  \label{eq:qnl1d:symmV}
  V\b( (g_\rho)_{\rho\in\Rg} \b) = V\b( (-g_{-\rho})_{\rho\in\Rg} \b).
\end{equation}
Immediate consequences of \eqref{eq:qnl1d:symmV} are the symmetries
\begin{equation}
  \label{eq:qnl1d:symm_dV_hV}
  V_{-\rho}(\mF\Rg) = -V_\rho(\mF\Rg) \quad \text{and} \quad
  V_{-\rho,-\vsig}(\mF\Rg) = V_{\rho\vsig}(\mF\Rg) \qquad \forall
  \rho,\vsig\in\Rg, \mF > 0.
\end{equation}

A macroscopic strain $\mF$ and a displacement $u \in \Usz$ induce a
deformed configuration $y(\xi) = \mF \xi + u(\xi)$, $\xi \in \Z$. To
such a configuration we assign the energy difference
\begin{equation}
  \label{eq:defn_Ea}
  \Ea(y) := \sum_{\xi \in \Z} \b[ V(Dy(\xi)) -
  V(\mF\Rg) \b].
\end{equation}
Since the lattice sum is finite, this expression is well-defined. The
first and second variations with respect to $u$ (in the sense of
Gateaux derivatives) are also well-defined and are given by
\begin{align*}
  \<\del\Ea(y), v \> &:= \sum_{\xi\in \Z} \sum_{\rho\in\Rg}
  V_\rho(Dy(\xi)) \cdot D_\rho
  v(\xi), \quad \text{and} \\
  \< \ddel\Ea(y) v, v \> &:= \sum_{\xi\in\Z} \sum_{\rho,\vsig\in\Rg}
  V_{\rho\vsig}(Dy(\xi)) \cdot D_\rho v(\xi) D_\vsig v(\xi), \qquad
  \text{for } v \in \Usz.
\end{align*}
We are particularly interested in the second variation evaluated at
the homogeneous deformation $y = \mF x$ (where $(\mF x)(\xi) := \mF
\xi$),
\begin{equation}
  \label{eq:defn_Ha}
  \< \Ha_\mF v, v \> := \< \ddel\Ea(\mF x) v, v \> = \sum_{\xi\in\Z}
  \sum_{\rho,\vsig \in \Rg} V_{\rho\vsig} \cdot D_\rho v(\xi) D_\vsig v(\xi),
\end{equation}
where, here and throughout most of this work, we are suppressing the
dependence of $V_{\rho\vsig}$ on $\mF\Rg$ when it is clear from the
context that we mean $V_{\rho\vsig}(\mF\Rg)$.

The stability of non-homogeneous states $y = \mF + u$ can be deduced
from the stability of homogeneous states; see \cite[Theorem
7.8]{acta}.



\subsection{A general QNL formulation}
The QNL approximation of $\Ea$ \cite{Shimokawa:2004} transitions
between the atomistic model and a continuum model by introducing
modified site potential at the a/c interface. To simplify our analysis
we focus on a single a/c interface. Let $\Z_- := \{-\infty, \dots,
0\}$ be the atomistic region and $\R_+ := (0, \infty)$ the continuum
region. In the atomistic region, we employ modified site energies
$\tilde{V}_\xi \in C^2(\R^\Rg)$, $\xi \leq 0$, while, in the continuum
region we employ the Cauchy--Born strain energy density
\cite{BLBL:arma2002,Hudson:stab,E:2007a,OrtnerTheil2012},
\begin{displaymath}
  W(\mG) := V(\mG \Rg).
\end{displaymath}
The QNL a/c coupling energy functional is then given by
\begin{equation}
  \label{eq:general_qnl_energy}
  \Eac(y) := \sum_{\xi = -\infty}^{0} \b[ \tilde{V}_\xi(Dy(\xi))  -
  V(\mF \Rg) \b] 
  + \int_{1/2}^\infty \b[W(\D y)-W(\mF)\b] \dx,
\end{equation}
for all deformations $y = \mF x + u, u \in \Usz$. 

We shall assume throughout that the modified site energies satisfy the
following conditions: there exists $\xi_0 \in \Z_-$ such that
\begin{align}
  \label{eq:qnl:noninteraction}
  &\tilde{V}_{\xi,\rho}(Dy(\xi)) = 0 \quad \text{whenever } \xi+\rho >
  1, \\
  \label{eq:qnl:finite_interface}
  & \tilde{V}_{\xi}(Dy) = V(Dy) \quad \text{ for } \xi \leq \xi_0, \\
  \label{eq:qnl:force_consistency}
  & \del\Eac(\mF x) = 0 \qquad \forall \mF > 0.
\end{align}
Condition \eqref{eq:qnl:noninteraction} states that atoms do not
interact with the continuum region, except for the interface atom at
$\xi = 0$. Condition \eqref{eq:qnl:finite_interface} states that the
transition region is bounded. Condition
\eqref{eq:qnl:force_consistency} is the force-consistency condition
(absence of ``ghost forces''), which ensures first-order consistency
of the QNL approximation \cite{Or:2011a}.

As in the case of the atomistic model, $\Eac$ is well-defined and has
variations in the sense of Gateaux derivatives. The second variation,
evaluated at the homogeneous deformation $y = \mF x$, $\Hac_\mF =
\ddel\Eac(\mF x)$, is given by
\begin{equation} 
  \label{eq:qnl:hessian}
  \< \Hac_\mF u, u \> = \sum_{\xi = -\infty}^0 \sum_{\rho,\vsig\in\Rg}
  \tilde{V}_{\xi,\rho\vsig} \cdot D_\rho u(\xi) D_\vsig u(\xi) +
  \ddW(\mF) \int_{1/2}^\infty |\D u|^2 \dx.
\end{equation}

\subsubsection{Error in critical strains}
\label{sec:err_crit_strains}
We shall be interested in understanding the regimes of strains $\mF$
for which $\Ha_\mF$ and $\Hac_\mF$ are stable. To explain why this is
relevant in practical simulations, consider the following description
of a quasi-static loading scenario (adapted from
\cite{Dobson:qce.stab}):
\begin{itemize}
\item[(i)] $\mF(t) \in C([t_0, t_*])$ is a given path in deformation
  space, where $t_*$ is a critical load, and constants $c_0, c_1 > 0$,
  such that
  \begin{displaymath}
    c_0 (t_* - t) \leq \gamma(\Ha_{\mF(t)}) \leq c_1 (t_*-t) \quad
    \text{for } t_0 \leq t \leq t_*.
  \end{displaymath}
  At the critical load $t_*$ (e.g., a bifurcation) an instability
  occurs, which typically indicates the onset of defect nucleation of
  defect motion (``critical event'').

\item[(ii)] Suppose now that QNL is initially stable but has a reduced
  stability region: $\gamma(\Hqnl_{\mF(t_0)}) > 0$ but
  $\gamma(\Hqnl_{\mF(t_*)}) < 0$. Then, there exists a reduced critical
  load $t_*^\qnl < t_*$ such that $\gamma(\Hqnl_{\mF(t_*^\qnl)}) =
  0$. 

  In such a situation we first of all predict an incorrect critical
  load, i.e., incorrect magnitude applied forces under which the
  critical event occurs. Moreover, since the event may occur in a
  different region of deformation space, it is even possible that a
  qualitatively different event is observed (e.g., a different type of
  defect is nucleated).
\end{itemize}

\subsubsection{Preliminary estimates}
We can immediately make the following generic observation.

\begin{proposition}
  \label{th:cac_leq_ca}
  $\gamma(\Hac_\mF) \leq \gamma(\Ha_\mF)$ for all $\mF > 0$.
\end{proposition}

\begin{remark}
  If the atomistic region is finite, then we would obtain that
  $\cac(\mG) \leq \ca(\mF) + {\rm err}$, where ${\rm err}$ decreases
  with increasing atomistic region size. See \cite{Hudson:stab} for
  results along these lines.
\end{remark}
\begin{proof}
  Let $\eps > 0$ and let $u \in \Usz$ such that $\| \D u \|_{L^2} = 1$
  and $\< \Ha u, u \> \leq \ca(\mF) + \eps$. Upon shifting $u$ by
  $v(\xi) = u(\xi + \eta)$ for $\eta$ sufficiently large, we can
  assume without loss of generality that $u(\xi) = 0$ for all $\xi
  \geq \xi_0 - \rcut-1$. Therefore, we obtain
  \begin{displaymath}
    \gamma(\Hac_\mF) \leq \< \Hac_\mF u, u \> = \< \Ha u, u \> \leq \gamma(\Ha_\mF)
    + \eps.
  \end{displaymath}
  Since $\eps$ was arbitary, the result follows.
\end{proof}

Proposition \ref{th:cac_leq_ca} ensures that, if a lattice $\mF \Z$ is
unstable in the atomistic model, then it must also be unstable in the
a/c coupling model. The converse question, whether stability of
$\Ha_\mF$ implies stability of $\Hac_\mF$ is more difficult to answer
in general. This question was first raised in \cite{Dobson:qce.stab}
for 1D second-neighbour Lennard-Jones type pair interactions, where
this implication holds. Further investigations in this direction can
be found in \cite{Dobson:qcf.stab,XHLi:3n,LiLuskin2012}. In the present
work we aim to present a more complete picture for the case of general
range many-body interactions.

To conclude this section, we present another elementary auxiliary
result that we will reference later on. Let the Cauchy--Born energy
functional be given by $\Ec(y) := \int_\R [W(\D y) - W(\mF)]\dx$, and
the corresponding hessian operator by
\begin{displaymath}
  \< \Hc_\mF u, u \> := \ddW(\mF)
  \| \D u \|_{L^2}^2.
\end{displaymath}
Then, we have the following result.

\begin{lemma}
  \label{th:cc_geq_ca}
  $\gamma(\Hc_\mF) = \ddW(\mF) \geq \gamma(\Ha_\mF)$ for all $\mF >
  0$.
\end{lemma}
\begin{proof}
  The idea of this result is classical; see for example
  \cite{Wallace98}.  A proof, which can be translated verbatim to our
  present setting, is given in \cite{Hudson:stab}.
\end{proof}


\section{(In-)stability of a Second-Neighbour QNL Method}
\label{sec:qnl1d}

\subsection{The second-neighbour QNL method}
\label{sec:qnl1d:defn_qnl}
%
%
The original QNL energy, in the case of second-neighbours ($\Rg = \{
\pm 1, \pm 2\}$), is given by \cite{Shimokawa:2004,Dobson:2008b}
\begin{equation}
  \label{eq:defn_qnl2}
  \begin{split}
    \Eqnl(y) &= \sum_{\xi = -\infty}^{-2} \b[ V(Dy(\xi)) - V(\mF\Rg) \b]
    + \sum_{\xi = -1}^{0} \b[ V(\tilD y(\xi)) - V(\mF \Rg) \b] \\
    & \qquad  + \int_{1/2}^\infty \b[ W(\D y) - W(\mF) \b] \dx,
  \end{split}
\end{equation}
where
\begin{displaymath}
  \tilD := ( D_{-2}, D_{-1}, D_1, 2 D_1 ).
\end{displaymath}
(That is, interaction of interface atoms with the atomistic region use
the atomistic finite difference, $D_{-j}$, while interaction of
interface atoms with the continuum region use only nearest-neighbour
finite difference, $j D_1$.)

It is well-known that this energy functional is force-consistent
\cite{Shimokawa:2004,E:2006},
\begin{displaymath}
  \< \del\Eqnl(\mF x), v \> = 0 \qquad \forall v \in \Usz,
\end{displaymath}
which implies a general first-order consistency result
\cite{acta,Or:2011a}.

Moreover, for the case of Lennard-Jones type interactions under
expansion, and periodic boundary conditions, it has been shown in
\cite{Dobson:qce.stab} that $\gamma(H^\qnl_\mF) > 0$ if and only if
$\gamma(\Ha_\mF) > 0$, up to a small error. This can be generalised
and translated to our setting as follows.

\begin{proposition}
  \label{th:qnl2_stab_pair}
  Suppose that $\Rg = \{\pm 1, \pm 2\}$ and $V(Dy) = \sum_{j \in \Rg}
  \phi(|D_j y|)$, where $\phi \in C^2(\R_+)$. Then, for $\mF > 0$,
  $\gamma(\Hqnl_\mF) > 0$ if and only if $\gamma(\Ha_\mF) > 0$.
\end{proposition}
\begin{proof}
  We only give a brief outline of the proof, as the essential ideas
  are already contained in \cite{Dobson:qce.stab}.

  We already know that $\gamma(\Hqnl_\mF) \leq \gamma(\Ha_\mF)$, hence
  we only prove the opposite inequality.
  
  A short calculation (see \cite{Dobson:qce.stab, acta} for more
  details), employing the identity
  \begin{displaymath}
    \phi''(2\mF) |D_2 u(\xi)|^2 = 2 \phi''(2\mF) \b\{ |D_1 u(\xi)|^2 +
    |D_1 u(\xi+1)|^2 \b\} - \phi''(2\mF) |D_1^2 u(\xi)|^2,
  \end{displaymath}
  yields
  \begin{equation}
    \label{eq:qnl2_stab_pair:critid}
    \< \Hqnl u, u \> = \< \Hc u, u \> - \phi''(2\mF) \sum_{\xi =
      -\infty}^{-2} |D_1^2 u(\xi)|^2.
  \end{equation}
  Hence, if $\phi''(2\mF) \leq 0$ (Lennard-Jones case), then
  $\gamma(\Hqnl_\mF) \geq \gamma(\Hc_\mF) \geq \gamma(\Ha_\mF)$.

  If $\phi''(2\mF) > 0$, then employing the identity $\< \Ha u, u \> =
  \< \Hc u, u \> - \phi''(2\mF) \| D_1^2 u \|_{\ell^2}^2$ (which
  follows from the same calculation as
  \eqref{eq:qnl2_stab_pair:critid}), we obtain
  \begin{displaymath}
    \< \Hqnl u, u \> = \< \Ha u, u \> + \phi''(2\mF) \sum_{\xi =
      - 1}^{\infty} |D_1^2 u(\xi)|^2.
  \end{displaymath}
  Hence, $\gamma(\Hqnl_\mF) \geq \gamma(\Ha_\mF)$.
\end{proof}

\subsection{Instability Example}
\label{sec:qnl1d:instab}
Proposition \ref{th:qnl2_stab_pair} leads us to investigate, whether
the result holds also for general many-body interactions. An analysis
of Li and Luskin \cite{LiLuskin2012} in a similar context, but ignoring
the transition from the atomistic to continuum model, indicates that
this may be false. Indeed, we can construct a counterexample. Here, we
present only a brief summary, but give more detail in \S \ref{sec:app_1dqnl2_example}. Our example is somewhat academic in that we do not show that
any concrete interaction potential exhibits this instability, but only
that it may occur {\em in principle}. 

For ease of notation, we write $V_{\rho,\vsig} = V_{\rho,\vsig}(\mF
\Rg)$. Exploiting the point symmetry of $V$, possibly rescaling by a
scalar, we assume that
\begin{align*}
  V_{1,1} = V_{-1,-1} &= 1, &  V_{1,-1} &=: \alpha,  \\
  V_{2,2} = V_{-2,-2} &=: \beta, & V_{2,-2} &=: \gamma, \\
   && \hspace{-6cm} V_{1,2} = V_{-1,-2} = -V_{-1,2} = -V_{1,-2} &=: \delta,
\end{align*}
for parameters $\alpha, \beta, \gamma, \delta \in \R$. The additional
symmetry $V_{1,2} = - V_{-1,2}$ that we employed is consistent with
EAM type potentials.

With these parameters, and a lengthy computation following
\cite{Dobson:qce.stab,LiLuskin2012}, we obtain
\begin{align}
  \label{eq:qnl1d:instab}
  \< \Hqnl_\mF u, u \> &=
  A \sum_{\xi \in \Z} |D_1 u(\xi)|^2 
  + \sum_{\xi = -\infty}^{-0} B_\xi |D_1^2 u(\xi)|^2  \\
  \notag
  & \qquad 
  + \sum_{\xi = -\infty}^{-1} C_\xi |D_1^3 u(\xi)|^2
  + D \sum_{\xi = -\infty}^{-2} |D_1^4 u(\xi)|^2,
\end{align}
where $A, B_\xi, C_\xi, D \in \R$ are coefficients that depend
linearly on the parameters $\alpha, \beta, \gamma, \delta$. 

Choosing the parameter values $\alpha = -0.99, \beta = 0.1, \gamma =
0.15, \delta = -0.2$ yields
\begin{align*}
  &A = 0.38; \\
  & B_0 = 0.91, \quad  B_{-1} = 3.26, \quad B_{-2} = 3.56, \quad
  B_\xi = 3.91 ~~\text{ for $\xi \leq -3$}; \\
  &C_{-1} = -0.5, \quad C_{-2} = -1.3, \quad C_\xi = -1.6 ~~\text{ for
    $\xi \leq -3$}; \\
  & D = 0.15.
\end{align*}
By a numerical calculation, we obtain that
\begin{displaymath}
  \gamma(\Hqnl_\mF) < -0.005.
\end{displaymath}

Conversely, using straightforward Fourier analysis, we can show that
\begin{displaymath}
  \gamma(\Ha_\mF) = 0.02.
\end{displaymath}
That is, $\Ha_\mF$ is stable, while $\Hqnl_\mF$ is unstable with this
choice of parameters. The details of the calculation are given in
\S~\ref{sec:app_1dqnl2_example}.

\subsection{A numerical example}
\label{sec:instab_qnl_num}
The counterexample from \S~\ref{sec:qnl1d:instab} is somewhat
dissatisfying in that it is based purely on experimenting with
coefficients, but there is no clear connection to a physical problem
of interest where the predicted discrepancy in stability occurs. We
therefore present a numerical example of a 1D chain with EAM type
interaction and applied external forces, for which we can still
observe this stability gap. We give a brief outline of the experiment
setup; the details of the model and of the setup are given in
\S~\ref{sec:app_1dqnl2_num}.

We reformulate the QNL model in a finite domain $\{-N, \dots, N\}$
with atomistic region $\{-K, \dots, K\}$. This is implemented by
applying the boundary condition $y(\xi) = \mF \xi$ for $|\xi| \geq
N$. Moreover, we apply an external force, to be able to observe
nonlinear deformation effects. Finally, we discretise the continuum
region using P1 finite elements. Given $N$, the atomistic region size
$K$ and the FE mesh are chosen quasi-optimally. 

We define the critical strain, $\mF^\qnl$, to be the smallest strain
greater than one, for which the corresponding equilibrium $y_\mF$ of
the energy is unstable, i.e., $\gamma(\ddel\Eqnl(y_\mF)) \leq 0$.

The exact critical strain $\mF^*$, against which the error is
measured, is defined to be the critical strain for the unrestricted
atomistic model.


\begin{figure}
  \begin{center}
    \includegraphics[width=14cm]{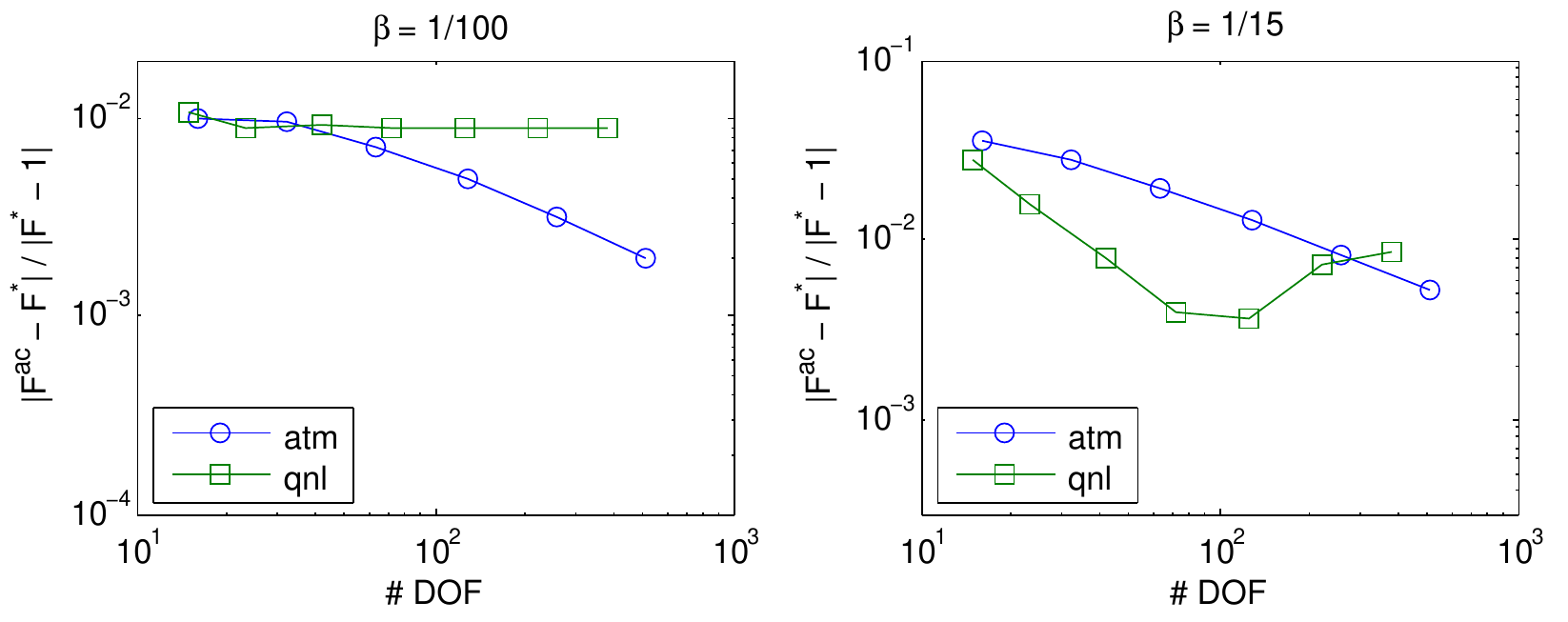}
    \caption{\label{fig:1dqnl_instab} Relative errors of critical
      strains for QNL and the restricted atomistic simulation. The
      external forces are parameterised by $\alpha = 1.5, \beta \in \{
      0.01, 0.066 \}$. See \S~\ref{sec:app_1dqnl2_num} for details of
      the model and the computation.}
  \end{center}
\end{figure}

In Figure \ref{fig:1dqnl_instab} we plot the relative errors in the
critical strains, for increasing domain sizes and hence increasing
computational cost (measured in terms of the number of degrees of
freedom required for the computation) for the QNL method and for the
restricted atomistic model. We observe that the critical strains in
the restricted atomistic model display clear systematic convergence,
whereas the critical strains of the QNL method appear to diverge or
converge to a wrong limit.

\section{A Universally Stable A/C Coupling in 1D}
\label{sec:refl}
Motivated by the results of \S~\ref{sec:qnl1d} we seek a/c couplings
with universally reliable stability properties. 

\begin{definition}
  An a/c coupling energy $\Eac$ is {\em universally stable} if, for
  all interaction potentials $V \in C^2(\R^\Rg)$ and strains $\mF >
  0$, $\gamma(\Hac_\mF) > 0$ if and only if $\gamma(\Ha_\mF) > 0$.
\end{definition}

\medskip

The analysis in \cite{acta} indicates that the behaviour we
observed in \S~\ref{sec:instab_qnl_num} is not possible if the QNL
method were universally stable, and indeed we saw in
\S~\ref{sec:qnl1d:instab} that counterexamples can be constructed.

We will now present the construction of a universally stable a/c
coupling. For simplicity, we consider again the case where the
atomistic region is given by $\Z_- = \{0, -1, -2, \dots \}$.  The {\it
  reflection method}, which we formulate in the following paragraphs
can be understood as a special case of the QNL and geometric
reconstruction ideas \cite{E:2006,PRE-ac.2dcorners,Shimokawa:2004},
but with a particularly simple reconstruction operator.

For any lattice function $z : \Z \to \R$ (both deformations and
displacements) we denote its anti-symmetric reflection about the
origin by
\begin{align*}
  z^* &:= \cases{ z(\xi), & \xi \leq 0, \\
    2 z(0) - z(-\xi), & \xi > 0.} 
\end{align*}
With this notation we define, for $y = \mF x + u, u \in \Usz$,
\begin{align*}
  \Erefl(y) &:= \E^*(y) + \int_0^\infty W(\D y) \dx, \qquad \text{where} \\
  \E^*(y) &:= \sum_{\xi = -\infty}^{-1} \b[ V(Dy^*(\xi))-V(\mF\Rg)
  \b] + \smfrac12 \b[V(Dy^*(0)) - V(\mF\Rg)\b].
\end{align*}
One may readily check that $\Erefl$ is of the general form
\eqref{eq:general_qnl_energy}.

The key property, the reason for the name ``reflection method'', and
in fact the motivation for the definition of $\Erefl$, is the
following.

\begin{lemma}
  \label{th:refl:lemma}
  Let $y = \mF x + u, u \in \Usz$, then $\E^*(y) = \smfrac12
  \Ea(y^*)$.
\end{lemma}
\begin{proof}
  By definition, $y^*$ is anti-symmetric about the origin, and 
  consequently,
  \begin{align*}
    D_\rho y^*(\xi) &= y^*(\xi+\rho) - y^*(\xi) \\
    &= \b[ 2y^*(0) - y^*(-\xi-\rho) \b] - \b[ 2 y^*(0) - y^*(-\xi) \b]
    \\
    &= - D_{-\rho} y^*(-\xi).
  \end{align*}
  Due to the reflection symmetry \eqref{eq:qnl1d:symmV} of $V$, we
  obtain $V(Dy^*(\xi)) = V(Dy^*(-\xi))$, which implies the stated
  result.
\end{proof}

\begin{theorem}
  The a/c coupling $\Erefl$ is force-consistent,
  \begin{equation}
    \label{eq:refl:patch_test}
    \b\< \del\Erefl(\mF x), v \b\> = 0,
  \end{equation}
  and universally stable,
  \begin{equation}
    \label{eq:refl:univstab}
    \gamma(\Hrefl_\mF) = \gamma(\Ha_\mF).
  \end{equation}
\end{theorem}
\begin{proof}
  {\it Proof of \eqref{eq:refl:patch_test}: } From Lemma
  \ref{th:refl:lemma} we obtain
  \begin{displaymath}
    \< \del\E^*(\mF x), v \> = \smfrac12 \< \del\Ea(\mF x), v^* \>, 
  \end{displaymath}
  where we note that $v^*$ does not necessarily belong to $\Usz$, but
  $\D v^*$ has compact support and hence the right-hand side is
  well-defined. Lemma 12 in \cite{Or:2011a} implies that
  \begin{equation}
    \label{eq:refl:delEa_dW}
    \< \del\Ea(\mF x), v^* \> = \dW(\mF) \int_\R \D v^*(x) \dx.
  \end{equation}
  (Note that \cite[Lemma 12]{Or:2011a} is in fact a 2D result,
  however, the 1D variant is proven verbatim using the 1D bond density
  formula \cite[Proposition 3.3]{Shapeev:2010a}. Alternatively,
  \eqref{eq:refl:delEa_dW} can be proven directly from
  \cite[Proposition 3.1]{Shapeev:2010a}.)

  Since $\D v^*$ is symmetric about the origin,
  \eqref{eq:refl:delEa_dW} implies that
  \begin{displaymath}
    \< \del\E^*(\mF), v \> = \smfrac12 \dW(\mF) \int_\R \D v^*(x) \dx =
    \dW(\mF) \int_{-\infty}^0 \D v(x) \dx = \dW(\mF) v(0).
  \end{displaymath}
  Inserting this into the definition of $\del\Erefl$, we obtain
  \begin{displaymath}
    \< \del\Erefl(\mF), v \> = \dW(\mF) v(0) + \int_0^\infty \dW(\mF) \D v(x)
    \dx = 0. 
  \end{displaymath}

  {\it Proof of \eqref{eq:refl:univstab}: } Applying again Lemma
  \ref{th:refl:lemma}, as well as the symmetry of $\D v^*$, we obtain
  \begin{align*}
    \< \Hrefl_\mF v, v \> &= \< \ddel\E^*(\mF) v, v \> + \ddW(\mF)
    \int_0^\infty |\D v|^2 \dx \\
    &= \smfrac12 \< \ddel\Ea(\mF x) v^*, v^* \> + \ddW(\mF) \| \D v
    \|_{L^2(0, \infty)}^2 \\
    &\geq \smfrac12 \gamma(\Ha_\mF) \| \D v^* \|_{L^2(\R)}^2 + \gamma(\Hc_\mF) \| \D v
    \|_{L^2(0, \infty)}\\
    &= \gamma(\Ha_\mF) \| \D v \|_{L^2(-\infty, 0)}^2 + \gamma(\Hc_\mF) \| \D v
    \|_{L^2(0, \infty)}^2 \geq \gamma(\Ha_\mF) \| \D v \|_{L^2(\R)}^2;
  \end{align*}
  that is, $\gamma(\Hrefl_\mF) \geq \gamma(\Ha_\mF)$.  Proposition \ref{th:cac_leq_ca} shows
  that this inequality is in fact an equality.
\end{proof}

\delfinal{\begin{remark}
  The reflection method $\Erefl$ is remarkably simple, both in its
  formulation and analysis. Unfortunately, the idea seems to be
  restricted to one dimension. Already for flat a/c interfaces in
  2D/3D, we observe that an antisymmetric reflection of a displacement
  does not give a generalisation of Lemma \ref{th:refl:lemma}, due to
  the fact that gradients do not become symmetric under this operation.
  %
\end{remark}}

\section{Stabilising the 1D QNL Method}
\label{sec:stab1d}
\subsection{The general strain gradient representation}
A key component in previous sharp stability analyses of a/c methods
was a decomposition of a/c hessians into the Cauchy--Born hessian and
a strain gradient correction \cite{Dobson:qce.stab, Ortner:qnl.1d,
  XHLi:3n}. Here, we generalise these representations to general
many-body finite range interactions.

\begin{lemma}
  \label{th:prod_D_lemma}
  For $\xi \in \Z, \rho \in \Rg$, define the sets 
  \begin{displaymath}
    A(\xi,\rho) := \cases{
      \{ \xi, \dots, \xi+\rho-1 \}, & \rho > 0, \\
      \{ \xi+\rho, \dots, \xi-1 \}, & \rho < 0.
    }
  \end{displaymath}
  Then, for $\xi \in \Z, \rho, \vsig \in \Rg$,
  \begin{displaymath}
    D_\rho u(\xi) D_\vsig u(\xi) = \frac{\rho\vsig}{2|\rho||\vsig|} \sum_{\eta \in A(\xi,\rho)}
    \sum_{\eta' \in A(\xi,\vsig)} \b\{ |D_1 u(\eta)|^2 + |D_1
    u(\eta')|^2 - |D_1u(\eta) - D_1 u(\eta')|^2 \b\}.
  \end{displaymath}
\end{lemma}
\begin{proof}
  It is clear from the definitions that
  \begin{displaymath}
    D_\rho u(\xi) = \frac{\rho}{|\rho|} \sum_{\eta \in A(\xi,\rho)}
    D_1 u(\eta),
  \end{displaymath}
  and therefore,
  \begin{displaymath}
    D_\rho u(\xi) D_\vsig u(\xi) = \frac{\rho\vsig}{|\rho||\vsig|} 
    \sum_{\eta \in A(\xi,\rho)}
    \sum_{\eta' \in A(\xi,\vsig)}  D_1 u(\eta) D_1 u(\eta').
  \end{displaymath}
  Applying the identity
  \begin{displaymath}
    D_{1} u(\eta)  D_{1} u(\eta') = \smfrac12 |D_{1} u(\eta)|^2 +
    \smfrac12 |D_{1} u(\eta')|^2 - \smfrac12 |D_{1} u(\eta) - D_{1} u(\eta')|^2,
  \end{displaymath}
  yields the stated result.
\end{proof}

\begin{lemma}
  \label{th:qnl:sgrad_c0_eq_ddW}
  Let $\Hac_\mF$ be of the general form \eqref{eq:qnl:hessian}, then
  \begin{equation}
    \label{eq:qnl:sgrad_2}
    \< \Hac_\mF u, u \> = \< \Hc_\mF u, u \> + \<
    \Delta^\ac_\mF u, u \>,
  \end{equation}
  where 
  \begin{align}
    \label{eq:defn_Delta_acF}
    \< \Delta^\ac_\mF u, u \> &= \sum_{j = 1}^{2\rcut-1}
    \sum_{\xi = -\infty}^0 c_j(\xi) |D_{1}u(\xi)-D_{1}u(\xi-j)|^2 \qquad
    \text{with} \\
    \notag
    c_j(\xi) &= \sum_{\rho,\vsig \in \Rg} \frac{\rho\vsig}{2|\rho||\vsig|} \sum_{\substack{\eta \in
        \Z_- \\
        \xi \in A(\eta,\rho), \xi-j \in A(\eta,\vsig)}} \tilde
    V_{\eta,\rho\vsig}(\mF \Rg).
  \end{align}
\end{lemma}
\begin{proof}
  Applying Lemma \ref{th:prod_D_lemma} to the representation
  (\ref{eq:qnl:hessian}) of the QNL hessian, we immediately obtain
  that
  \begin{equation}
    \label{eq:1d:ddelEac:prelim}
    \< \Hac_\mF u, u \> = \sum_{\xi \in \Z} c_0(\xi) |D_1 u(\xi)|^2 + \<
    \Delta^\ac_\mF u, u \>,
  \end{equation}
  where $\Delta^\ac_\mF$ is of the form \eqref{eq:defn_Delta_acF}, and
  $c_0(\xi) \in \R$ are some coefficients that still need to be
  determined. The stated identity for $c_j(\xi), j \geq 1$ in the
  definition of the strain gradient operator $\Delta^\ac_\mF$, follows
  from a straightforward exchange of summation.

  To determine $c_0(\xi)$, we first note that
  \eqref{eq:qnl:noninteraction} implies $c_0(\xi) = \ddW(\mF)$ for
  $\xi \geq 1$.
 
  To determine the remaining coefficients we apply the
  force-consistency condition (\ref{eq:qnl:force_consistency}). We
  know from (\ref{eq:qnl:force_consistency}) that
  \begin{displaymath}
    \<\del\Eac((\mF+t\mG) x), v \> = 0 \qquad \forall v \in \Usz,
  \end{displaymath}
  for all $\mF > 0$, $\mG \in \R$ and $t$ sufficiently small. Taking
  the derivative with respect to $t$, evaluated at $t = 0$, yields
  \begin{displaymath}
    \< \ddel\Eac(\mF x) \mG x, v \> = 0 \qquad \forall v \in \Usz,
  \end{displaymath}
  or, written in terms of the representation
  \eqref{eq:1d:ddelEac:prelim},
  \begin{displaymath}
    \sum_{\xi \in \Z} c_0(\xi) {(\mG\cdot a_1)} D_{1} v(\xi) 
  + \< \Delta^\ac_\mF \mG x, v \> = 0,
  \end{displaymath}
  where we extended the definition of $c_0(\xi)$ by $c_0(\xi) =
  \ddW(\mF)$ for $\xi > 0$.

  Since $\mG x$ is an affine function, $\< \Delta^\ac_\mF \mG x, v \>
  = 0$, and hence we obtain that
  \begin{displaymath}
    \sum_{\xi \in \Z} c_0(\xi) D_{1} v(\xi)  = 0 \qquad \forall v \in \Usz.
  \end{displaymath}
  This implies that $\xi \mapsto c_0(\xi)$ must be a constant, and in
  particular, $c_0(\xi) \equiv \ddW(\mF)$.
\end{proof}

\subsection{The stabilised QNL method}
We observed in Lemma \ref{th:qnl:sgrad_c0_eq_ddW} that the QNL hessian
can be written as the Cauchy--Born hessian with a strain gradient
correction in the atomistic and interface region. Moreover, due to the
``bounded interface condition'' \eqref{eq:qnl:finite_interface}, we
know that the strain gradient correction is the same for the QNL and
for the reflection hessians, except in a bounded neighbourhood of the
interface. More precisely, we can write
\begin{equation}
  \label{eq:qnl1d:Hqnl.minus.Hrfl}
  \< \Hqnl_\mF u, u \> = \< \Hrefl_\mF u, u \> + \b\< (\Delta^\qnl_\mF
  - \Delta^\refl_\mF) u, u \>,
\end{equation}
where 
\begin{align*}
  \b\< (\Delta^\qnl_\mF
  - \Delta^\refl_\mF) u, u \> &= \sum_{j = 1}^{2\rcut-1} \sum_{\xi =
    \xi_1}^{-1} c_j'(\xi) |D_1 u(\xi) - D_1 u(\xi-j)|^2,
\end{align*}
for some $\xi_1 \leq 0$ that depends on $\xi_0$ and on $\rcut$, and
for coefficients $c_j'(\xi) := c_j(\xi) - c_j^\refl(\xi)$. If
$c_j'(\xi) \geq 0$ for all $\xi$, then we would obtain that $\<
\Hqnl_\mF u, u \> \geq \< \Hrefl_\mF u, u \>$ and hence the QNL method
is universally stable.

If $c_j'(\xi) < 0$ for some $j, \xi$, then we can redefine a {\em
  stabilised QNL energy}
\begin{displaymath}
  \label{eq:1dqnl:defn_Estab}
  \Estab(y) := \Eqnl(y) + \kappa \< S u, u \>, \qquad \text{for } y = \mF x
  + u, u \in \Usz,
\end{displaymath}
where $\kappa > 0$ is a stabilisation constant and $S$ is the
stabilisation operator defined through
\begin{equation}
  \label{eq:qnl1d:defn_S}
  \< S u, u \> := \sum_{\xi = \xi_1-2\rcut+2}^{-1} |D_{-1}D_1 u(\xi)|^2.
\end{equation}
Because the stabilisation involves only second derivatives, this
modification does not affect the first-order consistency of the QNL
method; see Remark \ref{rem:cons_stabQNL}.

\begin{theorem}
  Fix a bounded set $\mathcal{F} \subset \R$ (a range of macroscopic
  strains $\mF$ of interest). Then there exists a constant $\kappa_0
  \geq 0$ such that, for all $\kappa \geq \kappa_0$ and for all $\mF
  \in \mathcal{F}$, $\ddel\Estab(\mF x)$ is stable if and only if
  $\Ha_\mF$ is stable.

  An upper bound on $\kappa_0$ is given by
  \begin{displaymath}
    \kappa_0 \leq \sup_{\mF \in \mathcal{F}} \sum_{\rho,\vsig \in \Rg} (|\rho|+|\vsig|)^2
    |\rho||\vsig| {\sup_{\xi\in\Z_-}} \b| V_{\xi,\rho\vsig}(\mF
    \Rg) - V^\refl_{\xi,\rho\vsig}(\mF) \b|,
  \end{displaymath}
  where $V^\refl_\xi$ is the effective site potential of the
  reflection scheme.
\end{theorem}
\begin{proof}
  We know from Proposition \ref{th:cac_leq_ca} that, if $\Ha_\mF$ is
  unstable, then $\Hstab_\mF$ is unstable, so we only need to prove
  the converse statement.

  Since the reflection method is universally stable, it follows from
  \eqref{eq:qnl1d:Hqnl.minus.Hrfl} that it is sufficient to prove that
  \begin{displaymath}
    \b\< (\Delta^\qnl_\mF
    - \Delta^\refl_\mF) u, u \>  + \kappa \< S u, u \> \geq 0,
  \end{displaymath}
  for $\kappa$ sufficiently large. To prove that this is indeed the
  case, we simply compute an upper bound on $|\b\< (\Delta^\qnl_\mF -
  \Delta^\refl_\mF) u, u \>|$:
  \begin{align*}
    \b|\< (\Delta^\qnl_\mF -
  \Delta^\refl_\mF)u, u \>\b| &\leq \sum_{j = 1}^{2\rcut-1} \sum_{\xi =
      \xi_1}^{-1} |c_j'(\xi)| |D_1 u(\xi) - D_1 u(\xi-j)|^2 \\
    &\leq \sum_{j = 1}^{2\rcut-1} \sum_{\xi =
      \xi_1}^{-1} |c_j'(\xi)| j \sum_{\eta = \xi-j+1}^{\xi} \b|D_{-1}D_1
    u(\eta)\b|^2,
  \end{align*}
  where we used the Cauchy--Schwarz (or, Jensen's) inequality. Upon
  reordering the summation, we obtain
  \begin{align*}
    \b|\< (\Delta^\qnl_\mF -
  \Delta^\refl_\mF) u, u \>\b| &\leq \sum_{\eta = \xi_1 - 2 \rcut +
      2}^{-1} \b|D_{-1}D_1
    u(\eta)\b|^2 \Bg\{ \sum_{j = \max(1, \xi_1 - \eta + 1)}^{2\rcut-1}
    \sum_{\xi = \max(\eta, \xi_1)}^{\min(\eta+j-1, -1)} |c_j'(\xi)| j
    \Bg\} \\
    &=: \sum_{\eta = \xi_1 - 2 \rcut +
      2}^{-1} \b|D_{-1}D_1
    u(\eta)\b|^2 d'(\mF, \eta).
  \end{align*}
  Letting $\kappa_0 := \max_{\eta, \mF \in \mathcal{F}} d'(\mF, \eta)$
  yields the result.

  To get an upper bound on this quantity, we next estimate
  $|c_j'(\xi)|$.
  %
  Let 
  \begin{displaymath}
    m'(\rho,\vsig) := \sup_{\xi\in\Z_-} \sup_{\mF \in \mathcal{F}}
    |V_{\xi,\rho\vsig} - V_{\xi,\rho\vsig}^\refl|,
  \end{displaymath}
  then
  \begin{align*}
    |c_j'(\xi)| \leq \frac12 \sum_{\rho,\vsig \in \Rg}
    \sum_{\substack{\eta \in \Z_- \\ \xi \in A(\eta,\rho), \xi-j \in
        A(\eta,\vsig)}} m'(\rho,\vsig),
  \end{align*}
  and noting that the sum over $\eta$ is taken over at most
  $\min(|\rho|,|\vsig|)$ sites and moreover that only the sum over
  $\rho,\vsig$ satisfying $|\rho| + |\vsig| \geq j$ needs to be taken
  into account, we obtain
  \begin{displaymath}
    |c_j'(\xi)| \leq \frac12 \sum_{\substack{\rho,\vsig \in \Rg \\
        |\rho|+|\vsig| \geq j}}
    \min(|\rho|,|\vsig|) m'(\rho,\vsig).
  \end{displaymath}
  Inserting this estimate into the definition of $d'(\mF, \eta)$ gives
  \begin{align*}
    d'(\mF, \eta) &\leq \sum_{j = \max(1, \xi_1 - \eta + 1)}^{2\rcut-1}
    \sum_{\xi = \max(\eta, \xi_1)}^{\min(\eta+j-1, -1)} \sum_{\substack{\rho,\vsig \in \Rg \\
        |\rho|+|\vsig| \geq j}}
    \smfrac12 \min(|\rho|,|\vsig|) (|\rho|+|\vsig|) m'(\rho,\vsig),
  \end{align*}
  where we estimated $j \leq (|\rho|+|\vsig|)$. Next, using $\smfrac12
  \min(|\rho|,|\vsig|) (|\rho|+|\vsig|) \leq |\rho||\vsig|$, and
  noting that the sum over $\xi$ ranges over at most $j$ values, we
  further estimate
  \begin{align*}
    d'(\mF, \eta)  &\leq  \sum_{j = \max(1, \xi_1 - \eta +
      1)}^{2\rcut-1} j \sum_{|\rho|+|\vsig| \geq j} |\rho||\vsig|
    m'(\rho,\vsig) \\
    &\leq \sum_{\rho,\vsig \in \Rg} |\rho||\vsig|
    m'(\rho,\vsig)
    \sum_{j = 1}^{\min(2\rcut-1, |\rho|+|\vsig|)} j 
    \leq \sum_{\rho,\vsig \in \Rg} |\rho||\vsig| (|\rho|+|\vsig|)^2
    m'(\rho,\vsig).
  \end{align*}
  This establishes the estimate for $\kappa_0$.
\end{proof}

\begin{remark}[Consistency of the stabilised QNL method]
  \label{rem:cons_stabQNL}
  If the cost of stabilising the QNL method is a loss in consistency,
  then little can be gained by the procedure proposed in the foregoing
  section. However, (ignoring finite element coarsening of the
  continuum region) it is easy to show that
  \begin{displaymath}
    \| \del\Estab(u) - \del\Ea(u) \|_{\Us^*} \leq \| \del\Eqnl(u) -
    \del\Ea(u) \|_{\Us^*} + 2 \kappa_0 \| D_{-1}D_1 u
    \|_{\ell^2(\mathcal{I})}, 
  \end{displaymath}
  where $\mathcal{I} := \{ \xi_1 - 2\rcut+1, \dots, -1 \}$.  That is,
  the additional consistency error committed by the stabilisation is
  of first-order, which is the same as the consistency error of the
  QNL method
  \cite{Ortner:qnl.1d,Or:2011a,PRE-ac.2dcorners,DobsonPRE}. 

  Moreover, the prefactor $\kappa_0$ is bounded in terms of the
  partial derivatives $V_{\xi,\rho\vsig}$. Having some uniform bound
  on these partial derivates $V_{\xi,\rho\vsig}$ is a prerequisite to
  obtain a first-order error estimate
  \cite{Or:2011a,PRE-ac.2dcorners}. For example, for geometric
  reconstruction type method
  \cite{Shimokawa:2004,E:2006,PRE-ac.2dcorners} one can show that
  these are bounded in terms of a norm on the reconstruction
  coefficients.

  In summary, we can conclude that the stabilisation
  \eqref{eq:qnl1d:Hqnl.minus.Hrfl} will normally not affect the
  consistency of the QNL method.
\end{remark}

\subsection{Numerical example}
We may now revisit the numerical example from
\S~\ref{sec:instab_qnl_num}, and add the universally stable reflection
method and the stabilised QNL method to the graph. We choose the QNL
stabilisation parameter $\kappa = 0.1$ by trial and error. The
extension of the two methods to the finite domain used in this
experiment is straightforward.

\begin{figure}
  \begin{center}
    \includegraphics[width=14cm]{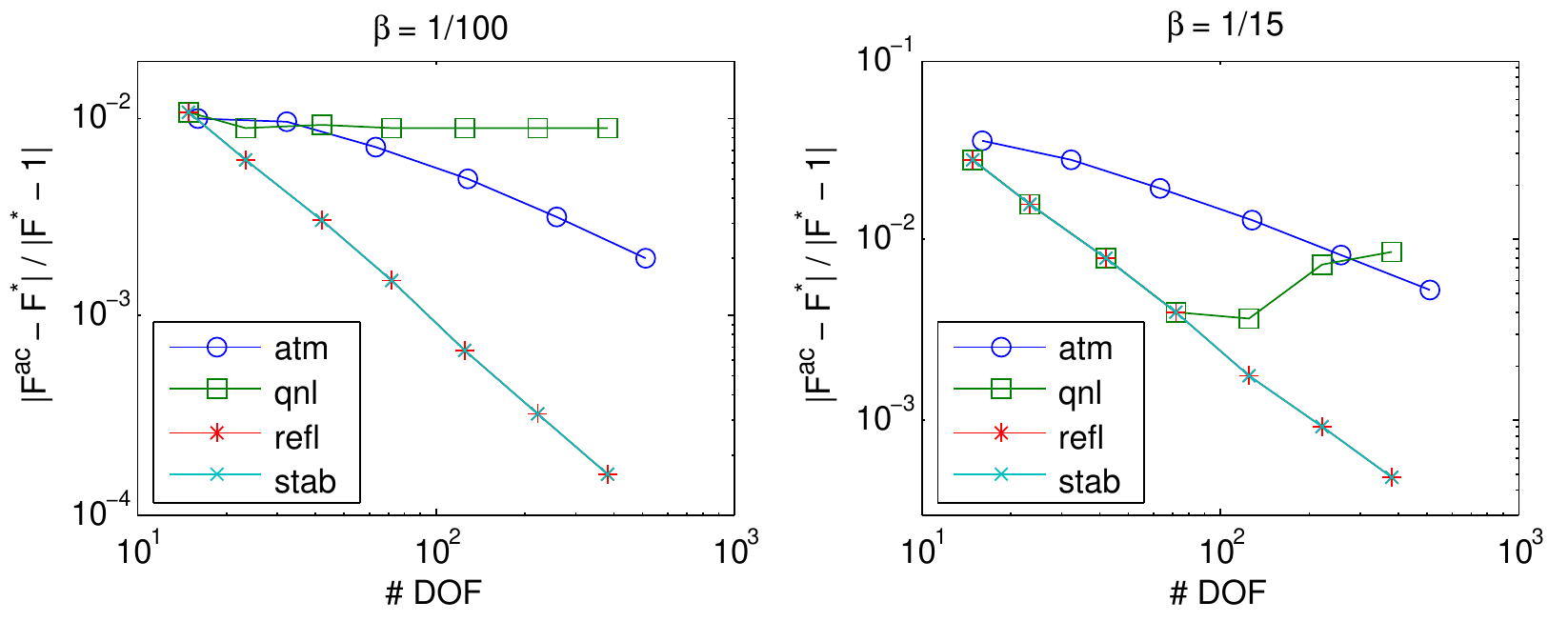}
    \caption{\label{fig:1dqnl_stab} Relative errors of critical
      strains for the QNL, REFL and stabilized QNL methods, and
      external forces parameterised by $\alpha = 1.5, \beta \in \{
      0.01, 0.066 \}$. The stabilisation parameter for the stabilised
      QNL method is $\kappa = 0.1$.}
  \end{center}
\end{figure}

The result is displayed in Figure \ref{fig:1dqnl_stab}. We observe
clear systematic convergence of the critical strains for both the
reflection method and the stabilised QNL method, which is consistent
with our analysis in the foregoing sections.

\bigskip

\bigskip

\section{QNL Formulation of a 2D Nearest-Neighbour Scalar Model}
\label{sec:qnl2d_model}
In the remainder of the paper we explore possible generalisations of
our foregoing results to higher dimensions. We are unable, at present,
to provide results of the same generality as in 1D, and we therefore
restrict our presentation to the setting of nearest-neighbour many
body interactions for scalar displacement fields (e.g., anti-plane
displacements) in two dimensions, with a ``flat'' a/c
interface. Already in this simple setting, we will encounter several
difficult new issues that must be overcome before focusing on the even
more challenging vectorial case, and general interface
geometries. (Admitting a wider interaction range does not seem to
cause major additional difficulties.)

\subsection{Notation for the 2D triangular lattice}
Our 2D analysis is most convenient to perform in the setting of the 2D
triangular lattice, which we denote by
\begin{displaymath}
  \L := \mA \Z^2, \quad \text{where} \quad
  \mA = \begin{pmatrix} 1 & \cos(\pi/3) \\ 0 & \sin(\pi/3) \end{pmatrix}.
\end{displaymath}
For future reference, we define the six nearest-neighbour lattice
directions by $a_1 := (1, 0)$, and $a_j := \mQ_6^{j-1} a_1$, $j \in
\Z$, where $\mQ_6$ denotes the rotation through angle $2\pi/6$ and we
note that $a_{j+3} = - a_j$.

For a lattice function $w : \L \to \R$, we define the
nearest-neighbour differences
\begin{displaymath}
  D_j w(\xi) := w(\xi + a_j) - w(\xi).
\end{displaymath}
The interaction range is defined as $\Rg = \{ a_1, \dots, a_6 \}$ and
the corresponding finite difference stencil by $Dw(\xi) = \{ D_j
w(\xi) \}_{j = 1}^6.$ Let $\| Dw \|_{\ell^2} :=
(\sum_{\xi\in\L}\sum_{j = 1}^3 |D_j w(\xi)|^2)^{1/2}$.

Let $\T$ denote the canonical triangulation of $\R^2$ with nodes $\L$,
using {\em closed} triangles, then each lattice function $v$ is
identified with its continuous piecewise affine interpolant with
respect to $\T$. In particular, we define $\D v_T$ to be the gradient
of $v$ in $T \in \T$ and we note that $\D v_T \cdot a_j = D_j v(\xi)$
if $\xi, \xi+a_j \in T$.

The space of admissible test functions is again the space of compactly
supported lattice functions, defined by
\begin{displaymath}
  \Usz := \b\{ u : \L \to \R \bsep {\rm supp}(u) \text{ is bounded} \b\}.
\end{displaymath}

For an operator $H : \Usz \to \Usz^*$ we define again $\gamma(H) :=
\inf_{u \in \Usz, \| \D u \|_{L^2} = 1} \< H u, u \>$.

\subsection{2D many-body nearest neighbour interactions}
We fix a nearest-neighbour many-body (i.e., 7-body) potential $V \in
C^2(\R^6)$, with partial derivatives
\begin{displaymath}
  V_{i}(\bfg) = \frac{\partial V(\bfg)}{\partial g_i} \quad \text{and}
  \quad
  V_{ij}(\bfg) = \frac{\partial^2 V(\bfg)}{\partial g_i \partial g_j}
  \quad \text{ for } \bfg = (g_i)_{i = 1}^6 \in \R^6.
\end{displaymath}

For a deformed configuration $y = \mF \cdot x + u$ (where $x(\xi) =
\xi$ and $\mF \in \R^2$) we define the energy difference by
\begin{equation}
  \label{eq:defn_2d_atm}
  \Ea(y) = \sum_{\xi\in\L} \b[ V(Dy(\xi)) - V(\mF \Rg) \b].
\end{equation}
Since the sum is effectively finite $\Ea$ is well-defined and admits
two variations in the sense of Gateaux derivatives, with the second
variation given by 
\begin{align*}
  \< \ddel\Ea(y) v, v \> &= \sum_{\xi \in \L} \sum_{i,j = 1}^6
  V_{ij}(Dy(\xi)) \cdot D_i v(\xi) D_j v(\xi).
\end{align*}

We are again particularly interested in homogeneous states $y(x) = \mF
x$ and define
\begin{equation}
  \label{eq:2d:HaF_defn}
  \< \Ha_\mF u, u \> = \sum_{\xi\in\L} \sum_{i,j = 1}^6 V_{ij} \cdot D_i
  u(\xi) D_j u(\xi),
\end{equation}
where, here and throughout we omit the argument $\mF \Rg$ in $V_{ij}$
when it is clear from the context that we mean $V_{ij}(\mF\Rg)$.

\subsubsection{Symmetries}
Inversion symmetry about each lattice point leads us to assume that
$V( (g_i)_{i = 1}^6 ) = V( (-g_{i'})_{i = 1}^6 )$, where $i' \in
\{1,\dots,6\}$ such that $a_{i'} = - a_i$. This yields the point
symmetry for the second derivatives $V_{i,j}(\mF\Rg) =
V_{i',j'}(\mF\Rg)$ for $i, j \in \{1,\dots, 6\}$; see, e.g.,
\cite{OrtnerTheil2012}.
Since the reference lattice $\L$ has full hexagonal symmetry, it
  is reasonable to make the stronger assumption that $V$ has full
  hexagonal symmetry as well, i.e.,
\begin{equation}
  V( \bfg) = V(g_6, g_1,
  \dots, g_5).
\end{equation} 
In this case, but only for the deformation $\mF = \mO$, one can
readily deduce the identities
\begin{equation}
  \label{eq:2d:hex_symm}
\begin{split}
  V_{1,1} = \dots = V_{6,6} &=: \alpha_0, \\
  V_{1,2} = \dots = V_{5,6} = V_{6,1} &=: \alpha_1, \\
  V_{1,3} = \dots = V_{4,6} = V_{5,1} = V_{6,2} &=: \alpha_2, \quad
  \text{and} \\
  V_{1,4} = V_{2,5} = V_{3,6} &=:  \alpha_3,
\end{split}
\end{equation}
where $V_{i,j} = V_{i,j}({\bf 0})$ and $\alpha_i \in \R$. 

Both symmetries can be derived, e.g., by reducing a 3D model to a
scalar 2D anti-plane model.

\subsection{QNL-type methods}
\label{sec:2d:def_qnl}
%
%
We define the Cauchy--Born approximation in a discrete sense,
\begin{displaymath}
  \Ec(y) := \frac12 \sum_{T \in \T} \b[ W(\D y_T) - W(\mF) \b],
\end{displaymath}
where $W(\mF) := V(\mF \Rg)$. Unusually, we have not normalised $W$
with respect to volume, which somewhat simplifies notation. (Since
  each site has associated volume $1$, each element has associated
  volume $3 / 6 = 1/2$.)

We define the atomistic and continuum lattice sites
\begin{displaymath}
  \La := \{ \xi \in \L \sep \xi_2 < 0 \}, \quad \Lc := \{ \xi \in \L
  \sep \xi_2 > 0 \}, 
\end{displaymath}
and in addition the $k$th ``row'' of atoms
\begin{displaymath}
  \L^{(k)} := \b\{ \xi \in \L \bsep \xi_2 = k \sqrt{3}/2 \b\},
\end{displaymath}
so that $\L^{(0)}$ is the set of interface lattice sites.

QNL methods are a/c coupling schemes with energy functional of the
form
\begin{align}
  E^\qnl(y) &:= \sum_{\xi \in \La} \b[ V(Dy(\xi)) - V(\mF \Rg) \b] 
  \label{eq:qnl2d:defn_Eqnl}
  +\sum_{\xi \in \L^{(0)}} \b[ \tilde{V}(Dy(\xi)) - V(\mF\Rg)
  \b]
  + \sum_{\xi \in \Lc} \frac13 \sum_{\substack{T \in \T \\
      \xi \in T}} \b[ W(\D y_T) - W(\mF) \b],
\end{align}
where $\tilde{V}$ is a modified interaction potential that is chosen
to transition between the atomistic and Cauchy--Born description. For
more detail we refer to \cite{Shimokawa:2004,E:2006,Or:2011a} and in
particular \cite{PRE-ac.2dcorners} which is closest in terms of
analytical setting and notation to our present work.

We assume throughout that $\tilde{V} \in C^2(\R^6)$, then the QNL
energy is well-defined for $y = \mF \cdot x + u, u \in \Usz$, and has
two variations in the sense of Gateaux derivatives.

We assume that $\Eqnl$ does not exhibit ghost forces,
\begin{equation}
  \label{eq:2d:qnl_nogf}
  \< \del\Eqnl(\mF x), v \> = 0 \qquad \forall v \in \Usz, \mF \in \R^2,
\end{equation}
and is {\em energy-consistent,}
\begin{equation}
  \label{eq:2d:qnl_energy_cons}
 \tilde{V}(\mF \Rg) = V(\mF\Rg)
 \qquad\forall \mF \in \R^2
 .
\end{equation}

Sometimes, to achieve a more compact notation, we write
\begin{displaymath}
  \Eqnl(y) = \sum_{\xi \in \La\cup\L^{(0)}} \b[\tilde{V}_\xi(Dy(\xi)) -
  V(\mF \Rg) \b] + \sum_{T \in \T} w_T \b[ W(\D y_T) - W(\mF) \b],
\end{displaymath}
where $\tilde{V}_\xi = \tilde{V}$ for $\xi \in \L^{(0)}$, $\tilde{V}_\xi
= V$ for $\xi \in \La$, and $w_T = \# (\Lc \cap T)/6$. The second
variation (hessian) at $y = \mF x$, $\Hac_\mF = \ddel\Eqnl(\mF x)$, is
then given by
\begin{equation}
  \label{eq:2d:qnl_hessian_defn}
  \< \Hac_\mF u, u \> = \sum_{\xi \in \La\cup\L^{(0)}} \sum_{i,j = 1}^6
  \tilde{V}_{\xi,ij} \cdot D_i u(\xi) D_j u(\xi) + \sum_{T \in \T} w_T
  (\D u_T)^\top \ddW(\mF) \D u_T,
\end{equation}
where $\ddW(\mF) \in \R^{2 \times 2}$ is the hessian of $W$.

As in the foregoing 1D results we shall focus exclusively on stability
at homogeneous states. We show in \S~\ref{sec:stab_inhom} how one may
extend such results to stability of non-homogeneous states including
defects.

We remark that the 2D variant of Lemma \ref{th:cc_geq_ca},
$\gamma(\Ha_\mF) \leq \gamma(\Hc_\mF)$, remains true
\cite{Hudson:stab}.


To illustrate that we are not talking about abstract methods, but
concrete practical formulations we now introduce three specific
variants.

\subsubsection{The QCE method}
The simplest QNL variant is the QCE method
\cite{Ortiz:1995a,Dobson:2008c}, which is defined by simply taking
$\tilde{V} = V$. It is shown in \cite{PRE-ac.2dcorners} that in our
present setting (nearest neighbour interaction, flat interface) it
satisfies the force-consistency condition \eqref{eq:2d:qnl_nogf}.

We denote the resulting energy functional by $\E^\qce$.

\subsubsection{The GRAC-2/3 method}
The QCE method does {\em not} satisfy the force consistency condition
\eqref{eq:2d:qnl_nogf} in domains with corners, nor for second
neighbour interactions
\cite{Shenoy:1999a,Shimokawa:2004,E:2006,Dobson:2008c,PRE-ac.2dcorners}
and it is still an open problem to formulate a general scheme that
does. A class of methods has been introduced in
\cite{PRE-ac.2dcorners}, extending ideas in \cite{Shimokawa:2004,
  E:2006}, which in our context can be defined through
\begin{displaymath}
  \tilde{V}(Dy) := V(\tilde{D} y), \quad \text{ where }\quad \tilde{D}_i y
  := \lambda_i D_{i-1} y + (1-\lambda_i) D_i y + \lambda_i D_{i+1} y,
\end{displaymath}
for $\lambda_i \in \R$. It is shown in \cite{PRE-ac.2dcorners} that,
for flat interfaces, all of these schemes satisfy
\eqref{eq:2d:qnl_nogf}, and for the choice
\begin{displaymath}
  \lambda_i = \cases{1/3, & i = 2, 3 \\
    0, & i = 1, 4, 5, 6},
\end{displaymath}
(and only for this choice) the resulting method (GRAC-2/3) can be
extended to domains with corners while still satisfying
\eqref{eq:2d:qnl_nogf}. We denote the resulting energy functional by
$\E^\grac$.

\subsubsection{The local reflection method}
Finally, we introduce a new a/c coupling scheme, inspired by our 1D
reflection method.

The idea is to apply the reflection method on each site $\xi \in
\L^{(0)}$, which amounts to defining
\begin{displaymath}
  \tilde{D}_i := \cases{ D_i, & i = 1, 4, 5, 6 \\
    - D_{i+3}, & i = 2, 3,
  } \quad \text{and} \quad  \tilde{V}(Dy) := \frac12 V(\tilde{D} y) + \frac16 \sum_{\substack{T \in \T^\c
      \\ \xi \in T}} W(\D y_T),
\end{displaymath}
where $\T^\c := \{ T \in \T \sep x_2 \geq 0 \text{ for all } x \in
T\}$.

The idea can be seen more clearly, if we write the resulting energy
functional in the form
\begin{align*}
  \E^\lref(y) &:= \sum_{\xi \in \La} \b[V(Dy(\xi))-V(\mF\Rg)\b] +
  \frac12 \sum_{\xi \in \L^{(0)}} \b[ V(\tilde{D}y(\xi))-V(\mF\Rg)\b] 
  + \frac12 \sum_{T \in \T^\c} \b[ W(\D y_T) - W(\mF) \b].
\end{align*}

It is straightforward to check that this method exhibits no ghost
forces.

\subsection{Atomistic and Cauchy--Born hessian representations}
\label{sec:2d:hessian_rep}
Our aim is to develop a generalisation of our 1D hessian
representation, Lemma \ref{th:qnl:sgrad_c0_eq_ddW}. Towards this end,
we first establish representations for the atomistic and Cauchy--Born
hessians. The result for the QNL hessian will be presented in
\S~\ref{sec:2d:qnl_hessian_rep}.

We first state two auxiliary lemmas. The first provides a mechanism
for establishing whether two symmetric bilinear forms are equal.

\begin{lemma}
  \label{th:equiv_hessians}
  Let $H_1, H_2$ be self-adjoint operators defined through
  \begin{displaymath}
    \< H_i u, u \> = \sum_{\xi \in \L} \sum_{j = 1}^3 h_{i,j}(\xi) |D_j u(\xi)|^2,
  \end{displaymath}
  then $H_1 = H_2$ if and only if $h_{1,j}(\xi) = h_{2,j}(\xi)$ for
  all $\xi \in \L$, $j = 1, \dots, 3$.
\end{lemma}
\begin{proof}
  For some $\eta \in \L$ and $j \in \{1,2,3\}$, we define $u(\xi)
  = \delta_{\xi,\eta}$ and $v(\xi) = \delta_{\xi,\eta+a_j}$, where $\delta$ is the Kronecker delta. Then the
  product $D_k u(\xi) D_k v(\xi)$ is non-zero if and only if $\xi =
  \eta$ and $k = j$. Hence,
  \begin{displaymath}
    0  = \< (H_1 - H_2) u, v \>  = - (h_{1,j}(\eta) - h_{2,j}(\eta)).
  \end{displaymath}
  Hence we conclude that $h_{1,j}(\eta) = h_{2,j}(\eta)$ for all $\eta
  \in \L$ and $j = 1, 2, 3$. The converse implication is trivial.
\end{proof}

In the ``canonical'' hessian representations of $\Ea, \Ec, \E^\qnl$
products of finite differences $D_iu(\xi)D_ju(\xi)$ occur; see
\eqref{eq:2d:HaF_defn} and \eqref{eq:2d:qnl_hessian_defn}. In 1D, we
converted these products into squares of strains and strain gradients.
The next lemma provides an analogous representation for general mixed
differences. In appendix \ref{sec:app_general_hessrep} we state the
generalisation for general finite range interaction.

\begin{figure}
  \begin{center}
    \includegraphics[height=2.7cm]{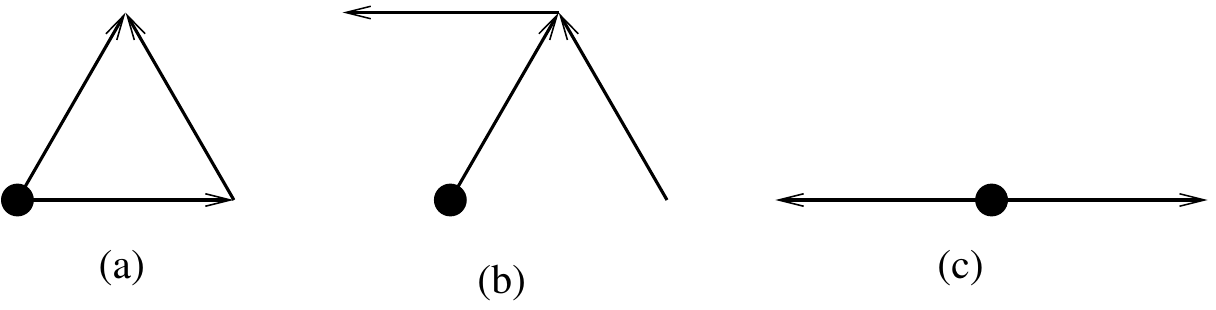}
    \caption{\label{fig:rules} Visualisation of the identities
      \eqref{eq:RULE-1}--\eqref{eq:RULE-3}. The bullets denote the
      sites $\xi$, while the arrows denote the terms $|D_j u(\eta)|^2$
      occuring in these identities. (a) visualises \eqref{eq:RULE-1};
      (b) visualises \eqref{eq:RULE-2}; (c) visualises
      \eqref{eq:RULE-3}.}
  \end{center}
\end{figure}
\begin{lemma}
	\label{th:elemprod0_2D_lemma}
  Let $u \in \Usz$, $\xi \in \L$ and $i \in \{1,\dots, 6\}$, then
  \begin{align}
    \label{eq:RULE-1}
    D_i u(\xi) D_{i+1} u(\xi) &= \smfrac12 |D_i u(\xi)|^2 +
    \smfrac12|D_{i+1} u(\xi)|^2 - \smfrac12 |D_{i+2} u(\xi+a_i)|^2, \\
    \label{eq:RULE-2}
    D_{i}u(\xi)D_{i+2}u(\xi) &= \smfrac12 |D_{i+1} u(\xi)|^2 -
    \smfrac12 |D_{i+2} u(\xi+a_i)|^2 \\
    \notag
    & \qquad - \smfrac12 |D_{i+3}
    u(\xi+a_{i+1})|^2 + \smfrac12 |D_i D_{i+2} u(\xi)|^2, \\
    \label{eq:RULE-3}
    D_i u(\xi) D_{i+3} u(\xi) &= - \smfrac12 |D_i u(\xi)|^2 -
    \smfrac12 |D_{i+3} u(\xi)|^2
    + \smfrac12 |D_{i+3}D_i u(\xi)|^2.
  \end{align}
\end{lemma}
\begin{proof}
  All three identities are straightforward to verify by direct
  calculations. 
\end{proof}

\begin{proposition}[Cauchy--Born Hessian]
  \label{th:2d:cb_hess}
  There exist $c_j = c_j(\mF)$, $j = 1, 2, 3$, such that
  \begin{displaymath}
    \< \Hc_\mF u , u \> = \sum_{j = 1}^3 c_j \sum_{\xi \in \L} |D_j u(\xi)|^2,
  \end{displaymath}
  where $\ddW(\mF) = \smfrac12 \sum_{j = 1}^3 c_j a_j \otimes a_j$.

  In the hexagonally symmetric case \eqref{eq:2d:hex_symm}, we have
  $c_1 = c_2 = c_3 =: c$.
\end{proposition}
\begin{proof}
  The result can be checked by a straightforward calculation. The
  complete proof is given in Appendix~\ref{sec:proof_2d:cb_hess}.
\end{proof}


Next, we establish the ``strain-gradient'' representation of the
atomistic hessian. We define a {\em sum of squares} $p : \R^K \to \R$
to be a diagonal homogeneous quadratic, i.e., a function of the form
$p(z) = \sum_{k = 1}^K c_k z_k^2$.

\begin{proposition}
  \label{th:2d:atm_hessrep}
  There exists a sum of squares $X = X_\mF : \R^{36} \to \R$, such
  that
  \begin{displaymath}
    \< \Ha_\mF u, u \> = \< \Hc_\mF u, u \> + \sum_{\xi \in \L} X(D^2 u),
  \end{displaymath}
  where $D^2 u(\xi) = (D_i D_j u(\xi))_{i,j = 1}^6$.
\end{proposition}
\begin{proof}
  Applying the identities \eqref{eq:RULE-1}--\eqref{eq:RULE-3} to the
  original form \eqref{eq:2d:HaF_defn} of $\Ha_\mF$, and noting the
  translation invariance of these operations, we obtain
  \begin{displaymath}
    \< \Ha_\mF u, u \> = \sum_{j = 1}^3 c_j^\a \sum_{\xi \in \L} |D_j
    u(\xi)|^2 + \sum_{\xi \in \L} X(D^2 u(\xi)),
  \end{displaymath}
  where $X(D^2 u) = \sum_{i,j} b_{ij} |D_iD_j u|^2$ for some
  coefficients $b_{i,j} \in \R$. It only remains to show that $c_j^\a
  = c_j$ for $j = 1, 2, 3$.

  To prove this, we use a scaling argument. Let $u \in
  C^\infty_0(\R^2)$, and let $u^{(\eps)}(\xi) := \eps u(\eps \xi)$,
  then it is elementary to show that
  \begin{align*}
    \smfrac{2}{\sqrt{3}} \b\< \Hc_\mF u^{(\eps)}, u^{(\eps)} \b\> &\to \int_{\R^2} \sum_{j = 1}^3 c_j
    |\D u \cdot a_j|^2 \dx = \int_{\R^2} \D u^T \mC \D u \dx, \quad \text{and} \\
    \smfrac{2}{\sqrt{3}} \b\< \Ha_\mF u^{(\eps)}, u^{(\eps)} \b\> &\to \int_{\R^2} \sum_{j = 1}^3 c_j^\a
    |\D u \cdot a_j|^2 \dx = \int_{\R^2} \D u^T \mC^\a \D u \dx,
  \end{align*}
  where $\mC = \sum_{j = 1}^3 c_j a_j \otimes a_j$ and $\mC^\a =
  \sum_{j = 1}^3 c_j^\a a_j \otimes a_j$. (The factor $2/\sqrt{3}$
  accounts for the density of lattice sites.)  Moreover, since
  $\Hc_\mF$ is the hessian of the Cauchy--Born continuum model,
  restricted to a P1 finite element space, we know that the two limits
  must be identical,
  $\int \D u^T \mC \D u \dx = \int \D u^T \mC^\a \D u \dx$,
  which is only possible if $\mC = \mC^\a$. Since the three rank-1
  matrices $a_j \otimes a_j, j = 1, 2, 3,$ are linearly independent,
  we can conclude that $c_j = c_j^\a$ for $j = 1, 2, 3$. 
  %
\end{proof}

\subsection{Simple cases}
\label{sec:2d:simple_cases}
  %
  1. Suppose that the potential $V$ is such that $V_{i,i+2} =
  V_{i,i+3} \equiv 0$ for all $i = 1, \dots, 6$; that is, only the
  ``neighbouring bonds'' interact. (In the hexagonally symmetry case,
  this amounts to assuming that $\alpha_2 = \alpha_3 = 0$.) This
  could, for example, be understood as a simple case of bond-angle
  interaction. Then, in the proof of Proposition
  \ref{th:2d:atm_hessrep}, only the identity \eqref{eq:RULE-1} is
  employed but neither \eqref{eq:RULE-2} nor
  \eqref{eq:RULE-3}. Therefore, $X \equiv 0$, and we obtain that
  $\Ha_\mF = \Hc_\mF$. 

  2. In the hexagonally symmetric case \eqref{eq:2d:hex_symm}, without
  assuming $\alpha_2 = \alpha_3 = 0$, a straightforward explicit
  computation yields 
  \begin{align}
    \label{eq:Ha_coeffs_hexsym}
    c &= 2 (\alpha_0 +\alpha_1 - \alpha_2 - \alpha_3), \qquad
    \text{and} \\
    \notag
    X(D^2 u) &= \sum_{i = 1}^6 \b(\alpha_2
    |D_{i+2}D_i u|^2 + \alpha_3 |D_{i+3}D_i u|^2 \b). \qedhere
  \end{align}


\section{Instability and Stabilization in 2D}
\label{sec:2d:qnl_hessian_rep}
In this section we will derive the ``strain gradient'' representation
of the QNL hessian.  We shall find that, in contrast to our
one-dimensional result (Lemma \ref{th:qnl:sgrad_c0_eq_ddW}), in 2D
there is a source of instability that is different from an error in
the strain gradient coefficients, and therefore more severe.

\subsection{QNL hessian representation}
\label{sec:2d:qnlhess_rep}
%





Applying the rules \eqref{eq:RULE-1}--\eqref{eq:RULE-3} to the
``canonical'' QNL hessian representation
\eqref{eq:2d:qnl_hessian_defn} we obtain the following result.

\begin{figure}
  \begin{center}
    \includegraphics[height=5cm]{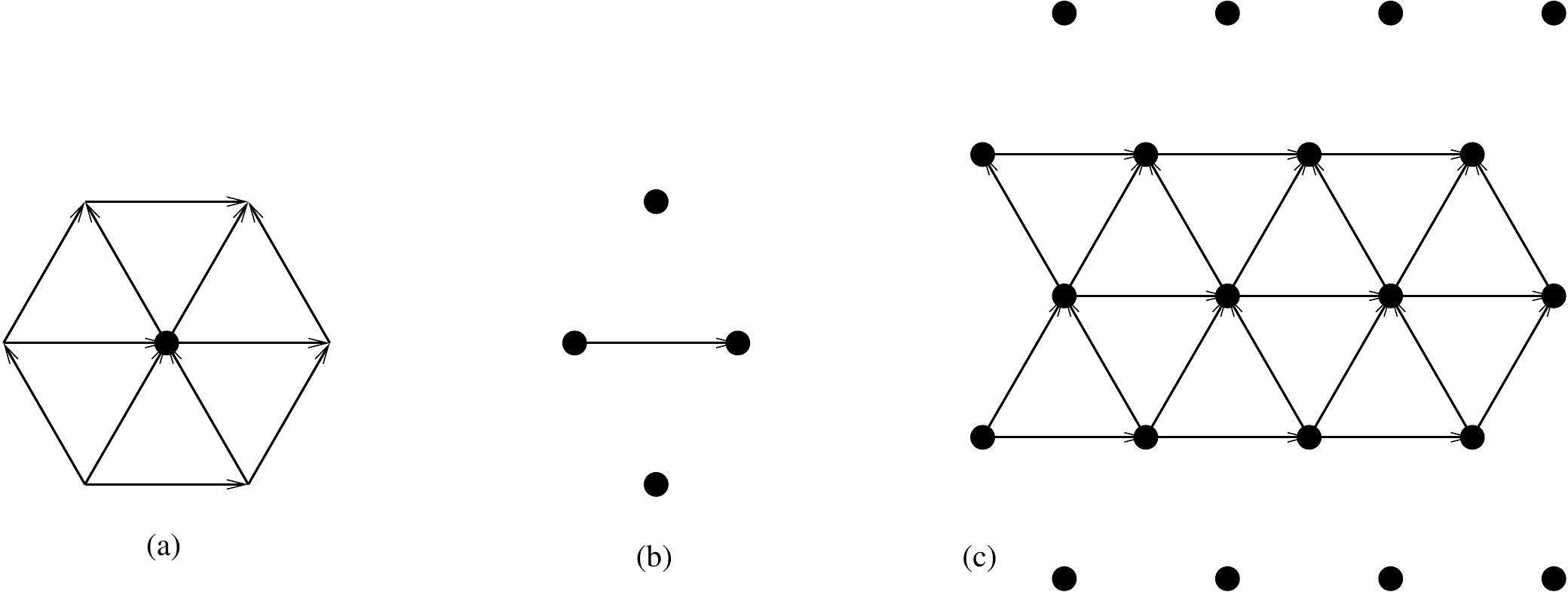}
    \caption{\label{fig:influence} (a) Bonds (arrows) that are
      affected by the operations \eqref{eq:RULE-1}--\eqref{eq:RULE-3},
      from a single size $\xi$ (black disk); (b) Sites (black disks)
      that affect a given bond (arrow) through the operations
      \eqref{eq:RULE-1}--\eqref{eq:RULE-3}. (c) Bonds for which the
      coefficients $\tilde{c}_j(\xi)$ of the a/c hessian differ from
      the coefficients $c_j$ of the Cauchy--Born hessian;
      cf. Proposition \ref{th:2d:qnl_hess_straingrad}.}
  \end{center}
\end{figure}

\begin{proposition}
  \label{th:2d:qnl_hess_straingrad}
  There exist coefficients $\tilde{c}_j(\xi) = \tilde{c}_j(\mF, \xi)$,
  and sums of squares $\tilde{X}_\xi : \R^{36} \to \R$ such that
  \begin{equation}
    \label{eq:2d:qnl_hessian_straingrad}
    \< \Hac_\mF u, u \> = \sum_{j = 1}^3 \sum_{\xi \in \L}
    \tilde{c}_j(\xi) |D_j u(\xi)|^2 + \sum_{\xi \in \L} \tilde{X}_\xi(D^2u(\xi)).
  \end{equation}
  Moreover, the following identities hold: 
  \begin{align}
    \label{eq:2d:qnlhess_strgrad_cj}
    \tilde{c}_j(\xi) = c_j \quad &\text{except if both } 
    \xi, \xi+a_j \in \L^{(-1)} \cup \L^{(0)} \cup \L^{(1)},  \\
    \label{eq:2d:qnlhess_strgrad_Xc}
     \tilde{X}_\xi = 0 \quad &\text{for } \xi_2 > 0, \quad \text{and} \\
    \label{eq:2d:qnlhess_strgrad_Xa}
     \tilde{X}_\xi = X \quad &\text{for } \xi_2 < 0.
  \end{align}
\end{proposition}
\begin{proof}
  Applying the identities \eqref{eq:RULE-1}--\eqref{eq:RULE-3} to the
  hessian representation \eqref{eq:2d:qnl_hessian_defn} we obtain
  \eqref{eq:2d:qnl_hessian_straingrad}, and it only remains to prove
  \eqref{eq:2d:qnlhess_strgrad_cj}--\eqref{eq:2d:qnlhess_strgrad_Xa}.

  The identities \eqref{eq:2d:qnlhess_strgrad_Xc} and
  \eqref{eq:2d:qnlhess_strgrad_Xa} simply follow from the fact that
  the operations \eqref{eq:RULE-1}--\eqref{eq:RULE-3} only create
  strain gradient terms associated with the centre atom $\xi$. 

  The remaining property \eqref{eq:2d:qnlhess_strgrad_cj} can be
  obtained by understanding which bond coefficients $c_i(\eta)$ are
  ``influenced'' by the operations
  \eqref{eq:RULE-1}--\eqref{eq:RULE-3} applied with a given centre
  atom $\xi$. These are depicted in Figure \ref{fig:rules} and after
  combining the graphs for the three identities and rotating them, we
  see that a lattice site $\xi$ only influences the coefficients
  $c_i(\eta)$ corresponding to the twelve bonds $D_j u(\xi), j = 1,
  \dots, 6$ and $D_{j+2} u(\xi+a_j), j = 1, \dots, 6$; cf. Figure
  \ref{fig:influence} (a). From this, it follows that a given
  coefficient $c_i(\eta)$ is influenced only by the four nodes of the
  two neighbouring triangles; cf. Figure \ref{fig:influence}
  (b). Thus, only the bonds depicted in Figure \ref{fig:influence} (c)
  are affected by the modified site potentials, which are precisely
  those bonds contained in the strip $\{ x \in \R^2 \sep - \sqrt{3}/2
  \leq x_2 \leq \sqrt{3}/2 \}$.
\end{proof}

Although we have always restricted our presentation to the flat
interface situation, all results up to this point are {\em
  generic}. That is, they can be generalised to interfaces with
corners and even to long range interactions. 

In the next result, where we provide some characterisation of the
coefficients $\tilde{c}_j(\xi)$ in the interface region, we exploit
tangential translation invariance.

\begin{lemma}
  \label{th:2dqnl:consequence_of_forcecons}
  Suppose that the modified site-energies are tangentially translation
  invariant, i.e., $\tilde{V}_\xi = \tilde{V}_{\xi+a_1}$ for all $\xi
  \in \L^{(0)}$. Then the coefficients in the strain gradient
  representation \eqref{eq:2d:qnl_hessian_straingrad} satisfy
  \begin{equation}
    \label{eq:2dqnl:j2_j3_exact}
    \tilde{c}_j(\xi) = c_j \qquad \text{for all } \xi \in \L, \quad j = 2, 3.
  \end{equation}

  Moreover, for $j = 1$ and $\xi \in \L^{(m)}, m = -1, 0, 1,$ we have
  $\tilde{c}_1(\xi) = \tilde{c}_1^{(m)}$ (tangential translation
  invariance) and
  \begin{equation}
    \label{eq:2dqnl:sum_tilc1_eq_3c1}
    \sum_{m = -1}^1 \tilde{c}_1^{(m)} = 3 c_1.
  \end{equation}
\end{lemma}
\begin{proof}
  {\it 1. Properties of $\tilde{c}_2, \tilde{c}_3$: } By the same
  argument as in the 1D case (cf. Lemma \ref{th:qnl:sgrad_c0_eq_ddW})
  we can prove that
  \begin{equation}
    \label{eq:2dqnl:no_gf_hess}
    \< \Hac_\mF \mG x, u \> = 0 \qquad \forall u \in \Usz.
  \end{equation}
  We fix $\xi \in \L$ and test \eqref{eq:2dqnl:no_gf_hess} with
  $u(\eta) := \delta_{\xi,\eta}$ (i.e., a ``hat-function'') to obtain
  \begin{displaymath}
    \sum_{j = 1}^3 \tilde{c}_j(\xi) (-\mG \cdot a_j) + \sum_{j = 1}^3 \tilde{c}_j(\xi-a_j)
    (\mG \cdot a_j) = 0.
  \end{displaymath}
  If we define $c_{j+3}(\xi) := c_{j}(\xi-a_{j})$ for $j = 1, 2, 3$,
  then this can equivalently be stated as
  \begin{displaymath}
    - \mG \cdot \sum_{j = 1}^6 \tilde{c}_j(\xi) a_j = 0.
  \end{displaymath}
  Since this must hold for all $\mG \in \R^2$, we deduce that
  \begin{equation}
    \label{eq:2dqnl_consforccons:20}
    \sum_{j = 1}^6 \tilde{c}_j(\xi) a_j = 0 \qquad \forall \xi \in \L.
  \end{equation}
  Using that fact that $a_{j+3} = - a_j$ and $a_1 + a_3 + a_5 = 0$, we
  deduce that \eqref{eq:2dqnl_consforccons:20} is equivalent to
  \begin{equation}
    \label{eq:2dqnl_consforccons:20a}
    \tilde{c}_1(\xi) - \tilde{c}_4(\xi) = \tilde{c}_3(\xi)-\tilde{c}_6(\xi) = \tilde{c}_5(\xi)-\tilde{c}_2(\xi) \qquad
    \forall \xi \in \L.
  \end{equation}
  
  We now test \eqref{eq:2dqnl_consforccons:20a} with $\xi \in
  \L^{(1)}$. Due to the translation invariance of the modified
  site-energies it follows that $\tilde{c}_1(\xi) = \tilde{c}_4(\xi)$.
  Moreover, $\tilde{c}_3(\xi) = c_3$ and $\tilde{c}_2(\xi) = c_2$,
  which implies that
  \begin{displaymath}
    0 = c_3 - \tilde{c}_6(\xi) = \tilde{c}_5(\xi) - c_2.
  \end{displaymath}
  This implies \eqref{eq:2dqnl:j2_j3_exact} for $\xi \in
  \L^{(0)}$. Analogously, testing \eqref{eq:2dqnl_consforccons:20a} with $\xi \in
  \L^{(-1)}$ gives \eqref{eq:2dqnl:j2_j3_exact} for $\xi \in
  \L^{(-1)}$.

  {\it Properties of $\tilde{c}_1$: } We are left to establish the
  statements concerning the coefficients $\tilde{c}_1$. Due to
  translation invariance of the site potential it immediately follows
  that $\tilde{c}_j(\xi) = \tilde{c}_j(\xi+a_1)$, hence we can write
  $\tilde{c}_j(\xi) = \tilde{c}_j^{(m)}$ for $\xi \in \L^{(m)}$, $m =
  -1, 0, 1$.

  Finally, \eqref{eq:2dqnl:sum_tilc1_eq_3c1} is a consequence of the
  energy consistency \eqref{eq:2d:qnl_energy_cons}.  If we allowed
  noncompact test functions (as, e.g., in a periodic setting), then we
  could take the second variation of $\Ea(\mF x)=\Eqnl(\mF x)$ along
  the displacement $u=\mG x$ and obtain $\<\Ha_\mF \mG x, \mG x\> =
  \<\Hqnl_\mF \mG x, \mG x\>$ which would imply
  \eqref{eq:2dqnl:sum_tilc1_eq_3c1}.  However, in our case $\mG
  x\notin \Usz$, which makes the proof of
  \eqref{eq:2dqnl:sum_tilc1_eq_3c1} more involved.
  
    We start with noticing that the energy consistency implies
  \[
  \sum_{i,j=1}^6 (\tilde{V}_{i,j}-V_{i,j}) \, D_i u(\xi) D_j u(\xi)
  =0
  \]
  for $u=\mG x$ and some $\xi\in\Lambda^{(0)}$. We then rewrite this using the rules \eqref{eq:RULE-1}--\eqref{eq:RULE-3} as
  \[
  \sum_{i=1}^3 \sum_{\substack{\rho\in\Lambda \\ \rho,\rho+a_i \in \Rg\cup\{0\}}} (\tilde{c}_{i,\rho} - c_{i,\rho} ) |D_i u(\xi+\rho)|^2
  + \tilde{X}(D^2 u(\xi))-X(D^2 u(\xi))
  =0
  \]
  with some $\tilde{c}_{i,\rho}$ and $c_{i,\rho}$.
  Next, we substitute $u=\mG x$ and use $D^2(\mG x)=0$:
  \begin{equation}\label{eq:2dqnl_consforccons:30}
  \sum_{i=1}^3 \sum_{\substack{\rho\in\Lambda \\ \rho,\rho+a_i \in \Rg\cup\{0\}}} (\tilde{c}_{i,\rho} - c_{i,\rho} ) |\mG a_i|^2
  =0
  .
  \end{equation}
  It remains to notice that, since $\tilde{c}_i^{(m)}$ and $c_i$ were constructed using the same rules as $\tilde{c}_{i,\rho}$ and $c_{i,\rho}$, we have
  \[
  \sum_{\substack{\rho\in\Lambda \\ \rho,\rho+a_i \in \Rg\cup\{0\}}} (\tilde{c}_{i,\rho} - c_{i,\rho} )
  =
  \sum_{m=-1}^1 (\tilde{c}_i^{(m)}-c_i)
  \qquad (i=1,2,3).
  \]
  Substituting this into \eqref{eq:2dqnl_consforccons:30} and using that $\tilde{c}^{(m)}_i=c^{(m)}_i$ for $i=2,3$, we get
  \[
  \sum_{m=-1}^1 (\tilde{c}_1^{(m)}-c_1) |\mG a_1|^2
  =0
  \]
  for all $\mG$, which immediately implies \eqref{eq:2dqnl:sum_tilc1_eq_3c1}.
\end{proof}

We see that the key difference, between 1D and 2D, for the stability
of homogeneous deformations is that the $|D_ju|^2$ coefficients in the
1D case are identical to those in the Cauchy--Born model for
force-consistent a/c couplings, while this need not be the case in 2D.
As a first step to showing that this can lead to an instability in
  2D, we establish another representation of $\Hqnl_\mF$.

\begin{lemma}
  \label{th:2d:better_qnl_hessrep}
  Under the conditions of Lemma
  \ref{th:2dqnl:consequence_of_forcecons}, we have
  \begin{equation}
    \label{eq:2d:better_qnl_hessrep}
    \< \Hqnl_\mF u, u \> = \< \Hc_\mF u, u \> + 2
    \b(\tilde{c}^{(1)}_1- \tilde{c}^{(-1)}_1\b) \< K_0 u, u \> + 
    \sum_{\xi\in\L} \hat{X}_\xi(D^2 u(\xi)),
  \end{equation}
  where $\hat{X}_\xi$ are quadratic forms of $D^2 u$ (not
    necessarily sums of squares), with $\hat{X}_\xi = 0$ for $\xi \in
    \L^\c$, and
  \begin{align*}
    \< K_0 u, u \> &:= 
    \sum_{\xi \in \L^{(0)}} D_2 D_1 u(\xi) D_1 u(\xi).
    %
  \end{align*}
\end{lemma}
\begin{proof}

\begingroup \allowdisplaybreaks
  From Lemma \ref{th:2dqnl:consequence_of_forcecons} we have
  \begin{align*}
\notag
  & \hspace{-1cm} 
  \< \Hqnl_\mF u, u \>  - \< \Hc_\mF u, u \> - \sum_{\xi \in \L} \tilde{X}_\xi(D^2u(\xi))
  \\=~&
    \b( \tilde{c}^{(1)}_1 - c_1 \b)
    \sum_{\xi \in \L^{(0)}}
    \b( |D_1 u(\xi+a_2)|^2 - |D_1
    u(\xi)|^2 \b)
	\\~&+
	\b( \tilde{c}^{(-1)}_1 - c_1 \b)
	\sum_{\xi \in \L^{(0)}}
    \b( |D_1 u(\xi+a_5)|^2 - |D_1
    u(\xi)|^2 \b)
  \\=~&
    \b( \tilde{c}^{(1)}_1 - c_1 \b)
    \sum_{\xi \in \L^{(0)}} \b( D_1 u(\xi+a_2) - D_1 u(\xi) \b) \b(
    D_1 u(\xi+a_2) + D_1 u(\xi) \b)
	\\~&+
	\b( \tilde{c}^{(-1)}_1 - c_1 \b)
    \sum_{\xi \in \L^{(0)}} \b( D_1 u(\xi+a_5) - D_1 u(\xi) \b) \b(
    D_1 u(\xi+a_5) + D_1 u(\xi) \b)
  \\=~&
    \b( \tilde{c}^{(1)}_1 - c_1 \b)
    \sum_{\xi \in \L^{(0)}} D_2 D_1 u(\xi)\,(2 D_1 u(\xi) + D_2 D_1 u(\xi))
	\\~&+
    \b( \tilde{c}^{(-1)}_1 - c_1 \b)
    \sum_{\xi \in \L^{(0)}} D_5 D_1 u(\xi)\,(2 D_1 u(\xi) + D_5 D_1 u(\xi))
  \\=~&
    (\tilde{c}^{(1)}_1 - \tilde{c}^{(-1)}_1)
    \sum_{\xi \in \L^{(0)}}
    D_2D_1 u(\xi) D_1 u(\xi)
  \\&
    -(\tilde{c}^{(-1)}_1 - c_1) \sum_{\xi
      \in \L^{(0)}} D_5D_2D_1 u(\xi) D_1u(\xi)  +  \dots,
  \end{align*}%
\endgroup
  where ``$\dots$'' stands for some sum of squares of $D^2
  u(\xi)$.

  Summation by parts,
  \begin{displaymath}
    \sum_{\xi \in \L^{(0)}}
    D_5D_2D_1 u(\xi) D_1 u(\xi) = - \sum_{\xi \in \L^{(0)}} D_5D_2
    u(\xi) D_4D_1u(\xi)
  \end{displaymath}
  completes the proof.
\end{proof}

\subsection{Non-existence of a universally stable method in 2D}
\label{sec:univ_stab_2d}
Lemma \ref{th:2d:better_qnl_hessrep} suggests that, unless
  $\tilde{c}^{(1)}_1 - \tilde{c}^{(-1)}_1=0$, there is a discrepancy
  between $\Hqnl_\mF$ and $\Hc_\mF$ that is not a quadratic in $D^2 u$
  (and, as will be shown in \S \ref{sec:2d:qnl_hess_instability},
  unavoidably leads to an instability).  We next establish that in
  fact $\tilde{c}^{(1)}_1 - \tilde{c}^{(-1)}_1 \not\equiv 0$ for a
  large family of a/c schemes, which not only includes examples from
  \S \ref{sec:2d:def_qnl} but also all geometric reconstruction type
  variants \cite{E:2006,PRE-ac.2dcorners}.  We also present in \S
  \ref{sec:2d:qnl_hess_explicit} explicit calculations for the three
  methods from \S \ref{sec:2d:def_qnl}.

\begin{proposition}\label{prop:univ_stab_2d:gen_grac}
  Consider the following generalization of the geometric
  reconstruction a/c (GRAC) method \cite{PRE-ac.2dcorners}:
  \begin{equation}\label{eq:univ_stab_2d:gen_grac_2}
    \tilde{V}(\bfg)= \sum_{\ell=1}^L w_\ell V(\mC_\ell \bfg),
\end{equation}
where $w_\ell \in \R$, $w_\ell\ne0$, and $\mC_\ell \in \R^{6\times 6}$
($\ell=1,\ldots,L$).
Assume that it satisfies the force and energy consistency conditions
\eqref{eq:2d:qnl_nogf}, \eqref{eq:2d:qnl_energy_cons}.  Further,
assume hexagonal symmetry \eqref{eq:2d:hex_symm} of $V$, with
$\alpha_2=\alpha_3=0$. 

Then, there exist $p_0, p_1 \in \R$ (depending on $w_\ell, \mC_\ell$)
such that $p_0 - p_1 = 1$ and
\begin{displaymath}
  \tilde{c}_1^{(1)} - \tilde{c}_1^{(-1)} = p_0 \alpha_0 + p_1 \alpha_1.
\end{displaymath}

In particular, there exists no choice of method parameters $w_\ell,
\mC_\ell$, such that $\tilde{c}_1^{(1)}-\tilde{c}_1^{(-1)} = 0$ for
all parameters $(\alpha_0, \alpha_1)$.
\end{proposition}
\begin{proof}
{\it Step 1 (reduction to a GRAC).}
Consider a method with interface site potential
\begin{equation}\label{eq:univ_stab_2d:gen_grac}
\ttilde{V}(\bfg):= V(\mB \bfg),
\end{equation}
where $\mB := \sum_{\ell=1}^L w_\ell \mC_\ell$.  We show that it is
energy and force consistent and moreover
$\<\ddel\tilde{V}(\mF\Rg)u,u\> - \<\ddel\ttilde{V}(\mF\Rg)u,u\>$ is a
sum of squares of $D^2 u$ (and hence
$\tilde{c}_1^{(1)}-\tilde{c}_1^{(-1)}$ is the same for both methods).

Indeed, substituting $V(\bfg) = v_0 + {\bm f}\cdot \bfg$ into the energy consistency condition \eqref{eq:2d:qnl_energy_cons} yields
\[
v_0 \bigg(\sum_{\ell=1}^L w_\ell - 1\bigg) + {\bm f}\cdot (\mB \mF \Rg - \mF \Rg)
=0
\qquad\forall v_0\in\R,\quad \forall {\bm f}\in\R^6\quad \forall \mF\in\R^2.
\]
Hence we get $\sum_{\ell=1}^L w_\ell=1$ and $\mB \mF \Rg = \mF \Rg$ for all $\mF$. These identities make it straightforward to verify the energy and force consistency of \eqref{eq:univ_stab_2d:gen_grac}, given the energy and force consistency of \eqref{eq:univ_stab_2d:gen_grac_2}.

Finally, to show that $\b\<\b(\ddel\tilde{V}(\mF\Rg)-\ddel\ttilde{V}(\mF\Rg)\b)u,u\b\>$ is a sum of squares of $D^2 u$, compute
\begin{equation}\label{eq:univ_stab_2d:ddelV}
\ddel\tilde{V}(\mF\Rg) = \sum_{\ell=1}^L w_\ell \mC_\ell^\transpose \mH \mC_\ell,
\end{equation}
where $\mH := \ddel V(\mF\Rg) \in\R^{6\times 6}$ is the hessian of $V$.
We apply the identity
\[
  w_\ell \mC_\ell^\transpose \mH \mC_\ell
+ w_j    \mC_j^\transpose \mH \mC_j
= (w_\ell+w_j) \big( \smfrac{w_\ell \mC_\ell + w_j \mC_j}{w_\ell+w_j} \big)^\transpose \mH \big( \smfrac{w_\ell \mC_\ell + w_j \mC_j}{w_\ell+w_j} \big)
+ \smfrac{w_j w_\ell}{w_\ell+w_j} (\mC_\ell-\mC_j)^\transpose \mH (\mC_\ell-\mC_j)
\]
to \eqref{eq:univ_stab_2d:ddelV} $L-1$ times,
noticing that the finite difference operator $(\mC_\ell-\mC_j) Du$ is zero on all affine functions and hence can be represented as a sum of second differences.
As a result, we express $\ddel\tilde{V}(\mF\Rg)$ as $\ddel\ttilde{V}(\mF\Rg)$ plus squares of second differences.

{\it Step 2 (proof for a GRAC).}  It is now sufficient to establish
this proposition for a simpler method
\eqref{eq:univ_stab_2d:gen_grac}.  Using the rules
\eqref{eq:RULE-1}--\eqref{eq:RULE-3}, we can express
\[
\tilde{c}_1^{(1)}-\tilde{c}_1^{(-1)} =
-\smfrac13 (\alpha_0+4 \alpha_1)
+\ttilde{V}_{1,3}+\ttilde{V}_{2,3}+\ttilde{V}_{2,4}
-\ttilde{V}_{4,6}-\ttilde{V}_{5,6}-\ttilde{V}_{5,1}
,
\]
which implies linearity of $\tilde{c}_1^{(1)}-\tilde{c}_1^{(-1)}$ with
respect to $\alpha_0$ and $\alpha_1$, that is, $\tilde{c}_1^{(1)} -
\tilde{c}_1^{(-1)} = p_0 \alpha_0 + p_1 \alpha_1$.

To see that $p_0 - p_1=1$, choose coefficients $\alpha_0 = 1$ and
$\alpha_1 = -1$, i.e., so that $p_0 - p_1 =
\tilde{c}_1^{(1)}-\tilde{c}_1^{(-1)}$).  In this case the hessian of
$V$ is given by
\begin{equation}\label{eq:univ_stab_2d:H}
\mH =
\ddel V(\mF\Rg)
= \left(\begin{smallmatrix}
-1 & 1 & 0 & 0 & 0 & 1 \\
1 & -1 & 1 & 0 & 0 & 0 \\
0 & 1 & -1 & 1 & 0 & 0 \\
0 & 0 & 1 & -1 & 1 & 0 \\
0 & 0 & 0 & 1 & -1 & 1 \\
1 & 0 & 0 & 0 & 1 & -1
\end{smallmatrix}\right)
\end{equation}
and $\ddel \ttilde{V} = \mB^\transpose \mH \mB$.
Next, denote the column-vectors of $\mB$ as $b_i\in\R^6$ and hence express
\[
\tilde{c}_1^{(1)}-\tilde{c}_1^{(-1)} = 1
+
 b_1^\transpose \mH b_3
+b_2^\transpose \mH b_3
+b_2^\transpose \mH b_4
-b_4^\transpose \mH b_6
-b_5^\transpose \mH b_6
-b_5^\transpose \mH b_1
\]
(here we used $\smfrac13 (\alpha_0+4 \alpha_1)=-1$).

Energy consistency \eqref{eq:2d:qnl_energy_cons} implies $\sum_{i=1}^6
(\mF a_i) b_i = (\mF a_j)_{j=1}^6$ (we refer to
\cite{PRE-ac.2dcorners} for details).  Using this identity with $\mF =
\smfrac23 (a_6+a_1)^\transpose$ and with $\mF = \smfrac23
(a_2+a_3)^\transpose$ allows to express
\[
b_1 = b_3+b_4-b_6 + (1,0,-1,-1,0,1)^\transpose
\quad\text{and}\quad
b_2 = b_5+b_6-b_3 + (0,1,1,0,-1,-1)^\transpose
.
\]
Substituting these expressions into $\tilde{c}_1^{(1)}-\tilde{c}_1^{(-1)}$ yields, after all cancellations,
\[
\tilde{c}_1^{(1)}-\tilde{c}_1^{(-1)} = 1 + (1,0,-1,-1,0,1)\, \mH (b_3-b_5) + (0,1,1,0,-1,-1)\, \mH (b_3+b_4),
\]
which equals identically $1$ once \eqref{eq:univ_stab_2d:H} is used.
\end{proof}

\begin{remark}
  Suppose that (in some practical problem) $\mF=\mF_0$ is fixed and
  given {\it a priori}.

  1. One can then consider energy consistent methods with ghost-force
  correction, such as \cite{Ortiz:1995a} (i.e., methods that satisfy
  \eqref{eq:2d:qnl_nogf} only for $\mF=\mF_0$).  Since we do not use
  explicitly force consistency \eqref{eq:2d:qnl_nogf} in the proof,
  Proposition \ref{prop:univ_stab_2d:gen_grac} would also be valid for
  such methods.

  2. Nevertheless, it is possible to precompute
  $\tilde{c}_1^{(1)}-\tilde{c}_1^{(-1)}$ and subtract the term
  $\smfrac12 (\tilde{c}_1^{(1)}-\tilde{c}_1^{(-1)})
  ((g_2-g_3)^2-(g_5-g_6)^2)$ from $\tilde{V}(\bfg)$, thus correcting
  the error in $\tilde{c}_i(\xi)$.  We will, however, not pursue in
  this work the questions of applicability of such correction beyond
  the nearest-neighbour plane-interface scalar setting.
\end{remark}

\subsection{Instability}
\label{sec:2d:qnl_hess_instability}
It is fairly staightforward to see that $\gamma(K_0)=\gamma(-K_0)<0$ (cf.\ \eqref{eq:2d:better_qnl_hessrep}).
In this section we will show that the strain gradient correction (third group) in \eqref{eq:2d:better_qnl_hessrep}) cannot improve this indefiniteness of $K_0$, which will immediately imply the instability result (Corollary \ref{th:2d:cor_instab}).

The strain gradient correction is clearly bounded by an operator of
the form
\begin{equation}
  \label{eq:2d:defn_S}
  \< S u, u \> := \sum_{\xi \in \L^{(0)}} | D^2 u(\xi) |^2,
\end{equation}
that is, $|\hat{X}_\xi(D^2 u)| \leq C | D^2 u(\xi) |^2$. We
therefore consider generic operators of the form 
\begin{equation}
  \label{eq:2d:defn_Kkappa}
  \< K_\kappa u, u \> := \< K_0 u, u \> + \kappa \< Su, u \>.
\end{equation}
We will show that $K_\kappa$ is indefinite, independent of the choice
of $\kappa$, and hence independent of the form the strain gradient
correction $\hat{X}_\xi$ takes. Note that this result is also a
preparation for our analysis of the 2D analogue of the stabilisation
\eqref{eq:1dqnl:defn_Estab}.

\begin{lemma}
  \label{th:2d:qnl_hess_instab}
  There exists a constant $c > 0$ such that
  \begin{equation}
    \inf_{\substack{ u \in \Usz \\ \| D u \|_{\ell^2} = 1}} \< K_\kappa u,
    u\> =: \lambda_\kappa \leq - \frac{c}{(\kappa+1)^2}.
  \end{equation}
\end{lemma}
\begin{proof}
  To obtain this bound, we make a separation of variables ansatz,
  \begin{displaymath}
    u(\xi) = u(m a_1 + n a_2) = \alpha_m \beta_n,
  \end{displaymath}
  and we define $\alpha_m' := \alpha_{m+1} - \alpha_m, \alpha_m'' :=
  \alpha_{m+1}-2\alpha_m + \alpha_{m-1}$, and analogous notation
  for~$\beta$.

  Next, let $A, B \in C^\infty(\R)$ be compactly supported with $B(0)
  = 1$, and $B'(0) = 1, B''(0) = 0$.

  Let $N \in \N$ and define $\alpha_m := A(m/N)$ and $\beta_n :=
  B(n/N)$, then simple scaling arguments show that, for $N \geq
  N_0$ (sufficiently large),
  \begin{align*}
    & \beta_0' \approx N^{-1}, \qquad |\beta''_0| \lesssim N^{-4}, \\
    & \| \alpha \|_{\ell^2}^2 \approx N \| A\|_{L^2}^2, \quad
    \| \alpha' \|_{\ell^2}^2 \approx N^{-1} \| A' \|_{L^2}^2, \quad
    \| \alpha'' \|_{\ell^2}^2 \approx N^{-3} \| A'' \|_{L^2}^2,
  \end{align*}
  and analogous bounds for $\beta$ in terms of $B$.  Here and for the
  remainder of the proof, ``$\approx$'' indicates upper and lower
  bounds up to constants that are independent of $\kappa, N$.

  With these definitions and derived properties we obtain (after some
  work) that
  \begin{align*}
    \< K_0 u, u \> &= - \beta'_0 \| \alpha' \|_{\ell^2}^2 \approx - N^{-2},  \\
    \< S u, u \> &\approx |\beta_0|^2 \| \alpha'' \|_{\ell^2}^2 + |\beta_0'|^2 \|
    \alpha' \|_{\ell^2}^2 \approx N^{-3}, \quad
    \text{and}
    \\
    \| D u \|_{\ell^2(\L)}^2 &\approx \| \alpha' \|_{\ell^2}^2 \| \beta
    \|_{\ell^2}^2 + \| \alpha \|_{\ell^2}^2 \| \beta' \|_{\ell^2}^2 \approx 1,
  \end{align*}
  that is,
  \begin{displaymath}
    \lambda_\kappa \leq \frac{\< K_\kappa u, u \>}{\|D u\|_{\ell^2}^2} \leq - C_1 N^{-2} +
    C_2 \kappa N^{-3},
  \end{displaymath}
  where $C_1, C_2 > 0$ depend on $A, B$ but are independent
  of $\kappa$ and of $N$ (provided $N \geq N_0$).

  If $\kappa = \frac{2 C_1}{3 C_2} N_0 =: \kappa_0$, choosing $N = N_0$,
  we obtain $\lambda(\kappa) \leq - \frac{C_1}{3} N_0^{-2}$. 

  For $\kappa > \kappa_0$, let $N = \frac{3 C_2}{2 C_1} \kappa$, then
  $N \geq N_0$ and this implies that $\lambda_\kappa \leq -
  \frac{4}{27} C_1^3 C_2^{-2} \kappa^{-2}$.  This completes the proof.
\end{proof}

We can deduce the following instability result.  Ignoring the
(non-trivial) technical conditions, the result can be read as follows:
{\it if the error in the coefficients $\tilde{c}^{(m)}_1$ does not
  cancel at a critical strain $\mG$ (where $\Ha_\mG$ becomes unstable)
  then the QNL method will necessarily predict a reduced critical
  strain with an $O(1)$ error. That is, the critical deformation $\mG$
  {\em cannot} be predicted with arbitrarily high accuracy by the QNL
  method. } See \S~\ref{sec:err_crit_strains} for further discussion
of this issue.

\begin{corollary}
  \label{th:2d:cor_instab}
  Consider the hexagonally symmetric case \eqref{eq:2d:hex_symm} with
  $\alpha_2 = \alpha_3 = 0$. Suppose, moreover, that
  \begin{itemize}
  \item[(i)] $\ca(\mO) = 0$, and that
  \item[(ii)] $\tilde{c}^{(1)}_1(\mO) - \tilde{c}^{(-1)}_1(\mO) \neq
    0$.
  \end{itemize}
  Then, $\cqnl(0) < 0$. 

  In particular, $\cqnl(\mG) < 0$ for sufficiently small $|\mG|$.
\end{corollary}
\begin{proof}
  The symmetry assumptions and (i) imply that $\Ha_\mO = \Hc_{\mO} =
  0$. Therefore, applying \eqref{eq:2d:better_qnl_hessrep} we obtain
  that
  \begin{displaymath}
    \< \Hqnl_{\mO} u, u \> \leq 2 (\tilde{c}_1^{(1)} -
    \tilde{c}^{(-1)}_1) \< K_0 u, u \> + \kappa \< S u, u \>
  \end{displaymath}
  for some $\kappa > 0$. Lemma \ref{th:2d:qnl_hess_instab} implies
  that $\cqnl(\mO) < 0$.
\end{proof}

  \begin{remark}
    1.  In the above corollary, (i) is an assumptions on $V$, whereas
    (ii) is the assumption on an a/c scheme.  We showed in \S
    \ref{prop:univ_stab_2d:gen_grac} that (ii) is generically
    satisfied.

    2. Our numerical investigations (\S\S \ref{sec:stab_regions} and
    \ref{sec:num2d}) indicate that similar results hold for more
    general $V$ and $\mG$, i.e., not necessarily satisfying $\mG =
    \mO$ and the simplifying condition $\alpha_2 = \alpha_3 = 0$. It
    does not, however, appear straightforward to extend our analysis.
  \end{remark}

\subsection{Stabilising the 2D QNL Method}
\label{sec:stab2d}
To conclude our analysis of the 2D case, we explore the issue of
stabilisation. Let $S$ be given by \eqref{eq:2d:defn_S} then we define
the stabilised QNL energy functional
\begin{equation}
  \label{eq:defn_Estab:1}
  \Estab(y) := \Eqnl(u) + \kappa \< S u, u \>,
\end{equation}
for some $\kappa \geq 0$.

A consequence of Corollary \ref{th:2d:cor_instab} is that (under its
technical conditions), for any fixed $\kappa$, if $\gamma(\Ha_\mG) =
0$ then $\gamma(H^{\rm stab}_\mG) < 0$, that is, the critical
deformation $\mG$ can still not predicted with arbitrarily high
accuracy. However, there is some hope that the error can be controlled
in terms of $\kappa$. To that end, we first show that Lemma
\ref{th:2d:qnl_hess_instab} is in fact sharp.

\begin{theorem}
  \label{th:sharp_stab_result}
  Let $K_\kappa$ and $\lambda_\kappa$ be defined by
  \eqref{eq:2d:defn_Kkappa}, then there exist constants $c_1, c_2 > 0$
  such that
  \begin{equation}
    \label{eq:2d:sharp_stab_result}
    - \frac{c_1}{(\kappa+1)^2} \leq \lambda_\kappa \leq -
    \frac{c_2}{(\kappa+1)^2} \qquad \forall \kappa \geq 0.
  \end{equation}
\end{theorem}
\begin{proof}
  The upper bound has already been established in Lemma
  \ref{th:2d:qnl_hess_instab}, hence we only have to show that it is
  sharp. For $\kappa \leq 1$, the lower bound is obvious, hence we
  assume that $\kappa > 1$.

  We first (crudely) estimate
  \begin{align*}
    \< K_\kappa u, u \> &\geq \sum_{\xi \in \L^{(0)}} \B( D_2 D_1
    u(\xi) D_1 u(\xi) + \kappa \sum_{i = 1}^6 |D_i D_1 u(\xi)|^2
    \B)  \\
    &\geq \sum_{\xi \in \L^{(0)}} \B( - \smfrac{1}{4 \kappa} |D_1
    u(\xi)|^2 + \kappa |D_1^2 u(\xi)|^2 \B).
  \end{align*}
  
  If we can prove the trace inequality
  \begin{equation}
    \label{eq:2d:trace_ineq}
    \| D_1 u \|_{\ell^2(\L^{(0)})}^2 \leq C_1 \B( \kappa^2 \| D D_1 u
    \|_{\ell^2(\L^{(0)})}^2 + \kappa^{-1} \| Du \|_{\ell^2(\L)}^2 \B),
  \end{equation}
  for some constant $C_1$, which can equivalently be rewritten as
  \begin{displaymath}
    - \smfrac{1}{4\kappa} \b\|D_1 u \b\|_{\ell^2(\L^{(0)})}^2 + \kappa
    \b\| D D_1 u \b\|_{\ell^2(\L^{(0)})}^2 \geq - \smfrac{c_1}{\kappa^2}
    \| D u \|_{\ell^2(\L)}^2,
  \end{displaymath}
  then the stated result follows.

  {\it Proof of \eqref{eq:2d:trace_ineq}: } It turns out that
  \eqref{eq:2d:trace_ineq} is a consequence of the embedding
  $\dot{H}^1(\R^2) \to \dot{H}^{1/2}(\R)$. To make this precise we
  resort to Fourier analysis. Let
  \begin{displaymath}
    \hat{u}(k) := \sum_{\xi_1 \in \Z} u(\xi_1, 0) e^{i k \xi_1},
  \end{displaymath}
  then $\hat{u}$ is a periodic smooth function on $(-\pi, \pi)$ and
  the following bounds hold:
  \begin{align}
    \notag
    \| D_1 u \|_{\ell^2(\L^{(0)})}^2 &\approx \int_{-\pi}^\pi |k|^2
    |\hat{u}|^2 \,{\rm d}k, \\
    \notag
    \| D_1^2 u \|_{\ell^2(\L^{(0)})}^2 &\approx \int_{-\pi}^\pi |k|^4
    |\hat{u}|^2 \,{\rm d}k,  \qquad \text{ and } \\    
    \label{eq:trace_ineq_fourier}
    \| D u \|_{\ell^2(\L)}^2 &\gtrsim \int_{-\pi}^\pi |k| |\hat{u}|^2 \,{\rm d}k.
  \end{align}
  The first two bounds are completely standard. The bound
  \eqref{eq:trace_ineq_fourier} is a discrete variant of a standard
  trace inequality, and is established below.
	
  We thus deduce that, to prove \eqref{eq:2d:trace_ineq} it is
  sufficient to show that there exists $C_1'$ such that
  \begin{displaymath}
    k^2 \leq C_1' \b( \kappa^2 k^4 + \kappa^{-1} |k| \b) \qquad
    \forall k \in [-\pi, \pi].
  \end{displaymath}
  But, in fact, it is easy to see that $k^2 \leq \max(\kappa^2 k^4,
  \kappa^{-1} |k|)$, and hence \eqref{eq:2d:trace_ineq} follows.
  
\delfinal{  
  {\it Proof of \eqref{eq:trace_ineq_fourier}}.  It is sufficient to
  establish a similar bound for functions $u\in\Usz(\Z^2)$.  Introduce
  the full Fourier transform $\hhat{u}(k_1,k_2) = \hhat{u}(k) :=
  \sum_{\xi \in \Z^2} u(\xi) e^{i k\cdot \xi}$ and, using the
  Cauchy-Schwarz inequality, estimate
  \begin{align*}
  \int_{-\pi}^\pi |k_1| |\hat{u}(k_1)|^2 \,{\rm d}k_1
  =~&
  \int_{-\pi}^\pi |k_1| \bigg(\frac{1}{2\pi} \int_{-\pi}^\pi \hhat{u}(k_1,k_2) {\rm d}k_2\bigg)^2 \,{\rm d}k_1
  \\ \lesssim ~&
  \int_{-\pi}^\pi \bigg(\int_{-\pi}^\pi |k_1| (k_1^2+k_2^2)^{-1} {\rm d}k_2\bigg) \bigg(\int_{-\pi}^\pi (k_1^2+k_2^2) (\hhat{u}(k_1,k_2))^2 {\rm d}k_2\bigg) \,{\rm d}k_1
  \\ = ~&
  \int_{-\pi}^\pi 2 \arctan\b(\smfrac{\pi}{|k_1|}\b) \bigg(\int_{-\pi}^\pi (k_1^2+k_2^2) (\hhat{u}(k_1,k_2))^2 {\rm d}k_2\bigg) \,{\rm d}k_1
  \\ \lesssim ~&
  \|Du\|_{\ell^2(\Z^2)}^2.\qedhere
  \end{align*}
}
\end{proof}

We can now refine the discussion at the beginning of the section to
obtain the following result.

\begin{corollary}
  \label{th:2d:cor_optim_stab}
  Let $V$ have hexagonal symmetry \eqref{eq:2d:hex_symm}, $V_{i,i+2} =
  V_{i,i+3} \equiv 0$, and $\tilde{c}_1^{(1)}-\tilde{c}_1^{(-1)}\ne
  0$; then there exists constants $c_1, c_2 > 0$ such that
  \begin{displaymath}
    \gamma(\Ha_\mO) - \frac{c_1}{\kappa^2} \leq \gamma(\Hqnl_\mO +
    \kappa S) \leq
    \gamma(\Ha_\mO) - \frac{c_2}{\kappa^2}.
  \end{displaymath}
\end{corollary}
\begin{proof}
  The result is an immediate consequence of Theorem
  \ref{th:sharp_stab_result}.
\end{proof}

To explain the relevance of Corollary \ref{th:2d:cor_optim_stab}
consider the setting of \S~\ref{sec:err_crit_strains} and suppose, for
the sake of argument, that the result holds at the critical strain,
\begin{displaymath}
  \gamma(\Ha_{\mG(t_*)}) - \frac{c_1}{\kappa^2} \leq \gamma(\Hqnl_{\mG(t_*)} +
  \kappa S) \leq
  \gamma(\Ha_{\mG(t_*)}) - \frac{c_2}{\kappa^2}.
\end{displaymath}
It is then easy to see that the error in the critical strain will be
of the order
\begin{equation}
  \label{eq:2d:err_crit_strain}
  |t^\kappa_* - t_*| \approx \frac{1}{\kappa^2}.
\end{equation}

Therefore, if we wish to admit at most an $O(\eps)$ error in the
critical strain, then we must accordingly choose $\kappa =
O(\eps^{-1/2})$. Unfortunately, this causes a larger consistency error
of the stabilised QNL method, which may again cause a feedback to
cause a larger error in the critical strain. This effect requires
further investigation in future work that would also need to
incorporate inhomogeneous deformations.


\section{Numerical Tests} 

\subsection{Regions of stability}
\label{sec:stab_regions}
We have analytically established the instability and stabilization
results for the case when only nearest-neighbour bonds interact (i.e.,
assuming $\alpha_2=\alpha_3=0$).  In this subsection we will study
these issues in the general hexagonally symmetric case
\eqref{eq:2d:hex_symm}, admitting $\alpha_2, \alpha_3 \neq 0$.  The
above analytic results cannot be readily extended to this case since
$\Ha_\mF\ne \Hc_\mF$, hence will use a semi-numeric approach.

We start with a characterization of the stability of $\Ha_\mF$.

\begin{lemma}
$\Ha_\mF$ is stable if and only if
\begin{align*}
\beta_1 :=~& \alpha_0+\alpha_1-\alpha_2-\alpha_3 > 0
, \\
\beta_2 :=~& \alpha_0+\alpha_1+\alpha_2+\alpha_3 > 0
, \quad\text{and} \\
\beta_3 :=~& 2\alpha_0+2\alpha_1+4\alpha_2+\alpha_3 > 0
. \\
\end{align*}
\delfinal{\begin{proof}
 We first notice that if $\beta_1 \leq 0$ then $\gamma(\Ha_\mF) \leq
  \gamma(\Hc_\mF) = \frac{4}{\sqrt{3}} \beta_1 \leq 0$
  (cf.~\eqref{eq:Ha_coeffs_hexsym}).  The necessity of the other two
  conditions follows from the identities
\begin{align*}
\Ha_\mF u^{(2)} = 16 \beta_2 u^{(2)}
\quad\text{where}&\quad
u^{(2)}(\xi_1,\xi_2) := \exp\b(\smfrac{2\pi \i}{\sqrt{3}}\xi_2\b)
,\quad\text{and}
\\
\Ha_\mF u^{(3)} = 9 \beta_3 u^{(3)}
\quad\text{where}&\quad
u^{(3)}(\xi_1,\xi_2) := \exp\b(\smfrac{4\pi \i}{3}\xi_1\b),
\end{align*}
which can be verified by a direct calculation.

To prove sufficiency, we use the following representation:
\begingroup \allowdisplaybreaks
\begin{align*}
\Ha_\mF =~& \beta_1 B_1 + \beta_2 B_2 + \beta_3 B_3,
\quad\text{where} \\
\<B_1 u, u\> =~& \sum_{\xi \in \L} \sum_{i = 1}^6
	\b(
		|D_i u(\xi)|^2-
			\smfrac16 |D_{i+2}D_i u(\xi)|^2 - \smfrac16 |D_{i+3}D_i u(\xi)|
	\b)
	\\=~&
	\smfrac{1}{12} \sum_{\xi \in \L} \sum_{i = 1}^6
		|(D_{i+1}-D_{i-1}+D_{i+2}-D_{i-2}) u(\xi)|^2\\
\<B_2 u, u\> =~&
	\sum_{\xi \in \L} \sum_{i = 1}^6
	\b(
			\smfrac12 |D_{i+3}D_i u(\xi)|^2 - \smfrac12 |D_{i+2}D_i u(\xi)|^2
		\b) \\
	=~&
	\smfrac14 \sum_{\xi \in \L} \sum_{i = 1}^6
		|D_i (D_{i+2}-D_{i+4}) u(\xi)|^2
\\
\<B_3 u, u\> =~&
	\sum_{\xi \in \L} \sum_{i = 1}^6
	\b(
		\smfrac23 |D_{i+2}D_i u(\xi)|^2 - \smfrac13 |D_{i+3}D_i u(\xi)|^2
	\b) \\
	=~&
	\smfrac23 \sum_{\xi \in \L} \bigg|{\sum_{i=1}^6} (-1)^i D_i u(\xi)\bigg|^2
,
\end{align*}%
\endgroup
which can also be verified by a direct calculation.

Assume that $\beta_1$, $\beta_2$, $\beta_3$ are all positive.  Then
for a sufficiently small $\eps>0$,
\[
\Ha_\mF
=
2 \eps \|D\cdot\|_{\ell^2}^2 + (\beta_1-\eps) B_1 + (\beta_2-\eps) B_2 + (\beta_3-\eps) B_3
\]
is a sum of four positive semidefinite operators with $\gamma(\Ha_\mF)
\geq \frac{4}{\sqrt{3}} \eps >0$.
\end{proof}}
\end{lemma}

The above lemma states that the region of stability of $\Ha_\mF$ is
the first octant of the three-dimensional space of parameters
$(\beta_1,\beta_2,\beta_3)$.  We will thus study the extent to which
different a/c methods reproduce this exact stability region.  For the
ease of visualization, we restrict ourselves to a hyperplane
$\beta_1=\beta_3$ and map the stability region into a triangle
\[
\{(x,y) : x>0, y>0, x+2 y<1\}
\]
by letting $\beta_1 = \beta_3 = y/(1-x-2y)$ and $\beta_2 = x/(1-x-2y)$.

We compute the boundary of the stability region semi-analytically in the following way.
First, due to translational symmetry in $\xi_1$, it is sufficient to (formally) consider the test functions of the form $u(\xi_1,\xi_2) = e^{\i \xi_1 k_1} \bar{u}(\xi_2)$ where $k_1\in(-\pi,\pi)$ and $\bar{u}\in \Usz(\Z)$.\footnote{To rigorously justify this step, one would need to introduce a cut-off to these test functions to ensure that they belong to $\Usz(\L)$.}
This reduces the problem to testing for positive definiteness of five-diagonal symmetric operators depending on $k_1\in(-\pi,\pi)$.
Because the operator coefficients on different diagonals for $\xi_2<-1$ and for $\xi_2>1$ are constant, these operators can be inverted analytically.
Hence, we used {\it Mathematica} to analytically check whether there are negative eigenvalues of these operators and used a numerical procedure of minimizing the smallest eigenvalue over $k_1\in(-\pi,\pi)$.

\begin{figure}
  \begin{center}
    \hspace{-4mm} 
    \includegraphics{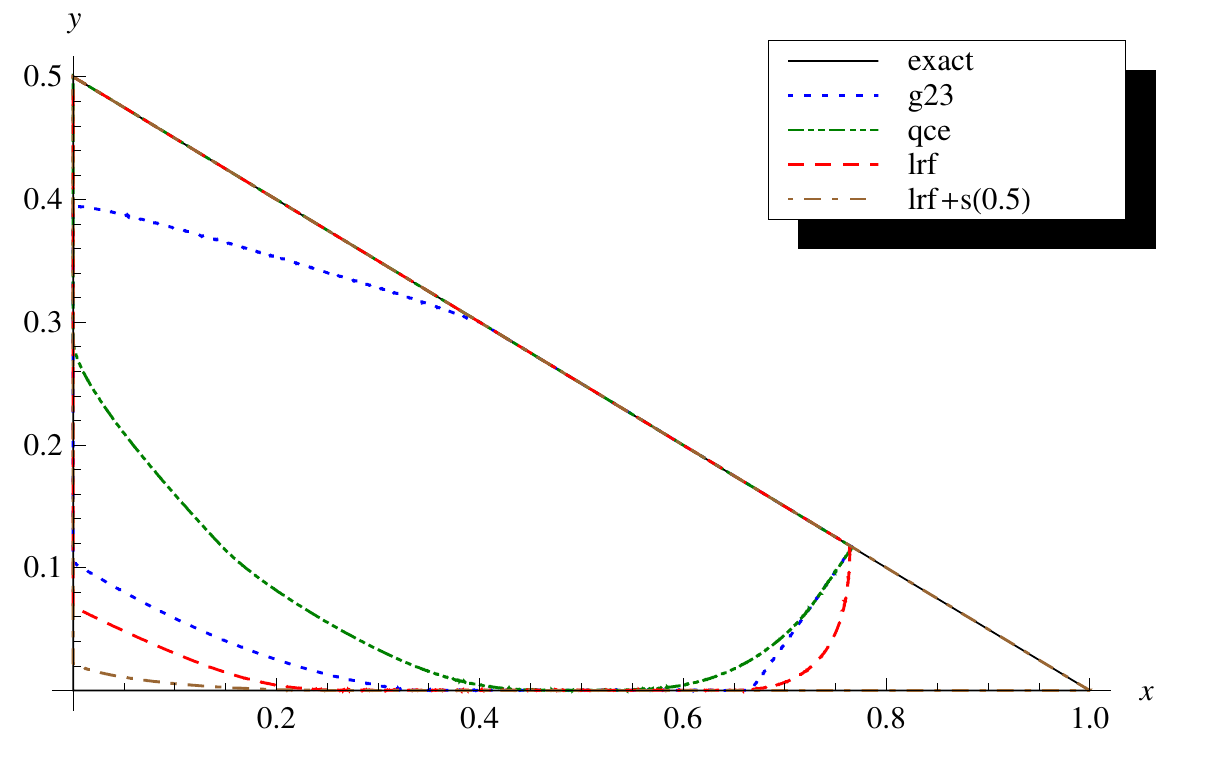}
    \hspace{-4mm}

    \caption{ \label{fig:stab_regions} Stability regions the hexagonally symmetric case, as described in \S~\ref{sec:stab_regions}.
    The exact (atomistic) stability region is the triangle and the stability regions of the a/c methods are proper subsets of it.
    The last method, lrf+s(0.5), is the stabilized coupling \eqref{eq:stab_reg:slrf} with $\kappa=0.5$.
    }
  \end{center}
\end{figure}

The regions of stability of different a/c methods are plotted in Figure \ref{fig:stab_regions}.
We observe that none of the methods reproduce the exact stability region, which is consistent with the results in the case $\alpha_2=\alpha_3=0$ (cf.~Corollary \ref{th:2d:cor_instab}).
Also, we see that the stabilized local reflection method
\begin{equation}\label{eq:stab_reg:slrf}
\<H_\mF^{\rm lrf+s(\kappa)} u, u\> := 
\<H_\mF^{\lref} u, u\> + \kappa (|\alpha_0|+|\alpha_1|+|\alpha_2|+|\alpha_3|) \,\smfrac16 \sum_{i=1}^6 |D_i D_{i+2} u|^2
\end{equation}
with $\kappa=0.5$ has an improved (but not exact) stability region.

\subsection{Critical eigenmodes}
\label{sec:num2d}
We conclude our investigations with some further numerical tests,
which aim to give a preliminary assessment of the effect of the
stability error on practical computations. Our experiments can only be
considered preliminary since we only consider a limited class of
interactions and, due to the significant computational cost involved,
we do not include extensive tests on domain size dependence.



\subsubsection{Stability gap}
In these experiments we admit vectorial deformations $y : \L \to
\R^2$, but otherwise use the same structure of atomistic and QNL
models. The potential used in our numerical experiments is a modified
EAM potential,
\begin{align*}
  V(g) &:= \sum_{\rho \in \Rg} \phi(|g_\rho|) + G\bg( \sum_{\rho\in
    \Rg} \psi(|g_\rho|) \bg) + D \sum_{j=1}^6 (r_j \cdot r_{j+1}-1/2)^2, \quad \text{where}  \\
  \phi(s) &:= e^{-2 A (s-1)} - 2 e^{-A (s - 1)},  \quad 
  \psi(s) := e^{- B s},  \quad \text{and} \quad 
  G(s) := C \b( (s - s_0)^2 + (s - s_0)^4 \b).
\end{align*}
Throughout, we fix the parameters $A = 3, B = 3, s_0 = 6 e^{-0.95 B}$,
but vary $C$ and $D$ between experiments.

Instead of a half-space, we perform our calculations in a hexagonal
domain with sidelength 18 atomic spacings and Dirichlet boundary
conditions. The atomistic region is a concentric hexagon with
sidelength 6 atomic spacings. We consider only the GRAC-2/3 method,
which is the only force-consistent method that we know of for this
setup.

Applying uniform expansion $\mF(t) = t \mI$ as load, we obtained the
results shown in Figure~\ref{fig:ex[1][05]} for parameters $C = 1, D =
-0.5$ and in Figure~\ref{fig:ex[1][0]} for parameters $C = 1, D = 0$.

In Figure~\ref{fig:ex[1][05]}(a) we observe a small but clear
gap in the stability constants where they cross zero. Realistically,
given the smallness of the gap, we must question whether it is
genuine or a numerical error such as a domain size effect. The plots
in Figure~\ref{fig:ex[1][05]}(b, c) suggest that the gap is genuine
since the unstable QNL eigenmode is concentrated on the interface, and
therefore of a different ``type'' than the unstable eigenmode of the
atomistic model.

Interestingly, in Figure~\ref{fig:ex[1][0]}, we still observe the same
characteristic difference in the eigenmodes, but the stability gap is
essentially absent. We can only conjecture that, analytically, a gap
must be present, but numerically it is too small to detect
reliably. And indeed, this means that it may be of little practical
relevance.

The two examples we have shown are prototypical for the entire
parameter range $C \in [-1, 1]$ and $D \in [-1, 1]$ that we
tested. Given how small the stability errors seem to be in practice
(at least in these experiments), this raises the question whether one
can quantify them, instead of trying to eradicate them completely.

\begin{figure}
  \begin{center}
    \hspace{-4mm} 
    \includegraphics[height=8.9cm]{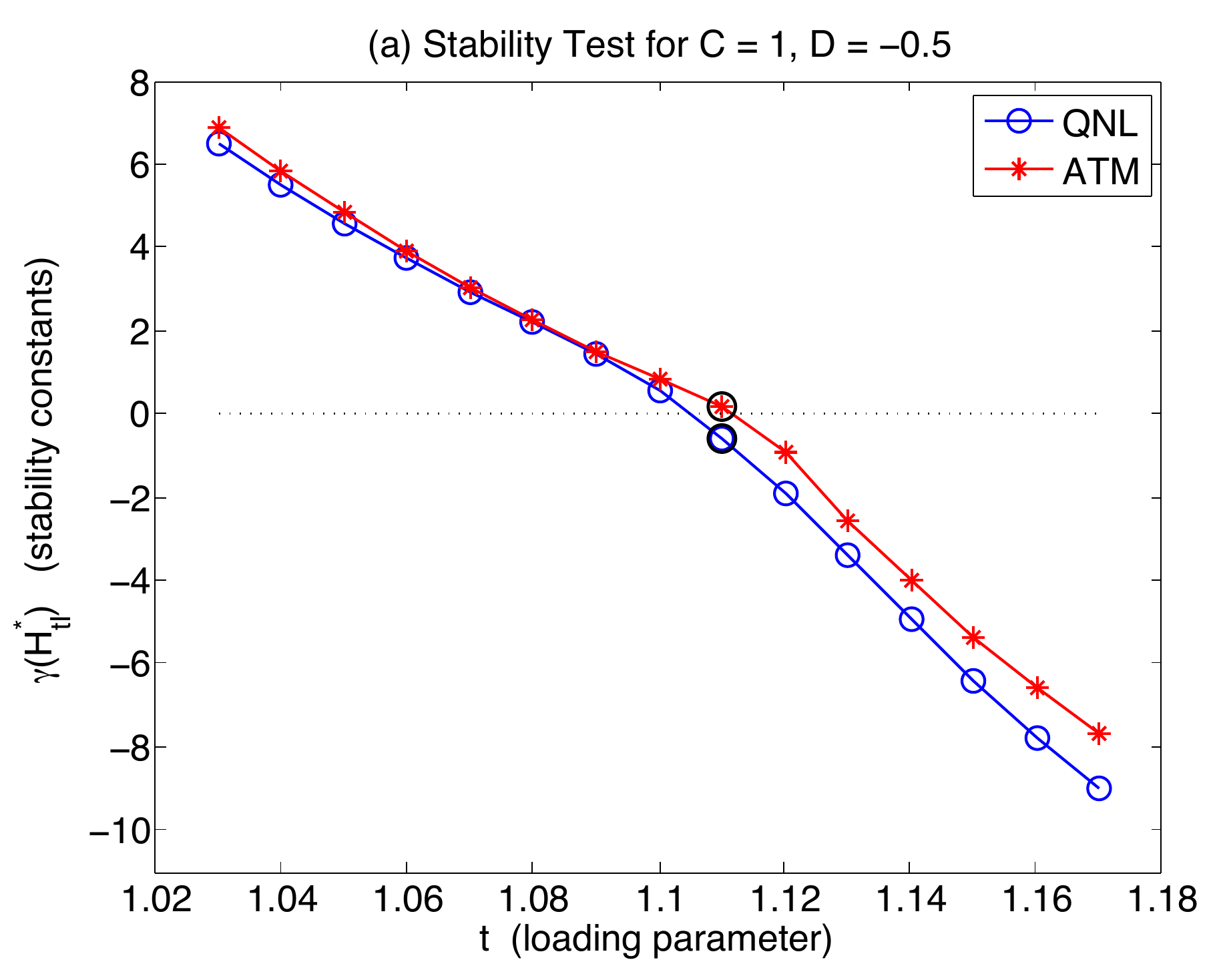}
    \hspace{-4mm}
    \includegraphics[height=8.9cm]{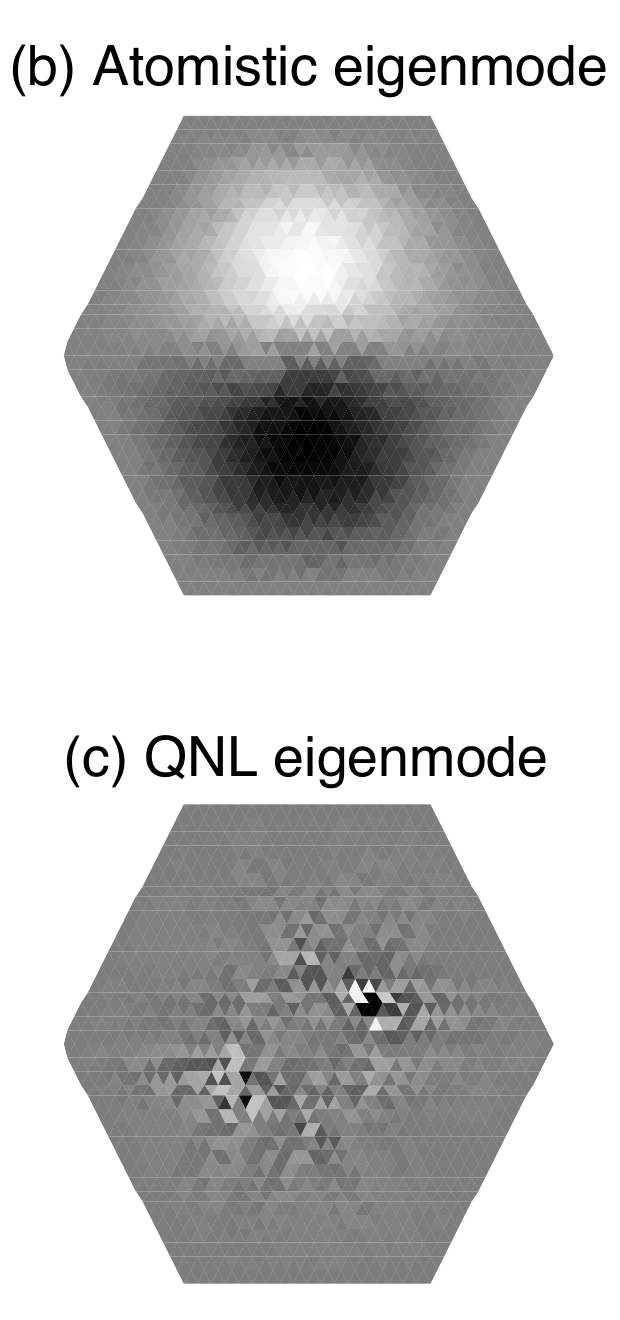}
    \hspace{-4mm} 
    \caption{ \label{fig:ex[1][05]} Stability test for $C = 1, D =
      -0.5$, as described in \S~\ref{sec:num2d}. The black circles
      indicate which eigenmodes ($u_1$-component) are plotted in (b,
      c).}
  \end{center}
\end{figure}

\begin{figure}
  \begin{center}
    \includegraphics[height=9cm]{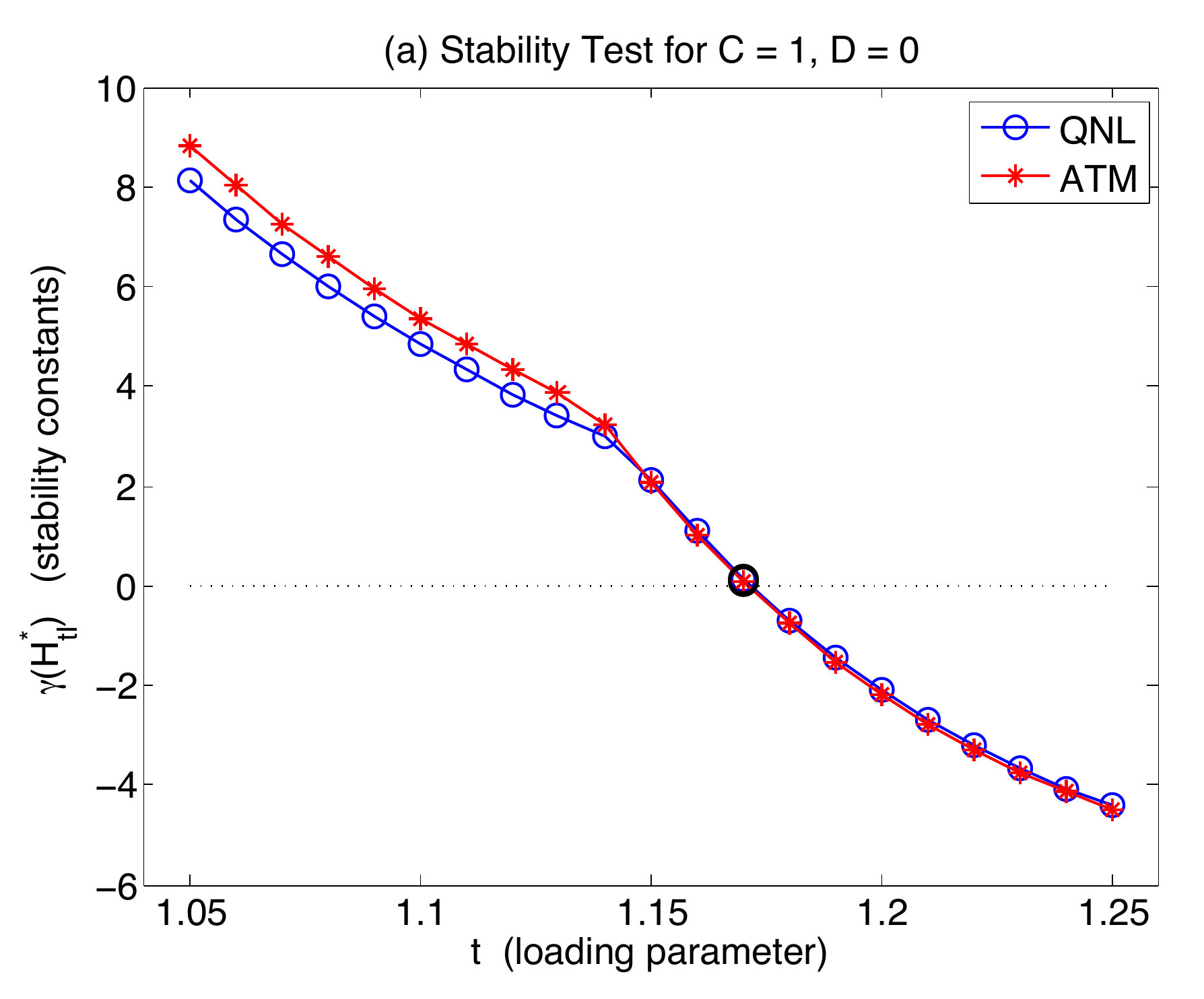}
    \includegraphics[height=9cm]{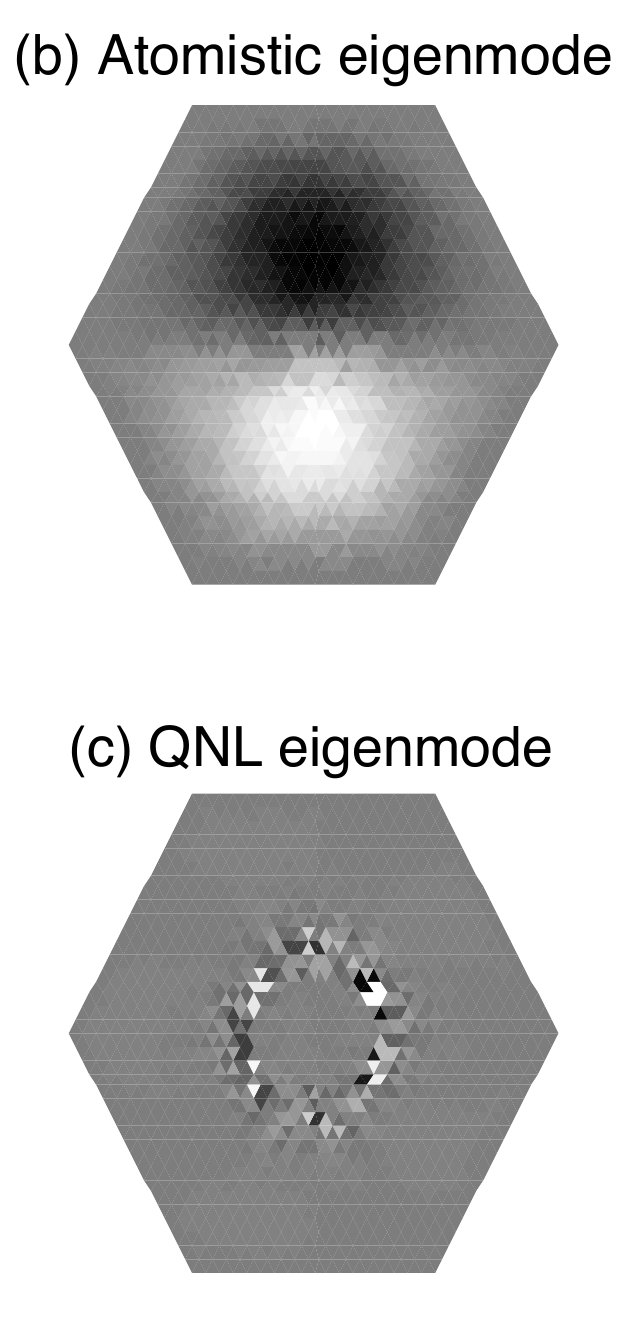}
    \caption{ \label{fig:ex[1][0]} Stability test for $C = 1, D = 0$,
      as described in \S~\ref{sec:num2d}. The black circles indicate
      which eigenmodes ($u_1$-copmponent) are plotted in (b, c).}
  \end{center}
\end{figure}

\subsubsection{Stabilisation}
For the parameters $C = 1, D = -0.5$, where we observed a visible
stability gap in Figure \ref{fig:ex[1][05]}, we now consider the
stabilised GRAC-2/3 scheme \eqref{eq:defn_Estab:1} with $S$ given by
\eqref{eq:2d:defn_S} and $\kappa \geq 0$. Repeating the numerical
experiment of the previous section we obtain the results shown in
Figure~\ref{fig:exstab[0.1]} for $\kappa = 0.1$ and in
Figure~\ref{fig:exstab[1]} for $\kappa = 1$.

In both experiments we observe a much smaller stability gap (for
$\kappa = 1$ no gap is visible with the plain eye), and this is
accompanied by a marked change in the qualitative behaviour of the
critical eigenmode. In both cases, the stabilisation has changed the
interface supported eigenmode into a bulk eigenmode, which one might
consider ``smooth''. This indicates that the stability gap has closed.

For the stronger stabilisation $\kappa = 1$, the critical QNL
eigenmode is now identical to the atomistic eigenmode, while for
$\kappa = 0.1$ the QNL eigenmode has a shorter wave length. The
existence of this ``weaker'' eigenmode explains the larger stability
gap for $\kappa = 0.1$ compared with $\kappa = 1.0$.

\begin{figure}
  \begin{center}
    \hspace{-4mm} 
    \includegraphics[height=8.9cm]{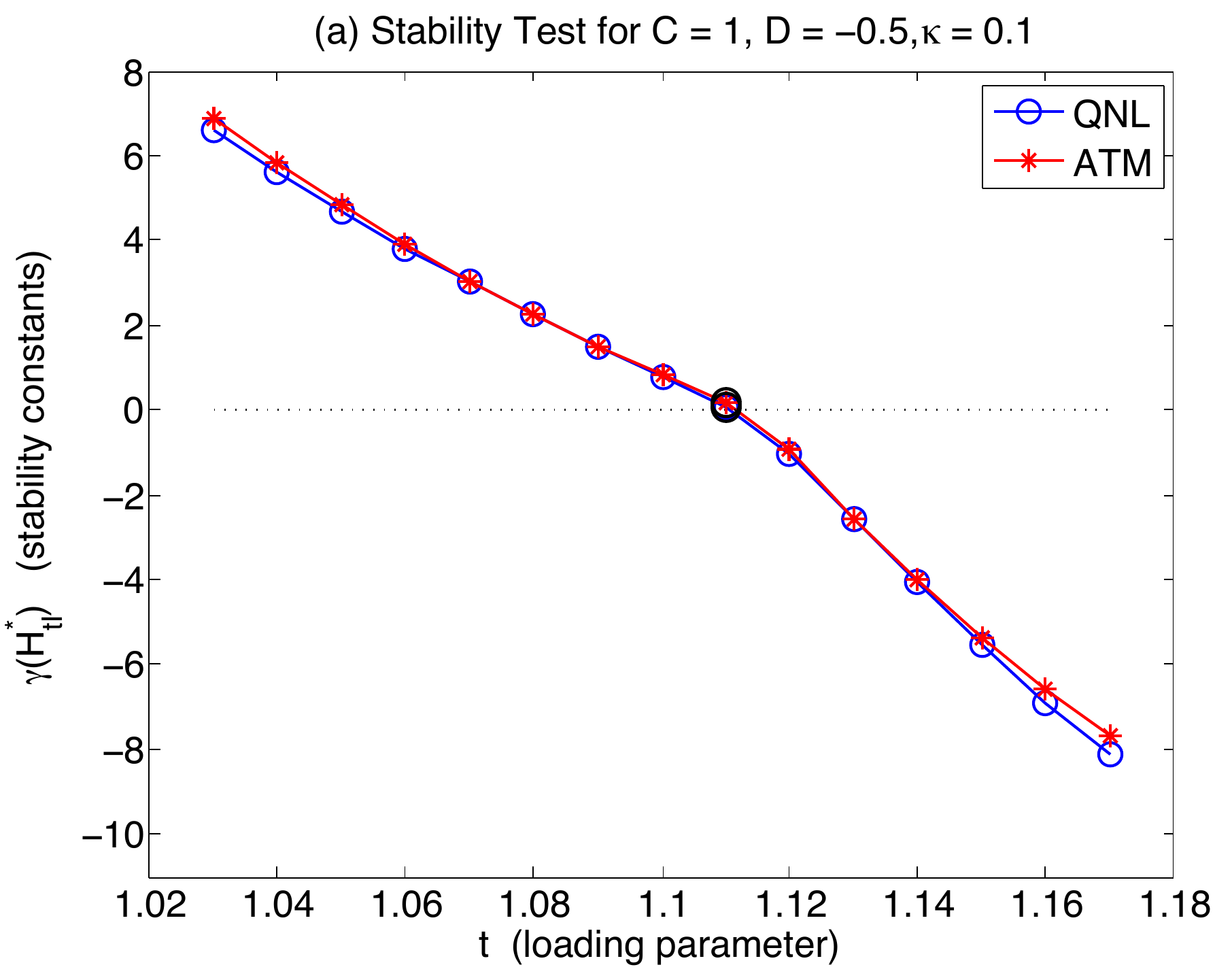}
    \hspace{-4mm}
    \includegraphics[height=8.9cm]{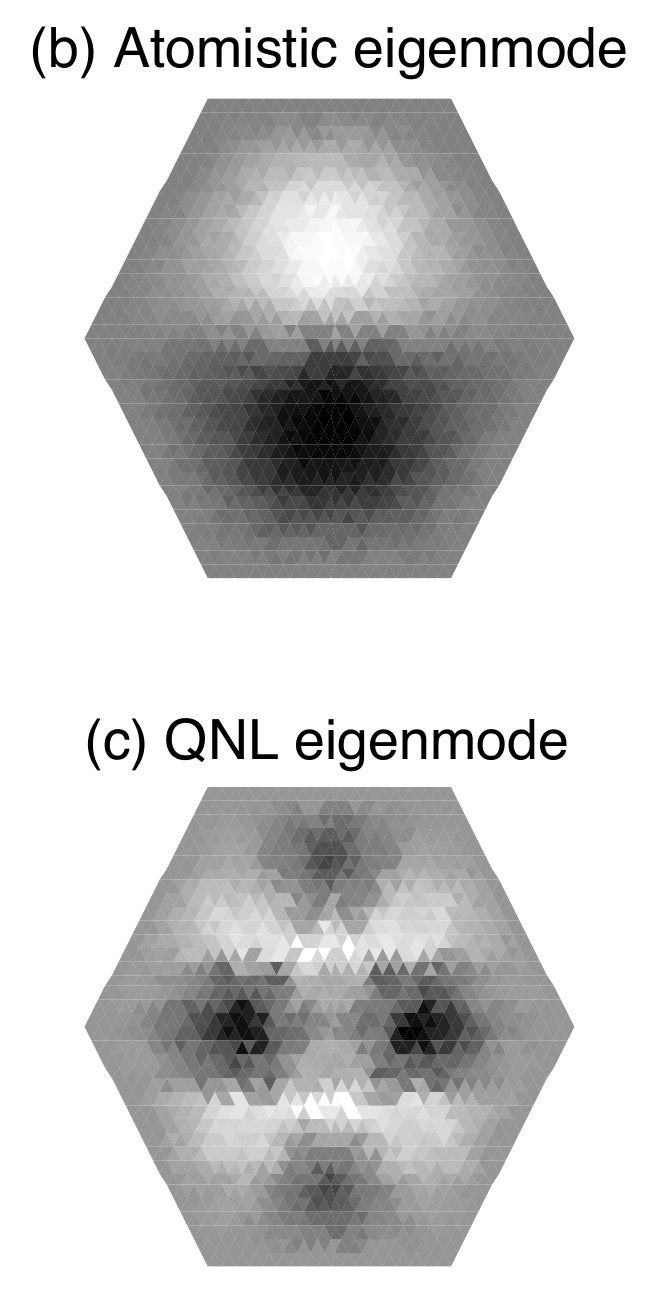}
    \hspace{-4mm} 
    \caption{ \label{fig:exstab[0.1]} Stability test for $C = 1, D =
      -0.5, \kappa = 0.1$, as described in \S~\ref{sec:num2d}. The
      black circles indicate which eigenmodes ($u_1$-component) are
      plotted in (b, c).}
  \end{center}
\end{figure}

\begin{figure}
  \begin{center}
    \hspace{-4mm} 
    \includegraphics[height=8.9cm]{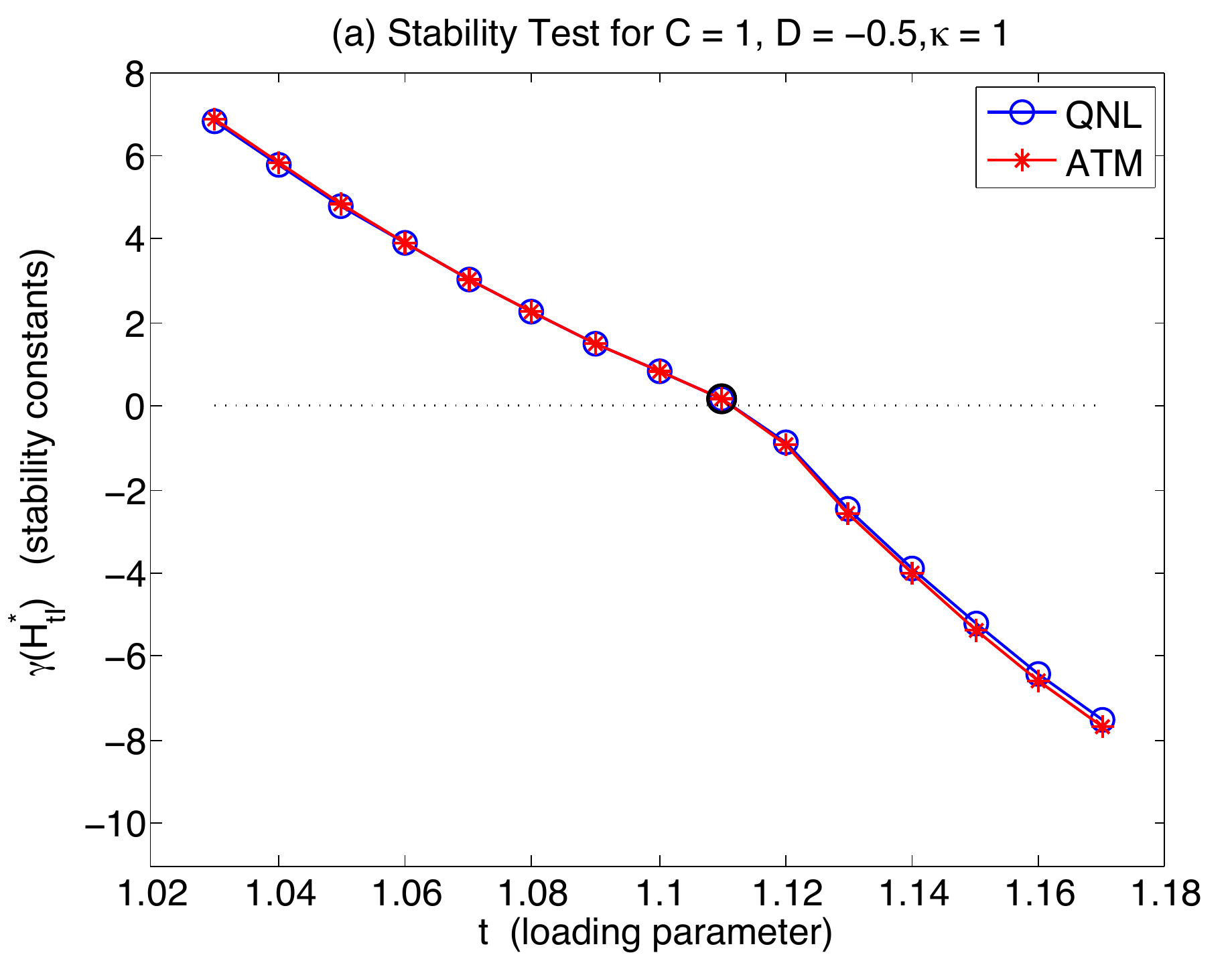}
    \hspace{-4mm}
    \includegraphics[height=8.9cm]{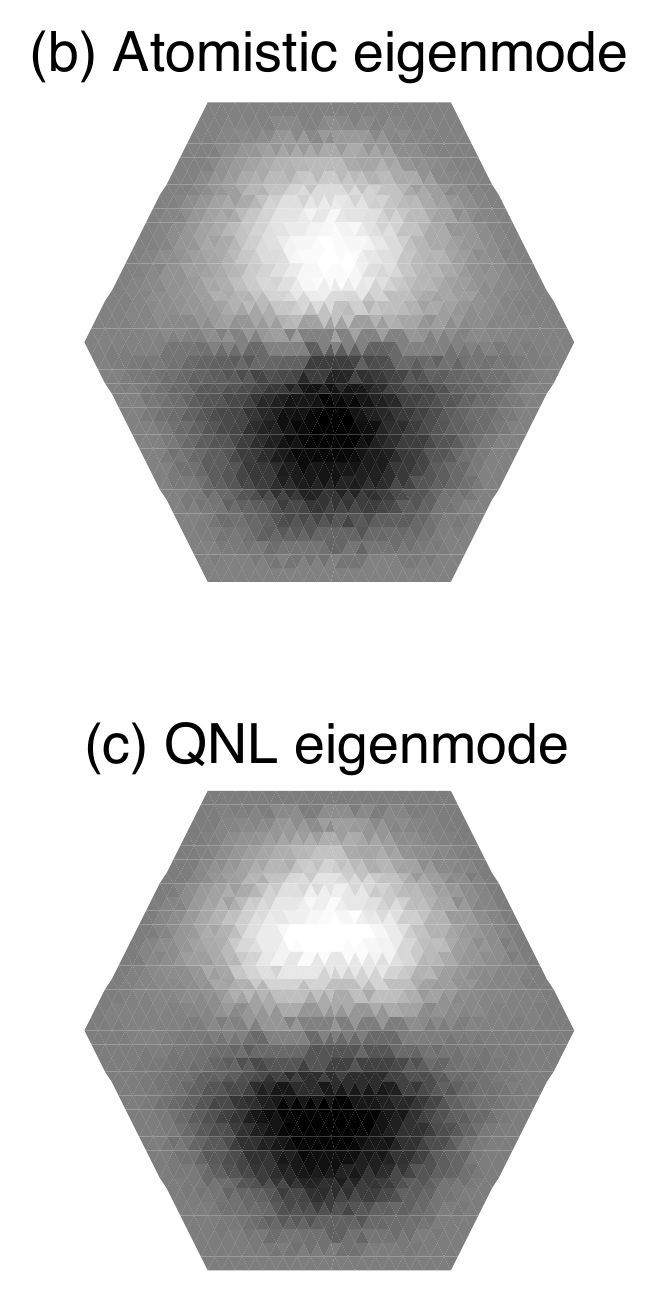}
    \hspace{-4mm} 
    \caption{ \label{fig:exstab[1]} Stability test for $C = 1, D =
      -0.5, \kappa = 1$, as described in \S~\ref{sec:num2d}. The black
      circles indicate which eigenmodes ($u_1$-component) are plotted
      in (b, c).}
  \end{center}
\end{figure}

\section{Conclusion}
The stability of QNL type a/c coupling mechanisms in dimension greater
than one remains an interesting issue. Our results in the present work
indicate that it is unlikely that there exists a {\em universally
  stable} coupling scheme (except in 1D), but that suitable
stabilisation mechanisms must be employed.

We have proposed and analysed a specific stabilisation mechanism in a
simplified setting. Our results indicate that this is a promising
avenue to explore further, but that much additional work is required
to establish this as a practical computational scheme.

We recall, however, that in \S~\ref{sec:num2d} we also raised the
question whether stabilisation is at all required in practice since
the stability errors, at least for the class of interactions we
considered there, appear to be fairly small. However, it is unclear to
us at this point how one might quantify such a statement.

\appendix

\section{Appendices}

\subsection{Details of the instability example in \S~\ref{sec:qnl1d:instab}}
\label{sec:app_1dqnl2_example}
This section describes the detail of the calculation in \S~\ref{sec:qnl1d:instab}. 

The second variation of the atomistic energy gives
\begin{align*}
 \<\Ha u, u\>  =& (2-2\alpha+8\beta-8\gamma+16\delta)\sum_{\xi \in \Z} |D_1 u(\xi)|^2 + (\alpha-2\beta+18\gamma-12\delta)\sum_{\xi \in \Z} |D^2_1 u(\xi)|^2 \\
		& + (-8\gamma+2\delta)\sum_{\xi \in \Z} |D^3_1 u(\xi)|^2+\gamma\sum_{\xi \in \Z} |D^4_1 u(\xi)|^2 
\end{align*}
which can be written in short form as
\begin{displaymath}
\<\Ha u, u\>  = A_1\sum_{\xi \in \Z} |D_1 u(\xi)|^2+  A_2\sum_{\xi \in \Z} |D^2_1 u(\xi)|^2 + A_3\sum_{\xi \in \Z} |D^3_1 u(\xi)|^2+A_4\sum_{\xi \in \Z} |D^4_1 u(\xi)|^2
\end{displaymath}

By \cite{Hudson:stab} (also Li \& Luskin paper), we have 
\begin{displaymath}\label{eqn:cc}
\ca(\mF) = \min_{0\leq s\leq 4} A_1 + A_2 s + A_3 s^2 + A_4 s^3
\end{displaymath}

With the parameters in \S~\ref{sec:qnl1d:instab}, $\alpha = -0.99, \beta = 0.1, \gamma = 0.15,
\delta = -0.2$, we obtain that $\ca(\mF) = 0.02$.

Similarly, the second variation of the QNL energy is
\begin{align*}
 & \<\Hqnl u, u\> \\
 = & (2-2\alpha+8\beta-8\gamma+16\delta)\sum_{\xi \in \Z} |D_1 u(\xi)|^2  + (\alpha-2\beta+18\gamma-12\delta)\sum_{\xi \leq -4} |D^2_1 u(\xi)|^2 \\ 
		  & + (\alpha-2\beta+17\gamma-12\delta)|D^2_1 u(-3)|^2  + (\alpha-2\beta+15\gamma-11\delta)|D^2_1 u(-2)|^2 \\
		  & + (\alpha+6\gamma-5\delta)|D^2_1 u(-1)|^2 +(-8\gamma+2\delta)\sum_{\xi \leq -4} |D^3_1 u(\xi)|^2 + (-6\gamma+2\delta)|D_1 u(-3)|^2 \\
		  & + (-2\gamma+\delta)|D_1 u(-2)|^2 +\gamma \sum_{\xi \leq -4} |D^4_1 u(\xi)|^2
\end{align*}
which gives the explicit expression for the coefficients $A$, $B_\xi$,
$C_\xi$ and $D$ in \eqref{eq:qnl1d:instab}. $\gamma^\qnl$ can be
estimated by numerical calculation.  For $u$ supported in $[-500,
500]$, we have $\gamma^\qnl<-0.005$. The unstable mode is plotted in
Figure \ref{fig:qnl2n_instab}.

\begin{figure}
  \begin{center}
    \includegraphics[height=5cm]{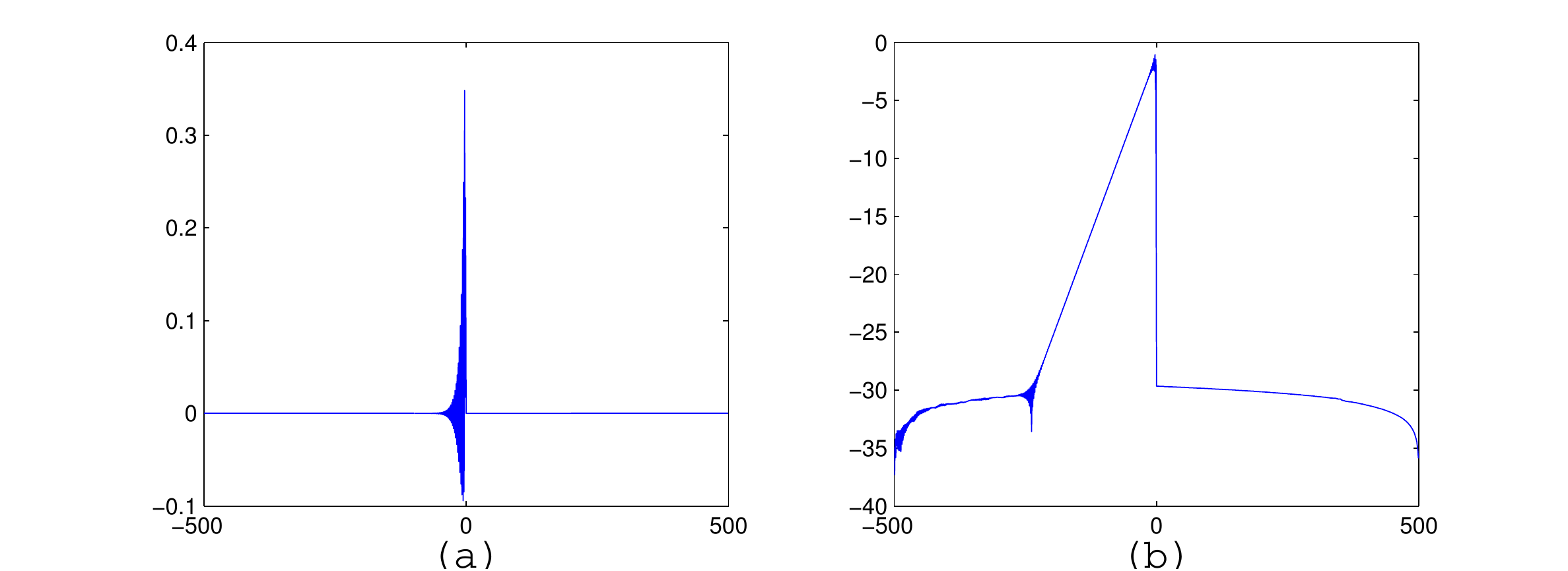}
    \caption{\label{fig:qnl2n_instab} (a) unstable mode of $u$
      supported in [-500, 500]; (b) $\log(|u|)$.}
  \end{center}
\end{figure}

\subsection{Details of the 1D QNL numerical test in \S~\ref{sec:instab_qnl_num}}
\label{sec:app_1dqnl2_num}
This section describes the details of the setup of the numerical test reported in \S~\ref{sec:instab_qnl_num}.

In these experiments we use $\Rg = \{\pm 1, \pm 2\}$, and the EAM type
interaction potential
\begin{align*}
  V(g) &:= \sum_{\rho \in \Rg} \phi(|g_\rho|) + G\bg( \sum_{\rho\in
    \Rg} \psi(|g_\rho|) \bg), \quad \text{where}  \\
  \phi(s) &:= e^{-2 A (s-1)} - 2 e^{-A (s - 1)}, \\
  \psi(s) &:= e^{- B s},  \quad \text{and} \\
  G(s) &:= C \b( (s - s_0)^2 + (s - s_0)^4 \b).
\end{align*}
Throughout, we use the parameters $A = 3, B = 3, C = 5$ and $s_0 = 2
e^{-0.95 B} + 2 e^{-1.9 B}$.

Next, we redefine $\Eqnl$ with finite atomistic and continuum regions. Fix $K, N \in \N$ and a macroscopic strain $\mF > 0$. Admissible deformations $y : \Z \to \R$ are those for which $D_1 y(\xi) > 0$ for all $\xi$ and $y(\xi) = \mF \xi$ for $|\xi| \geq N$. Let $\Us_N := \{ u \in \Us \sep u(\xi) = 0$ for $|\xi| \geq N \}$, then the admissible deformation space is $\mF x + \Us_N$.

For any admissible deformation, we then define 
\begin{align*}
  \Eqnl(y) &:= \sum_{\xi = -K+2}^{K-2} \b[ V(Dy(\xi))  - V (\mF \Rg) \b] \\
  & \quad + \sum_{\xi = -K}^{-K+1} \b[ V(\tilde{D}^- y(\xi))  - V (\mF \Rg) \b] 
  + \sum_{\xi = K-1}^{K} \b[ V(\tilde{D}^+ y(\xi))  - V (\mF \Rg) \b] \\
  & \quad + \int_{-N}^{-K-1/2} \b[ W(\D y) - W(\mF) \b] 
  + \int_{K+1/2}^N \b[ W(\D y) - W(\mF) \b],
\end{align*} 
where $\tilde{D}^+ = (D_{-2}, D_{-1}, D_1, 2 D_2)$ and $\tilde{D}^- =
(2 D_{-1}, D_{-1}, D_1 D_2)$.

We will also compare the results against an atomistic model restricted to a finite domain (by simply restricting the admissible deformations as above), and against the reflection method defined in \S~\ref{sec:refl}, which can be analogously formulated on the finite domain.

Moreover, given parameters $\alpha, \beta \in \R$, define an external force
\begin{displaymath}
  f(\xi) := \beta (1+\xi^2)^{-(\alpha+1)/2}.
\end{displaymath}

Finally, we discretise the continuum region using P1 finite elements. Motivated by the analysis in \cite{acta}, we choose a scaling for the atomistic region size and a scaling for the mesh size, according to the decay of the external force: $K = \lceil N^{(\alpha-1/2) / (\alpha+1/2)}\rceil$ and $h(x) \approx (|x|/K)^{\smfrac23 (\alpha+1)}$. We create the FE mesh using the algorithm described in \cite{acta}. Let $\Us_h$ denote the FE displacement space, of piecewise affine functions extended by zero outside $[-N, N]$.

Using Newton's method, we compute a continuous path of equilibria of the energy \begin{displaymath}
  \Eqnl(y_\mF) - \sum_{\xi \in \Z} f(\xi) \cdot y_\mF(\xi), \quad y_\mF \in \mF x + \Us_h,
\end{displaymath}
starting with $\mF = 1$ and incrementing $\mF$ in small steps, using the previous step as starting guess. Using a bisection type approach, we can define the critical strain $\mF^\qnl$ to be the smallest value of $\mF$ for which $\ddel\Eqnl(y_\mF)$ ceases to be positive definite on $\Us_h$. Analogously, we define the critical strains for the reflection method, $\mF^\refl$, and for the atomistic model restricted to $\Us_N$, $\mF^\a$.

The exact critical strain, $\mF^*$, is defined to be the critical
strain for the unrestricted atomistic model. Since we have shown that
the reflection method is universally stable, which is extended to a
nonlinear deformation in \cite{acta}, we compute $\mF^*$ by
extrapolating the computed critical $\mF^\refl$ for increasing domain
sizes. The results for increasing domain sizes $N$, with corresponding
choices of $K$ and the FE mesh, are displayed in
\S~\ref{sec:instab_qnl_num}.

\subsection{Proof of Proposition \ref{th:2d:cb_hess}}
\label{sec:proof_2d:cb_hess}
%
  Recall that $W(\mG) = V(\mG \Rg) = V(\mG a_1, \dots, \mG
  a_6)$. Therefore,
  \begin{align*}
    \mG^\top \ddW(\mF) \mG = \sum_{i, j = 1}^6 V_{i,j}(\mF\Rg)
    (\mG\cdot a_i) (\mG \cdot a_j).
  \end{align*}

  Fix an element $T$ and let $\xi_j \in \L \cap T$ such that $\xi_j +
  a_j \in T$, then 
  \begin{align}
    \notag
    \D y_T^\top \ddW \D y_T &= \sum_{i,j = 1}^6 V_{i,j} (\D y_T \cdot
    a_i) (\D y_T \cdot a_j) \\
    \label{eq:2d:hess_rep_cb:10}
    &= \sum_{i,j = 1}^6 V_{i,j} \, D_i y(\xi_i) \, D_j y(\xi_j).
  \end{align}
  We now observe that all products $D_i y(\xi_i) D_j y(\xi_j)$ can be
  rewritten in the form $\pm D_i y(\xi_i) D_{i+1} y(\xi_i)$ or $\pm
  D_i y(\xi_i) D_{i-1} y(\xi_i)$, and hence as a sum of squares
  \begin{displaymath}
    D_i y(\xi_i) \, D_j y(\xi_j) = \sum_{k = 1}^3 b_{T,i,j,k} |D_k y(\xi_k)|^2.
  \end{displaymath}
  Inserting this back into \eqref{eq:2d:hess_rep_cb:10}, and then
  summing over $T \in \T$, we conclude that there exist coefficients
  $c_i(\xi, \mF), i = 1, 2, 3$ such that
  \begin{displaymath}
    \< \Hc_\mF u, u \> = \sum_{\xi \in \L} \sum_{i = 1}^3 c_i(\xi,\mF)
    |D_i u(\xi)|^2.
  \end{displaymath}

  Now define $v(\xi) := u(\xi+a_k)$, then clearly $\Ec(v) = \Ec(u)$
  and hence $\< \Hc u, u \> = \< \Hc v, v \>$, which can
  equivalently written as
  \begin{displaymath}
    \sum_{\xi \in \L} \sum_{i = 1}^3 c_i(\xi,\mF)
    |D_i u(\xi)|^2   =  \sum_{\xi \in \L} \sum_{i = 1}^3 c_i(\xi-a_k,\mF)
    |D_i u(\xi)|^2.
  \end{displaymath}
  Lemma \ref{th:equiv_hessians} implies that $c_j(\xi-a_k,\mF) =
  c_j(\xi,\mF)$, that is, the coefficients are independent of $\xi$.

  If we have the full hexagonal symmetry \eqref{eq:2d:hex_symm} then
  an analogous argument, employing again Lemma
  \ref{th:equiv_hessians}, implies that $c_j$ does not depend on the
  directions $j$ either.

  The stated formula for $\ddW$ follows from 
  $\D y_T^\top \ddW(\mF) \D y_T = \frac12 \sum_{j = 1}^3 c_j |D_j
    u(\xi_j)|^2.$

This completes the proof of Proposition \ref{th:2d:cb_hess}.

\subsection{Explicit examples}
\label{sec:2d:qnl_hess_explicit}
Here we explicitly compute the QNL hessian representation that we
established in Proposition \ref{th:2d:qnl_hess_straingrad} and in
Lemma \ref{th:2dqnl:consequence_of_forcecons}, for the three QNL
methods introduced in \S~\ref{sec:2d:def_qnl}. We will consider the
case of full hexagonal symmetry with $V_{i,i+2} = V_{i,i+3} = 0$ since
the atomistic hessian has no strain gradient terms in this case;
\eqref{eq:Ha_coeffs_hexsym} then becomes
\begin{equation}
  \label{eq:2d:hex_sym_simple_Ha_Hc}
  \< \Ha_\mF u, u \> = \< \Hc_\mF u, u \> = 2(\alpha_0 + \alpha_1)
  \| Du \|_{\ell^2}^2.
\end{equation}

\begin{proposition}
  \label{th:2d:Hqnl_simple}
  Suppose that we have hexagonal symmetry \eqref{eq:2d:hex_symm} and
  that $\alpha_j = 0$ for $j = 2, 3$, then there exist sums of squares
  $X^\lref, X^\grac$ such that
  \begin{align}
    \label{eq:2d:Hqce_simple}
    \< H^\qce u, u \> &= \< \Ha u, u \> + {\frac{\alpha_0+4\alpha_1}{3}}
    \! \sum_{\xi \in \L^{(0)}} \! \Big( |D_1 u(\xi)|^2 - |D_1 u(\xi+a_2)|^2 \Big), \\
    \label{eq:2d:Hlref_simple}
    \< H^\lref u, u \> &= \< \Ha u, u \> {- \alpha_1} \! \sum_{\xi \in
      \L^{(0)}} \! \B( |D_1u(\xi)|^2 - |D_1 u(\xi+a_5)|^2 \B) \\
    \notag
    & \hspace{2.2cm} + \! \sum_{\xi \in
      \L^{(0)}} \! X^\lref(D^2 u(\xi)),  \qquad \text{and} \\
    \label{eq:2d:Hg23_simple}
    \< H^{\rm g23} u, u \> &= \< \Ha u, u \> + {(\alpha_0 + 2\alpha_1)} \!\sum_{\xi \in \L^{(0)}}\! \B( |D_1u(\xi)|^2 -
    |D_1 u(\xi+a_2)|^2 \B) \\
    \notag
    & \hspace{2.2cm} + \!\sum_{\xi \in \L^{(0)}}\! X^{\rm g23}(D^2
    u(\xi)).
  \end{align} 
\end{proposition}
\begin{proof}
  The proofs of these three identities are purely algebraic, but
  fairly tedious. One simply applies the three ``rules''
  \eqref{eq:RULE-1}--\eqref{eq:RULE-3} and then collects coefficients.
  We will give an outline of the proof of \eqref{eq:2d:Hqce_simple}
  and then briefly remark on how to obtain \eqref{eq:2d:Hlref_simple}
  and~\eqref{eq:2d:Hg23_simple}.

  Since $\E^\qce$ is force consistent \eqref{eq:2d:qnl_nogf} Lemma
  \ref{th:2dqnl:consequence_of_forcecons} implies that
  $\tilde{c}_j(\xi) = c_j$ except possibly for $\xi \in \L^{(m)}, m =
  -1, 0, 1$ and for $j = 1$. That is, we only have to compute
  $\tilde{c}_1^{(m)}, m = -1, 0, 1$. Since the bonds $(\xi,\xi+a_1),
  \xi \in \L^{(-1)}$ experience a purely atomistic environment, it
  follows that $\tilde{c}^{(-1)}_1 = c_1$.

  Next, we compute $\tilde{c}^{(0)}_1$. Consider the representative
  bond $(0, a_1)$. The coefficient for this bond receives
  contributions from $\xi \in \{0, a_1, a_6\}$ and from the element $T
  = {\rm conv}\{ 0, a_1, a_2\}$ weighted with a factor
  $1/3$. Computing all of these contributions, we obtain that
  $\tilde{c}_1^{(0)} = c_1 + {(\alpha_0 + 4 \alpha_1)/3}$.

  The coefficient $\tilde{c}^{(1)}_1$ is obtained from the identity
  \eqref{eq:2dqnl:sum_tilc1_eq_3c1}, which implies $\tilde{c}^{(1)}_1
  = c_1 - {(\alpha_0 + 4 \alpha_1)/3}$.

  Finally, the QCE hessian has no strain gradient correction at the
  interface due to the fact that writing either $\D u^\top \ddW(\mF)
  \D u$ or $D_i u D_{i+1} u$ as sums of squares does not produce any
  such terms. This establishes \eqref{eq:2d:Hqce_simple}.

  To obtain \eqref{eq:2d:Hg23_simple} it is easiest to write out $\<
  (H^\grac - H^\qce) u, u \>$ and covert it into strain gradient form,
  by applying the ``rules'' \eqref{eq:RULE-1}--\eqref{eq:RULE-3}. Note
  that, since  for the GRAC-2/3 method,
  \begin{displaymath}
    \tilde{V}(Dy) = V\b(D_1 y, \smfrac12 D_1 y + \smfrac23 D_2 y +
    \smfrac13 D_3 y, \dots\b),
  \end{displaymath}
  computing its hessian we obtain mixed terms of the form 
  \begin{displaymath}
    V_{1,2} D_1 u \b(\smfrac12 D_1 u + \smfrac23 D_2 u +
    \smfrac13 D_3 u \b),
  \end{displaymath}
  and therefore products $D_1 u D_3 u$, and similarly also $D_2 u D_4
  u$ and $D_1 u D_4 u$ occur. These give rise to strain gradient terms
  at the interface.

  To obtain \eqref{eq:2d:Hlref_simple} one can first shift the
  interface by $a_5$ (or $a_6$), again write the difference $\<
  (H^\lref - H^\qce) u, u \>$, and then proceed analogously as
  above. Again, due to the reorganisation of the bond directions,
  strain gradient terms occur.
\end{proof}

\subsection{Generalisation of the 2D strain gradient representation}
\label{sec:app_general_hessrep}

The strain gradient representation in section \ref{sec:stab1d} and
\ref{sec:2d:qnl_hessian_rep} can be generalized to 2D general many-body
potentials, in particular, the a/c Hessian $\Hac_\mF$ of a QNL-type
method can be represented as a sum of squares of the higher order
derivatives of strains. $\Hac_\mF$ is defined as,
\begin{equation}
  \label{eq:app:qnl_hessian_defn}
  \< \Hac_\mF u, u \> = \sum_{\xi \in \La\cup\L^{\i}} \sum_{\rho,\vsig \in \Rg}
  \tilde{V}_{\xi,\rho\vsig} \cdot D_\rho u(\xi) D_\vsig u(\xi) + \sum_{T \in \T^{\c}} w_T
  (\D u_T)^\top \ddW(\mF) \D u_T,
\end{equation}
where $\La$ is the set of atomistic nodes, $\L^{\i}$ is the set of interface nodes. $\tilde{V}$ is energy-consistent and force-consistent, $\tilde{V}=V$ for $\xi\in\La$.  

For 2D lattice $\L$ with 6 nearest neighbor directions $a_i$, $i=1,\dots,6$, suppose that $\xi, \eta \in\L$, a path $\Gamma$ connecting $\xi$ and $\eta$ is a sequence of 
lattice nodes $\mu^i$, $j=0,\dots, N$, such that $\mu^0 = \xi$, $\mu^N = \eta$, and $\mu^{j+1}-\mu^j = e^j$ with $e^j\in\{a_i\}_{i=1}^6$. $N$ is the length of the path $\Gamma$. The shortest path between $\xi$ and $\eta$ is the path connecting $\xi$ and $\eta$ which attains the hopping distance between $\xi$ and $\eta$, such path exists and may not be unique. 
 
Define $R:=\max_{\rho\in\Rg}|\rho|_h$, where $|\cdot|_h$ is the hopping distance. The following lemma is a genalization of Lemma \ref{th:prod_D_lemma} in 2D.

\begin{lemma}
  \label{th:prod_2D_lemma}
  For $\xi \in \L, \rho, \vsig \in \Rg$, we have
  \begin{displaymath}
    D_\rho u(\xi) D_\vsig u(\xi) = \sum_{e_1\in \Ed\cap S(\xi, \rho, \vsig), e_1=(\eta,\eta+\mathbf{a}_1)}c(e_1)|D_{\mathbf{a}_1}u(\eta)|^2 + \sum_{2 \leq r \leq R, \eta\in S(\xi, \rho, \vsig)} X(D^r u(\eta)).
  \end{displaymath}
	where $S(\xi,\rho, \vsig)$ is a union of some $T\in \T$, $\xi,\xi+\rho, \xi+\vsig \in S(\xi, \rho, \vsig)$, such that
$S(\xi,\rho, \vsig)$ is convex and attains minimum area. X is a sum of squares, and the coefficient is nonzero only when the support of higher order strain gradient $D^r u(\eta)$ in the above sum is contained in $S(\xi, \rho, \vsig)$.
\end{lemma}
\begin{proof}
	Let $\Gamma_\rho$ be a shortest path between $\xi$ and $\xi+\rho$ with nodes $\mu^j_\rho$, $j=0,\dots,N_\rho$, and 
	directions $e^j_\rho$, $j = 0, \dots, N_\rho-1$. It is clear from the definitions that
  \begin{displaymath}
    D_\rho u(\xi) = \sum_{j = 0}^{N_\rho -1} D_{e_\rho^j}u(\mu^\rho_j) ,
  \end{displaymath}
  and therefore,
  \begin{displaymath}
    D_\rho u(\xi) D_\vsig u(\xi) = \sum_{j = 0}^{N_\rho -1}
    \sum_{j' =0}^{N_\vsig-1}  D_{e_\rho^j} u(\mu^\rho_j) D_{e_\vsig^{j'}} u(\mu^\vsig_{j'}).
	\end{displaymath}
	To finish the proof, we just need the following Lemma \ref{th:elemprod_2D_lemma}.

\end{proof}

Let $\Ed$ denote all the edges in the triangulation $\T$. If $e\in\Ed$, then there exist $\xi\in\L$ and $\mathbf{a}\in \{a_i\}_{i=1}^6$, such that $e=(\xi, \xi+\mathbf{a})$.
Define $S_1(e)$ as the union of two triangles in $\T$ which share $e$ as a common edge, $S_1(e) := \cup_{e\in T} T$, $T_1(e) = S_1(e)$, for $n>1$, $S_n(e) := \cup_{T\in\T, T\cap S_{n-1}(\e)\neq \emptyset}T$, and $T_n(e) := \overline{S_n(e)\setminus S_{n-1}(e)}$.


For $e, e'\in\Ed$, let $S(e,e')$ be a union of some $T\in \T$, $e,e'\in S(e,e')$, such that
$S(e,e')$ is convex and attains minimum area. If $e'\in T_n(e)$, since $S_n(e)$ is convex, we have $S(e,e')\subset S_n(e)$.

\begin{lemma}
	\label{th:elemprod_2D_lemma}
	For $e, e' \in \Ed$, $e=(\xi, \xi+\mathbf{a})$, $e'=(\xi',\xi'+\mathbf{a}')$. If $e'\in T_n(e)$, then 
	\begin{displaymath}
		D_{\mathbf{a}}u(\xi)  D_{\mathbf{a}'}u(\xi') = \sum_{e_1\in \Ed\cap S(e, e'), e_1=(\eta,\eta+\mathbf{a}_1)} c(e_1)|D_{\mathbf{a}_1}u(\eta)|^2 + \sum_{2 \leq r \leq n, \eta\in S(e,e')} X(D^r u(\eta)).
	\end{displaymath}
	where $X$ is a sum of squares, and the coefficients of the sum is nonzero only when the support of higher order strain gradient $D^r u(\eta)$ is contained in $S(e, e')$.
\end{lemma}

\begin{proof}
  Notice that Lemma \ref{th:elemprod0_2D_lemma} gives special case for
  $e'\in T_1(e)$ and $e'\in T_2(e)$. The general situation can be
  proved by induction.
\end{proof}


Now we have the following generalization of Proposition \ref{th:2d:qnl_hess_straingrad}

\begin{proposition}
  \label{th:app:sgrad_c0_eq_ddW}
  Let $\Hac_\mF$ be of the general form \eqref{eq:app:qnl_hessian_defn}, then
  \begin{equation}
    \label{eq:app:qnl_hessian_straingrad}
    \< \Hac_\mF u, u \> = \sum_{\rho\in\Rg}\sum_{\xi\in\L}\tilde{c}_\rho(\xi) |D_\rho u(\xi)|^2 + \sum_{\xi \in \L, 2\leq r \leq R} \tilde{X}_\xi(D^ru(\xi)),
  \end{equation}
	Furthermore, let $\Ha$ be the atomistic hessian, we obtain that 
  \begin{equation}
    \label{eq:app:atm_hessian_straingrad}
    \< \Ha_\mF u, u \> = \sum_{\rho\in\Rg}\sum_{\xi\in\L}c_\rho |D_\rho u(\xi)|^2 + \sum_{\xi \in \L, 2\leq r \leq R} X(D^ru(\xi)),
  \end{equation}
	Let $\Hc$ be the Cauchy-Born hessian, we obtain that
	\begin{equation}
    \label{eq:app:cb_hessian_straingrad}
    \< \Hc_\mF u, u \> = \sum_{\rho\in\Rg}\sum_{\xi\in\L}c_\rho |D_\rho u(\xi)|^2,
  \end{equation}
	
	Here we use consistent paths to compute all the hessians, namely, for each pair $\rho, \vsig$ in \eqref{eq:app:qnl_hessian_defn}, the shortest paths for different 
	$\xi$ are invariant with respect to translation. Furthermore, we have 	
  \begin{align}
    \label{eq:app:qnlhess_strgrad_cj}
    \tilde{c}_\rho(\xi) = c_\rho \quad &\text{except if both } 
    \xi, \xi+\rho \in \L^{\i}+\Rg,  \\
    \label{eq:app:qnlhess_strgrad_Xc}
     \tilde{X}_\xi = 0 \quad &\text{for } (\xi+\Rg)\cap(\La\cup \L^{\i}) = \emptyset, \quad \text{and} \\
    \label{eq:app:qnlhess_strgrad_Xa}
     \tilde{X}_\xi = X \quad &\text{for } (\xi+\Rg)\cap(\L^{\i}\cup \T^{\c}) = \emptyset.
  \end{align}
\end{proposition}
\begin{proof}
  Applying Lemma \ref{th:prod_2D_lemma} to the hessian representation
  \eqref{eq:app:qnl_hessian_defn} we obtain
  \eqref{eq:app:qnl_hessian_straingrad}. Other conclusions can be
  proved similiarly as in Proposition \ref{th:2d:cb_hess},
  \ref{th:2d:atm_hessrep} and \ref{th:2d:qnl_hess_straingrad}
\end{proof}

\subsection{Stability of inhomogeneous configurations}
\label{sec:stab_inhom}
Throughout this paper we only considered the stability of the QNL
method at homogeneous crystalline states. Here, we present a simple
argument, that can be used to extend stability results to
inhomogeneous states. We consider a point defect and show that, if the
a/c interface is sufficiently far from the defect core, then stability
of the defect in the atomistic model and stability of the a/c method
in the reference state imply stability of the defect in the a/c
model. The following is a 2D, scalar, flat interface variant of the 1D
result \cite[Theorem 7.8]{acta}. It is weaker in that it is not
quantitative, but has the advantage that it is readily extended to a
variety of situations, including 3D and vectorial problems.

We employ the notation and definitions of \S~\ref{sec:qnl2d_model}. In
addition, let $\ddel\Ea(y), \ddel\Eqnl(y)$ denote, respectively, the
hessians of the atomistic and QNL methods at a configuration $y : \L
\to \R$. That is, for $v \in \Usz$,
\begin{align*}
  \< \ddel\Ea(y) v, v \> &= \sum_{\xi \in \L} \sum_{i,j = 1}^6
  V_{i,j}(Dy(\xi)) D_i v(\xi) D_j v(\xi), \qquad \text{and} \\
  \< \ddel\Eqnl(y) v, v \> &= \!\!\!\!\! \sum_{\xi \in \L^\a \cup \L^{(0)}}
  \sum_{i,j = 1}^6 \tilde{V}_{\xi,ij}(Dy(\xi)) D_iv(\xi) D_j v(\xi) + \sum_{T \in \T} w_T
  (\D v_T)^\top W''(\D y_T) \D v_T;
\end{align*}
cf. \eqref{eq:2d:HaF_defn} and \eqref{eq:2d:qnl_hessian_defn}.

We model a point defect by adding a perturbation $P(y) = P( \{ y(\xi);
|\xi| \leq R^{\rm def} \} )$, for some $R^{\rm def} > 0$. We assume
that $P$ is twice Gateux differentiable. See \cite[\S
2.3.2]{EhrOrtSha:2013} for a justification that this model includes
essentially all point defects.

Finally, we define the shifted deformation $y_N(\xi) := y(\xi + N
a_2)$, which effectively moves the interface away from the defect. The
shifted defect potential is defined by $P_N(y) := P(y_{-N})$, so that
$P_N(y_N) = P(y)$.

\begin{proposition}
  Let $\mF \in \R^{d}, u : \L \to \R$ such that $\D u \in L^2$. Assume
  that the configuration $y := \mF x + u$ is stable in the atomistic
  model: there exists $\gamma^{\rm def} > 0$ such that
  \begin{equation}
    \label{eq:app_stab_a}
    \< [\ddel\Ea(y) + \ddel P(y)] v, v \> \geq \gamma^{\rm def} \| \D
    v \|_{L^2}^2 \qquad \forall v \in \Usz;
  \end{equation}
  and that the QNL energy is stable ``at infinity'': $\gamma^{\rm qnl}
  := \gamma(\ddel\Eqnl(\mF x)) > 0$.

  Finally, suppose that $V_{i,j}$ and $\tilde{V}_{i,j}$ are locally
  Liptschitz continuous at $\mF\Rg$: there exist $L, \epsilon > 0$
  such that
  \begin{equation}
    \label{eq:inhom_Lip}
    |\hat{V}_{i,j}(\mF\Rg+\bfg) - \hat{V}_{i,j}(\mF\Rg)| \leq L |\bfg| \qquad
    \text{for } |\bfg| \leq \epsilon, \quad \hat{V} \in \{V, \tilde{V}\}.
  \end{equation}
  
  Then,
  \begin{equation}
    \label{eq:1}
    \liminf_{N \to \infty} \inf_{\substack{v \in \Usz \\ \|\D v
        \|_{L^2} = 1}} \b\< [\ddel \Eqnl(y_N) + P_N(y_N)] v, v \b\> = \min
    \{ \gamma^{\rm def}, \gamma^{\rm qnl} \}.
  \end{equation}
\end{proposition}
\begin{proof}
  The proof employs a concentration compactness argument, and is
  similar to the proof of \cite[Theorem 4.8]{EhrOrtSha:2013}.

  Let $H_N := \ddel\Eqnl(y_N) + \ddel P_N(y_N)$ and $\gamma_N :=
  \inf_{\| \D v \|_{L^2} = 1} \< H_N v, v \>$. Let $v^{(N)} \in \Usz$
  such that $\|\D v^{(N)}\|_{L^2} = 1$ and $\< H_N v^{(N)}, v^{(N)} \>
  \leq \gamma_N + 1/N$.

  {\it 1. Decomposition: } We shift $v^{(N)}$ to recentre the
  defect. Since $\| \D v^{(N)}_{-N}\|_{L^2} = \| \D v^{(N)} \|_{L^2} =
  1$, there exists $v : \L \to \R, \D v \in L^2$ such that $\D
  v^{(N_j)}_{-N_j} \rightharpoonup \D v$, weakly in $L^2$, for some
  subsequence $N_j \uparrow \infty$. For the sake of simplicity of
  notation we drop the subscript of $N_j$.

  Lemma 4.9 in \cite{EhrOrtSha:2013} asserts that we may decompose
  $v^{(N)} = w^{(N)} + z^{(N)}$ such that the following properties
  hold: 
  \begin{align}
    \label{eq:inhom_split_1}
    & \D w^{(N)}_{-N} \to \D v \quad \text{strongly in } L^2 \quad
    \text{and} \quad \D z^{(N)}_{-N} \rightharpoonup 0 \quad \text{weakly
      in } L^2; \\
    \label{eq:inhom_split_2}
    & D w^{(N)}_{-N}(\xi) = \cases{D v^{(N)}_{-N}(\xi), &  |\xi| \leq
      R_N, \\
    0, & |\xi| \geq 2 R_N, } 
  \end{align}
  where $R_N \uparrow \infty$, and the speed of convergence of $R_N
  \uparrow \infty$ may be chosen arbitrarily slowly. Note also that
  \eqref{eq:inhom_split_1}, \eqref{eq:inhom_split_2} imply that $\D
  w^{(N)}$ and $\D z^{(N)}$ are bounded in $L^2$.

  Thus, we have
  \begin{align*}
    \gamma_N + 1/N &\geq \< H_N v^{(N)}, v^{(N)} \> \\
    &= \< H_N w^{(N)}, w^{(N)} \> + 2 \< H_N w^{(N)}, z^{(N)} \> + \<
    H_N z^{(N)}, z^{(N)} \> \\
    &=: a_N + b_N + c_N.
  \end{align*}

  {\it 2. Estimate of $a_N$: } Choosing $R_N \ll N$, property
  \eqref{eq:inhom_split_2} implies that
  \begin{displaymath}
    a_N = \< H_N w^{(N)}, w^{(N)} \> = \< [\ddel\Ea(y_N)+\ddel P_N(y_N)] w^{(N)},
    w^{(N)} \> \geq \gamma^{\rm def} \| \D w^{(N)} \|_{L^2}^2,
  \end{displaymath}
  where we have also applied \eqref{eq:app_stab_a} in the last inequality.

  {\it 3. Estimate of $c_N$: } Using the fact that
  $Dz^{(N)}_{-N}(\xi) = 0$
  for $|\xi| \leq R_N$, and the Lipschitz condition
  \eqref{eq:inhom_Lip}, a straightforward calculation yields
  \begin{align*}
    c_N &= \< \ddel \Eqnl(\mF x) z^{(N)}, z^{(N)} \> + \< [\ddel
    \Eqnl(y_N) - \ddel \Eqnl(\mF x)] z^{(N)}, z^{(N)} \> \\
    &\geq \gamma^\qnl \| \D z^{(N)} \|_{L^2}^2
    - C L \| D y - \mF\Rg \|_{\ell^\infty(\L
      \setminus B_{R_N})} \| D z^{(N)} \|_{\ell^2}^2\\
    &\geq \gamma^\qnl \| \D z^{(N)} \|_{L^2}^2 - s_N,
  \end{align*}
  where $s_N \to 0$ as $N \to \infty$ due to the fact that $D y -
  \mF\Rg \in \ell^2(\L)$.

  {\it 4. Estimate of $b_N$: } Recall that we chose $R_N \ll N$;
    hence, the support overlap of $Dw^{(N)}$ and $Dz^{(N)}$ is fully
  contained in the atomistic region.
  Since the coefficients $V_{ij}(Dy(\xi))$ are bounded, it follows
  that
  \begin{displaymath}
    b_N = \sum_{i,j = 1}^6  \sum_{\xi \in \L} V_{i,j}(Dy_{-N}(\xi)) D_i
    w^{(N)}_{-N}(\xi) D_j z^{(N)}_{-N}(\xi) \to 0 \quad \text{as } N
    \to \infty,
  \end{displaymath}
  due to the strong convergence of $D_i w^{(N)}_{-N}$ and hence of $
  V_{i,j}(Dy) D_i w^{(N)}_{-N}$ in $\ell^2$ and the weak convergence
  of $z^{(N)}_{-N}$ to zero.

  {\it 5. Conclusion of proof: } Combining steps 1--4 we obtain that
  \begin{displaymath}
    \liminf_{N \to \infty} \gamma_N \geq \min(\gamma^{\rm def}, \gamma^\qnl)
    \liminf_{N \to \infty} \b( \| \D w^{(N)} \|_{L^2}^2 + \| \D
    z^{(N)} \|_{L^2}^2 \b). 
  \end{displaymath}
  Arguing as in step 4, we obtain $\int \D w^{(N)} \D z^{(N)} \dx \to
  0$ as $N \to \infty$, from which we conclude that
  \begin{align*}
    \liminf_{N \to \infty} \gamma_N &\geq  \min(\gamma^{\rm def}, \gamma^\qnl)
    \liminf_{N \to \infty} \B( \| \D w^{(N)} \|_{L^2}^2 + 2
    {\textstyle \int \D w^{(N)} \D z^{(N)} \dx } +  \| \D
    z^{(N)} \|_{L^2}^2 \B) \\
    & = \min(\gamma^{\rm def}, \gamma^\qnl) \liminf_{N \to \infty} 
    \| \D w^{(N)} + \D z^{(N)} \|_{L^2}^2 \\
    &=  \min(\gamma^{\rm def}, \gamma^\qnl). \qedhere
  \end{align*}
\end{proof}

\bibliographystyle{plain}
\bibliography{qc}
\end{document}

%% file: notation.tex
\renewcommand{\cases}[1]{\left\{ \begin{array}{rl} #1 \end{array} \right.}
\newcommand{\smfrac}[2]{{\textstyle \frac{#1}{#2}}}

\def\XXint#1#2#3{{\setbox0=\hbox{$#1{#2#3}{\int}$ }
\vcenter{\hbox{$#2#3$ }}\kern-.6\wd0}}

\usepackage{accents}
\newlength{\dhatheight}
\newlength{\dtildeheight}
\newcommand{\hhat}[1]{%
    \settoheight{\dhatheight}{\ensuremath{\hat{#1}}}%
    \addtolength{\dhatheight}{-0.35ex}%
    \hat{\vphantom{\rule{1pt}{\dhatheight}}%
    \smash{\hat{#1}}}}
\newcommand{\ttilde}[1]{%
    \settoheight{\dtildeheight}{\ensuremath{\tilde{#1}}}%
    \addtolength{\dtildeheight}{-0.15ex}%
    \tilde{\vphantom{\rule{1pt}{\dtildeheight}}%
    \smash{\tilde{#1}}}}

\def\b{\big}
\def\B{\Big}
\def\bg{\bigg}
\def\Bg{\Bigg}

\def\sep{\,|\,}
\def\bsep{\,\b|\,}

\def\R{\mathbb{R}}
\def\N{\mathbb{N}}
\def\Z{\mathbb{Z}}

\def\dx{\,{\rm d}x}

\def\<{\langle}
\def\>{\rangle}

\def\mA{{\sf A}}
\def\mB{{\sf B}}
\def\mC{{\sf C}}
\def\mF{{\sf F}}
\def\mG{{\sf G}}
\def\mH{{\sf H}}
\def\mI{{\sf I}}

\def\mQ{{\sf Q}}

\def\mO{{\sf 0}}

\def\bfg{{\bm g}}

\newcommand{\transpose}{{\!\top}}

\def\D{\nabla}
\def\del{\delta}
\def\ddel{\delta^2}

\def\eps{\varepsilon}

\def\a{{\rm a}}
\def\c{{\rm c}}
\def\ac{{\rm ac}}
\def\i{{\rm i}}
\def\refl{{\rm rfl}}
\def\qnl{{\rm qnl}}
\def\stab{{\rm stab}}
\def\qce{{\rm qce}}
\def\grac{{\rm g23}}
\def\lref{{\rm lrf}}

\def\L{\Lambda}

\def\Us{\mathscr{W}}
\def\Usz{\Us_0}

\def\E{\mathcal{E}}
\def\Ea{\E^\a}
\def\Ha{H^\a}
\def\Ec{\E^\c}
\def\Hc{H^\c}
\def\Eqnl{\E^\qnl}
\def\Hqnl{H^\qnl}
\def\Erefl{\E^\refl}
\def\Hrefl{H^\refl}
\def\Eac{\E^\ac}
\def\Hac{H^\ac}
\def\Estab{\E^\stab}
\def\Hstab{H^\stab}

\def\dW{W'}
\def\ddW{W''}

\def\vsig{\varsigma}
\def\Rg{\mathcal{R}}
\def\tilD{\tilde{D}}
\def\rcut{r_{\rm cut}}

\def\ca{\gamma^\a}
\def\cac{\gamma^\ac}

\def\cqnl{\gamma^\qnl}

\def\La{\L^\a}
\def\Lc{\L^\c}

\def\T{\mathcal{T}}
\def\Ed{\mathcal{E}}
\def\e{\mathrm{\mathbf{e}}}
